\documentclass[12pt,reqno]{amsart}  

\usepackage{amsmath,amsthm,amssymb,amscd,mathdots,mathrsfs,mathtools,shuffle,stmaryrd}
\usepackage[pdfborder={0 0 0}]{hyperref}
\usepackage[shortlabels]{enumitem}
\usepackage{tikz}
\usepackage{tikz-cd}
\usepackage{graphics}
\usepackage{tablefootnote}
\newtheorem{theorem}{Theorem}[section]
\newtheorem{proposition}[theorem]{Proposition}

\newtheorem{lemma}[theorem]{Lemma}
\newtheorem{corollary}[theorem]{Corollary}
\newtheorem{conjecture}[theorem]{Conjecture}

\theoremstyle{definition}
\newtheorem{definition}[theorem]{Definition}
\newtheorem{ex}[theorem]{Example}
\newtheorem{remark}[theorem]{Remark}

\newtheorem{example}[theorem]{Example}
\numberwithin{equation}{section}
\mathtoolsset{showonlyrefs}

\renewcommand{\vec}{\underline}

\DeclareMathOperator{\Aut}{Aut}

\newcommand{\N}{\mathbb{N}}
\newcommand{\q}{\mathbb{Q}}
\newcommand{\Q}{\mathbb{Q}}
\newcommand{\Z}{\mathbb{Z}}
\newcommand{\z}{\mathbb{Z}}
\newcommand{\R}{\mathbb{R}}

\newcommand{\mzv}{\mathcal{Z}} % space of MZV

\newcommand{\qsh}{\ast}
\newcommand{\QA}{\Q\langle A \rangle}
\DeclareRobustCommand{\ai}{\genfrac{[}{]}{0pt}{}}
\newcommand{\dep}{\operatorname{dep}}
\newcommand{\PP}{\mathscr{P}}
\newcommand{\MS}{\mathbb{P}}
\newcommand{\MM}{\mathbb{M}}

\newcommand{\QM}{ \widetilde{M}}

\DeclareRobustCommand{\ar}{\genfrac{}{}{0pt}{}}
\newcommand{\ms}[3][]{P_{#1}\!\left(\ar{#2}{#3} \right)}
\newcommand{\mssmall}[3][]{P_{#1}\bigl(\ar{#2}{#3} \bigr)}

\newcommand{\hm}[2][]{H_{#1}\!\left( #2 \right)}
\newcommand{\jm}[2][]{J_{#1}\!\left( #2 \right)}
\newcommand{\msl}[4][]{P_{#1}\!\left( \ar{#2}{#3} ; #4  \right)}
\newcommand{\msexpl}[4][]{\exp P_{#1}\!\left( \ar{#2}{#3} ; #4  \right)}
\newcommand{\msexp}[3][]{\exp P_{#1}\bigg(\ar{#2}{#3} \bigg)}
\newcommand{\FF}{\mathcal{F}}  %polynomial sequence

\newcommand{\wh}{\mathfrak{H}} 
\newcommand{\wj}{\mathfrak{J}} 
\newcommand{\ws}{\mathfrak{P}}
\newcommand{\df}{\diamond_{\FF}}
\newcommand{\qsf}{\qsh_{\FF}}
\newcommand{\wsn}{\mathfrak{N}} % Ideal generated by elements in S / S^1

\newcommand{\siso}{\Xi} % isomorphism from words to pol. fct 

\newcommand{\gena}{\mathfrak{A}} % generating series for word

\newcommand*\circled[1]{\tikz[baseline=(char.base)]{
    \node[shape=circle,draw,inner sep=0.5pt] (char) {$#1$};}}
\newcommand{\ost}{\circledast}
\newcommand{\osh}{\raisebox{1.3pt}{\;\circled{\scalebox{0.53}{$\shuffle$}}}\;}
\newcommand{\oshh}{\raisebox{1.3pt}{\circled{\scalebox{0.53}{$\shuffle$}}}}

\DeclareRobustCommand{\ai}{\genfrac{[}{]}{0pt}{}}
\DeclareRobustCommand{\bi}{\genfrac{(}{)}{0pt}{}}
\DeclareRobustCommand{\zebi}{\zeta\genfrac{(}{)}{0pt}{}}
\DeclareMathOperator{\degree}{deg}
\DeclareMathOperator{\wt}{wt}

\DeclareMathOperator{\Zdegree}{Z^{\degree}}

\renewcommand{\=}{\: =\: }
\newcommand{\defis}{\: :=\: }
\newcommand{\+}{\,+\,}
\newcommand{\meno}{\,-\,}

\title{Partitions, Multiple Zeta Values and the $q$-bracket}

\author{Henrik Bachmann}
\address{Graduate School of Mathematics,  Nagoya University, Nagoya, Japan.}
\email{henrik.bachmann@math.nagoya-u.ac.jp}

\author{Jan-Willem van Ittersum}
\address{Mathematisch Instituut, Universiteit Utrecht, Postbus 80.010, 3508 TA Utrecht, The Netherlands}
\curraddr{Max-Planck-Institut f\"ur Mathematik, Vivatsgasse 7, 53111 Bonn, Germany.
}
\email{\href{mailto:ittersum@mpim-bonn.mpg.de}{ittersum@mpim-bonn.mpg.de}}


\begin{document}
\date{\today}

\maketitle

\begin{abstract} We provide a framework for relating certain $q$-series defined by sums over partitions to multiple zeta values. In particular, we introduce a space of polynomial functions on partitions for which the associated $q$-series are $q$-analogues of multiple zeta values. By explicitly describing the (regularized) multiple zeta values one obtains as $q\to 1$, we extend previous results known in this area. 
Using this together with the fact that other families of functions on partitions, such as shifted symmetric functions, are elements in our space will then give relations among ($q$-analogues of) multiple zeta values. Conversely, we will show that relations among multiple zeta values can be `lifted' to the world of functions on partitions, which provides new examples of functions where the associated $q$-series are quasimodular.
%a `lift' of relations among multiple zeta values to the world of functions on partitions. 
\end{abstract}
%\tableofcontents

% \begin{abstract}
% We provide a framework to relate results on certain quasimodular $q$-series defined by sums over partitions to results on multiple zeta values. In particular, we introduce a certain space of polynomial functions on partitions for which the associated $q$-series %under the $q$-bracket
% are $q$-analogues of multiple zeta values. We extend previous works on these $q$-analogues, by explicitly describing the (regularized) multiple zeta values one obtains as $q\to 1$. %as certain (regularized) bi-multiple zeta values. 
% Using this together with the fact that other families of functions on partitions, such as shifted symmetric functions, are elements in our space will then give a `lift' of relations among multiple zeta values to the world of functions partitions. \jw{Second approach, but we are certainly not there yet.}
% \end{abstract}

\section{Introduction}
The purpose of this note is to introduce a framework that can be seen as a bridge between the theory of functions on partitions and ($q$-analogues of) multiple zeta values. Multiple zeta values (see \eqref{eq:defmzv}) are real numbers appearing in various areas of mathematics and theoretical physics. These real numbers satisfy numerous relations, such as the so-called double shuffle relations \cite{IKZ}. For these numbers, there exist various different $q$-analogue models, which are $q$-series degenerating to multiple zeta values as $q\rightarrow 1$. For most of these $q$-analogues, there exist counterparts for the double shuffle relations, and their algebraic setups are well-understood  \cite{BK2,Bra,Bri2,Sin,Zh}. In this note, we will generalize this setup even further.
%by considering a certain class of functions on partitions. 
We will show that there is a natural analogue of the double shuffle relations on all functions on partitions and then introduce a class of functions that we call \emph{partition analogues of multiple zeta values}. These functions can be seen as the counterpart of $q$-analogues of multiple zeta values after applying the so-called $q$-bracket, introduced by Bloch and Okounkov in \cite{BO}. The space of partition analogues of multiple zeta values contains various classical types of functions on partitions, such as the shifted symmetric functions, which then, by using the results in this work, provide new tools to obtain relations among ($q$-analogues of) multiple zeta values.

Denote by~$\PP$ the set of all partitions of integers. To a function $f:\PP\to \q$, we associate (i) a degree, (ii) a limit~$\Zdegree(f)$ and (iii) a power series $\langle f \rangle_q \in \q\llbracket q \rrbracket$, in such a way that asymptotically
\[ (1-q)^{\mathrm{deg}(f)}\langle f \rangle_q \= \Zdegree(f) \+ O(1-q) \]
for real~$q$.

To start with the latter, the \emph{$q$-bracket} of $f$ is defined as
\begin{align}\label{eq:qbrac}
\langle f \rangle_q \defis \frac{\sum_{\lambda \in \PP} f(\lambda)\,q^{|\lambda|}}{\sum_{\lambda\in\PP} q^{|\lambda|}} \:\in\: \q\llbracket q \rrbracket,
\end{align}
where $|\lambda|$ denotes the integer that $\lambda$ is a partition of. In case $f(\lambda)$ has at most polynomial growth in $|\lambda|$, its $q$-bracket is holomorphic for $|q|<1$. Moreover, we can associate a \emph{degree} to $f$ by
\begin{equation}\label{eq:degree} \degree(f) \= \inf_{a\in \R}\bigl\{\lim_{q\to 1}(1-q)^a \langle f \rangle_q \text{ converges}\bigr\}.\end{equation}
We, then, define $\Zdegree(f)\in \mathbb{R}\cup\{\pm\infty\} $ to be the value of the corresponding limit\footnote{Here and in the rest of the work we understand $q \rightarrow 1$ as the limit where $q$ is real and $0 < q < 1$.} $\lim_{q\to 1}(1-q)^{\degree(f)} \langle f \rangle_q$ whenever it exists (as it does for all functions in our work). %More precisely,
%\[\Zdegree(f) = \liminf_{q\to 1} (1-q)^{\degree(f)}\langle f \rangle_q,\]
%
%Example: $\zeta(2)$
For example, for $f(\lambda)=|\lambda|$ we have
\[ \langle f \rangle_q \= \sum_{n=1}^\infty \sigma(n) \,q^{n}, \quad \degree(f)=2, \quad \Zdegree(f) \= \zeta(2),\]
where $\sigma(n)=\sum_{d\mid n} d$ denotes the divisor sum of~$n$ and $\zeta(k)=\sum_{m\geq 1} \frac{1}{m^k}$ the Riemann zeta value. 

Such limits~$\Zdegree(f)$ occur as the volumes of certain moduli spaces, e.g., in the case of the stratum of one-cylinder square-tiled surfaces \cite{DGZ} or in the case of flat surfaces \cite{CMZ}. In both cases, there are associated functions $f$ in $\Lambda^*$, the space of shifted symmetric functions. That is, $\Lambda^*=\q[Q_2,Q_3,\ldots]$, where $Q_k:\PP\to \q$ is given 
\begin{align}\label{eq:defQk}
   Q_k(\lambda) \defis  \beta_k \+ \frac{1}{(k-1)!}\sum_{i=1}^\infty \bigl( (\lambda_i-i+\tfrac{1}{2})^{k-1}-(-i+\tfrac{1}{2})^{k-1}\bigr), 
\end{align}
where $\lambda=(\lambda_1,\lambda_2,\ldots)$ and 
$\beta_k=\bigl(\frac{1}{2^{k-1}}-1)\frac{B_k}{k!}$ with $B_k$ the $k$-th Bernoulli number.
%and the $\beta_k$ are constants defined by $\sum_k \beta_k z^{k-1} \defis \frac{1}{2\sinh(z/2)}.$ 
By the Bloch--Okounkov theorem \cite{BO} (going back to work of Dijkgraaf and Kaneko--Zagier \cite{Dij,KZ95}) the $q$-brackets~$\langle f \rangle_q$ for $f\in \Lambda^*$ are quasimodular forms. The space of  quasimodular forms~$\QM=\q[G_k \mid k=2,4,6,\ldots]$ is generated by the Eisenstein series~$G_k$ for all even~$k\geq 2$
\begin{align}\label{eq:eisenstein}
	G_k(q) \defis -\frac{B_k}{2k!}\+ \frac{1}{(k-1)!}\sum_{m,r\geq 1} m^{k-1} q^{mr}\,,
\end{align}
which are holomorphic functions for $|q|<1$ (or equivalently, for $\tau$ in the complex upper half plane, with $q=e^{2\pi \mathrm{i}\tau}$). 
%that is $\QM=\q[G_k \mid k=2,4,6,\ldots] = 
In fact, $\QM=\q[G_2,G_4,G_6].$ As $\Zdegree(G_k) = \zeta(k)$ and every modular form can be written as the sum of an Eisenstein series and a cusp form $F$ with $\Zdegree(F)=0$, it follows that for $f\in \Lambda^*$ the corresponding limits~$\Zdegree(f)$ are single zeta values. %i.e., $\Zdegree(f)\in \q[\zeta(2)]$. 
Besides the shifted symmetric functions there are various other functions on partitions which give rise to quasimodular forms (see \ \cite{Z}). All of these have limits which lie in $\q[\zeta(2)]$.
We will introduce a space $\MS \subset \Q^\PP$ of partition analogues of multiple zeta values, whose elements always have
\emph{multiple zeta values} as their limit.
 Multiple zeta values are defined for $r\geq 1$ and $k_1\geq 2, k_2,\dots,k_r \geq 1$ by
\begin{align}\label{eq:defmzv}
	\zeta(k_1,\ldots,k_r)\defis \sum_{m_1>\cdots>m_r>0 }\frac{1}{m_1^{k_1}\cdots m_r^{k_r}}\:\in \:\R.
\end{align}
A natural subspace of~$\MS$ is the space $\MM = \{f \in \MS \mid \langle f \rangle_q \in \QM\} \subset \MS$, containing~$\Lambda^*$, whose elements have a quasimodular $q$-bracket and a limit in $\q[\zeta(2)]$. 
%Since there are various results on when certain linear combinations of multiple zeta values are in $\q[\zeta(2)]$, we will see that some of these families can be lifted to linear combinations of elements in $\MS$ which are in $\MM$.
% We aim to extend the Bloch--Okounkov theorem in such a way that the corresponding limits are \emph{multiple zeta values}.
In particular, the space $\MS$ completes the following diagram:

\vspace{-0.4cm}
%\begin{figure}[h!]
%\centering
\begin{equation}    \label{fig:overview}
\begin{tikzcd}[remember picture]
 \   \q^\PP 
 \arrow["\langle \ \rangle_q"  , r] &
 \mathbb{Q}\llbracket q \rrbracket
 \arrow["\text{``}\lim\limits_{q \to 1}\text{''}",r] & 
 \mathbb{R} \\
\MS \,	
\arrow[r] &
\mzv_q 
\arrow[r] &
\mzv\\
\mathllap{\Lambda^*\subset\ }\MM \arrow[r] &
\QM 
\arrow[r] &
\mathbb{Q}[\zeta(2)].\\
\end{tikzcd}
\begin{tikzpicture}[overlay,remember picture]
\path (\tikzcdmatrixname-1-1) to node[midway,sloped]{$\supset$} (\tikzcdmatrixname-2-1);
\path (\tikzcdmatrixname-1-2) to node[midway,sloped]{$\supset$}
(\tikzcdmatrixname-2-2);
\path (\tikzcdmatrixname-1-3) to node[midway,sloped]{$\supset$}
(\tikzcdmatrixname-2-3);
% \path (\tikzcdmatrixname-1-4) to node[midway,sloped]{$\supset$}
% (\tikzcdmatrixname-2-4);
\path (\tikzcdmatrixname-2-1) to node[midway,sloped]{$\supset$} (\tikzcdmatrixname-3-1);
\path (\tikzcdmatrixname-2-2) to node[midway,sloped]{$\supset$} (\tikzcdmatrixname-3-2);
\path (\tikzcdmatrixname-2-3) to node[midway,sloped]{$\supset$}
(\tikzcdmatrixname-3-3);
% \path (\tikzcdmatrixname-2-4) to node[midway,sloped]{$\supset$}
% (\tikzcdmatrixname-3-4);
\end{tikzpicture}
\end{equation}
\vspace{-15pt}
%\end{figure}
\vspace{-10pt}

\noindent Here, $\mzv$ denotes the $\Q$-vector space generated by multiple zeta values, $\mzv_q$ denotes the space of $q$-analogues of multiple zeta values, studied by many authors (see e.g., \cite{BK2,Bri1,Zh} for an overview of different types of $q$-analogues; in this work, we define $\mzv_q$ in~\eqref{eq:defqmzv}, following \cite{BK2}), and the map $\lim_{q \to 1}$ is between quotation marks to indicate that it concerns infinitely many maps ${f\mapsto}{\lim_{q \to 1} (1-q)^a f}$, which are well-defined/ill-defined/regularized depending on $f$ and the value of $a$. 

\subsection*{Polynomial functions on partitions}
Let $p\in \q[x,y]$. In \cite{vI}, the second author studied functions ${T_p:\PP\to \q}$ (not to be confused with Hecke operators) of the form
\[ T_p(\lambda) \= \sum_{m=1}^\infty \sum_{r=1}^{r_m(\lambda)} p(m,r),\]
where $r_m(\lambda)$ denotes the number of times $m$ occurs as a part in the partition~$\lambda$. Similar to the Bloch--Okounkov theorem, for every $f\in \q[T_p \mid p\in \q[x,y], \mathrm{deg}(p) \text{ is odd}]$, the $q$-bracket $\langle f \rangle_q$ is a quasimodular form.
%Note that polynomials in these functions are given by (linear combinations of) $T_p:\PP\to \q$ of the form
%\[ T_p(\lambda) = \sum_{m_1,\ldots,m_n} \sum_{r_i\leq r_{m_i}(\lambda)} p(m_1,\ldots,m_n,r_1,\cdots,r_n),\]
% now with $p\in \q[x_1,\ldots,x_n,y_1,\ldots,y_n]$.
Motivated by this construction, we define the space of \emph{partition analogues of multiple zeta values} $\MS$ as the following space which can be thought of a space of polynomial functions on partitions.
\begin{definition}\label{def:MS}
Let $\MS$ be the image of 
\[\Psi: \bigoplus_{n\geq 0} \q[x_1,\ldots,x_n,y_1,\ldots,y_n] \to \q^{\PP}, \]
where $\Psi$ maps the polynomial~$p(x_1,\ldots,x_n,y_1,\ldots,y_n)$ to 
\[ \lambda \mapsto \sum_{m_1>\ldots>m_n>0}\sum_{r_1=1}^{r_{m_1}(\lambda)}\cdots \sum_{r_n=1}^{r_{m_n}(\lambda)} p(m_1,\ldots,m_n,r_1,\ldots,r_n).\]
Moreover, define $\MM = \{f \in \MS \mid \langle f \rangle_q \in \QM\}.$
\end{definition}
We show that this space $\MS$ completes the diagram \eqref{fig:overview}. In particular, we are able to compute the \emph{degree} and \emph{limit} of elements of $\MS$:
\begin{theorem}\label{thm:main}
Given $r\geq 1$ and $d_i,l_i\in \z_{\geq 0}$ for $i=1,\ldots,r$, let $f=\Psi(\prod_{i} x_i^{d_i}y_i^{l_i})$. Then, 
\begin{equation}\label{eq:deg} \mathrm{deg}(f) \= \max_{j\in \{0,\ldots,r\}}\biggl\{\sum_{i\leq j}(d_i+1)\+\sum_{i> j}(l_i+1)\biggr\}.\end{equation}
Moreover, if the maximum is attained for a unique value of~$j$, then
%, then~$f$ is admissible and
$\Zdegree(f) \in \mzv_{\leq \degree(f)}\,,$ where $\mzv_{\leq k}$ denotes the $\Q$-vector space of multiple zeta values $\zeta(k_1,\ldots,k_r)$ with $k_1+\ldots+k_r\leq k$.
\end{theorem}
%In case the maximum is not unique, we are in a situation for which $\Zdegree(f)$ diverges, as (formally) it involves non-admissible multiple zeta values.

% The case $k_1=\ldots=k_t=d_{t+1}=\ldots=d_r=0$ is particularly interesting. Namely, in this case we have the following expression
% \[ \Zdegree(f) = \xi(d_1,\dots,d_t) \cdot \zeta(k_{t+1}, \dots , k_r),\]
We will see that the regularized limit of polynomial functions on partitions always gives regularized multiple zeta values. These limits will be given by \emph{bi-multiple zeta values} $\zebi{k_1,\dots,k_r}{d_1,\dots,d_r} \in \R[T]$, which are defined for $k_1,\dots,k_r\geq 1$, $d_1,\dots,d_r\geq 0$ in Definition~\ref{def:bimzv}, and which generalize (harmonic regularized) multiple zeta values in the sense that for $k_1\geq 2,k_2,\dots,k_r\geq 1$
\begin{align*}
    \zebi{k_1,\dots,k_r}{0,\dots,0} \= \zeta(k_1,\dots,k_r)\,.
\end{align*}
The other special case of $d_1,\dots,d_{r-1}\geq 0, d_r\geq 1$ and $k_1=\dots=k_r=1$ is given by 
\begin{align*}
    \zebi{1,\dots,1}{d_1,\dots,d_r} \= \xi(d_1,\dots,d_r)\,,
\end{align*}
where we define the \emph{conjugated zeta values}~$\xi$ as follows:
\begin{definition}\label{def:conjzeta}
For $d_1,\dots,d_{r-1}\geq 0, d_r\geq 1$, define the \emph{conjugated multiple zeta value} by
\[ \xi(d_1,\ldots,d_r) \defis \sum_{0<m_1<\ldots<m_r} \frac{1}{m_1\cdots m_r}\, \Omega\biggl[\prod_{i=1}^r \Bigl(\frac{1}{m_i}+\ldots+\frac{1}{m_r}\Bigr)^{d_i}\biggr], \]
where $\Omega:\Q[m_1^{-1},\ldots,m_r^{-1}]\to \Q[m_1^{-1},\ldots,m_r^{-1}]$ is the linear mapping
\begin{align}\label{def:omega}
\Omega\Bigl[\frac{1}{m_1^{l_1}\cdots m_r^{l_r}}\Bigr] \defis \frac{l_1!\cdots l_r!}{m_1^{l_1}\cdots m_r^{l_r}} \,.
\end{align}
\end{definition}
By definition it is clear that the conjugated zeta values can be written as linear combinations of multiple zeta values and as a result of Theorem \ref{thm:bimzv} their product can be expressed by the index-shuffle product formula, e.g., $\xi(d_1)\xi(d_2) =\xi(d_1,d_2) + \xi(d_2,d_1)$ for $d_1,d_2\geq 1$. In general, the bi-multiple zeta values will be given by sums of products of multiple zeta values and their conjugated analogues $\xi$, e.g., we have (for $k_1\geq 2$, $d_m\geq 1$)
\begin{align*}
    \zebi{1,\dots,1,k_1,\dots,k_r}{d_1,\dots,d_m,0,\dots,0} \= \xi(d_1,\ldots,d_m)\, \zeta(k_1,\dots,k_r)\,.
\end{align*}
The product of the bi-multiple zeta values can be expressed by a generalization of the usual harmonic product, e.g.\ we have 
\begin{align*}
    \zebi{k_1}{d_1}\zebi{k_2}{d_2}\=\zebi{k_1,k_2}{d_1,d_2}+\zebi{k_2,k_1}{d_2,d_1}+\zebi{k_1+k_2}{d_1+d_2}\,.
\end{align*}
Further, we will see in Theorem~\ref{thm:bimzv} that the bi-multiple zeta values are invariant under a certain involution $\iota_\mathrm{s}$, which will appear naturally when considering functions on partitions. In particular, these give a realization of the so-called formal double Eisenstein space introduced in \cite{BKM}.

% \subsection*{Application: relations between MZVs} \hen{maybe just do not make this an extra section here but just a paragraph?}

In Section~\ref{sec:algsetup}, we will introduce an algebraic setup for the elements in $\MS$, which can be seen as a generalization of the classical algebraic setup for multiple zeta values and their double shuffle relations (cf., \cite{H, IKZ}). Starting with a linear combination of multiple zeta values which evaluates to an element in $\Q[\zeta(2)]$, we then show, with diagram \eqref{fig:overview} in mind, how to `lift' these to obtain new families of functions on partitions with quasimodular $q$-bracket (see Proposition~\ref{prop:pkkkisquasimodular} and Example~\ref{ex:pi}). In Section~\ref{sec:mzv}, we will show that not only the analogue of the double shuffle relation give relations among ($q$-analogues of) multiple zeta values, but also how families of polynomial functions on partitions with quasimodular $q$-bracket can be used to obtain relations. For example, we will see that the shifted symmetric functions together with the Bloch--Okounkov theorem implies a special case of the Ohno--Zagier relations and that the arm-leg moments of Zagier (\cite{Z}) imply the sum formula for multiple zeta values.

% The map~$\Psi$ in the definition of $\MS$ is injective. Hence, in contrast to the space of ($q$-analogues of) multiple zeta values, there are no relations between a set of generators of $\MS$. Surprisingly, however, this makes the space~$\MS$ even more suited to study the relations in $\mzv_q$ and $\mzv$, as we explain now. 

% \jw{summarize algebraic structure of $\MS$ and how relations can be obtained from this construction}

%\subsection*{Contents}

\subsection*{Acknowledgements} This project was partially supported by JSPS KAKENHI Grants 19K14499 and 21K13771. Part of the work has been carried out during a visit of the second author at Nagoya, and he would like to thank Nagoya University for hospitality and support. He also thanks Utrecht University, where he carried out most of this work, and the Max Planck Institute for Mathematics, where he finished this work. 
We would like to thank the referee for helpful comments and corrections and Nobuo Sato and Yuta Suzuki for the idea of the proof of Proposition~\ref{prop:review} and Remark~\ref{rem:inequcounterex}.

\section{Functions on partitions}
%\subsection{Partitions}
%\jw{Set notation for strict partitions, conjugation, M\"obius function, depth, exponential notation}

\subsection{Multiplication, conjugation, brackets and derivations}\label{sec:operations}
Denote by~$\PP$ the set of all partitions of integers. We make use of the following equivalent definitions of partitions:
\begin{enumerate}[(i), leftmargin=*]
    \item Finite non-increasing sequences of positive integers $(\lambda_1,\lambda_2, \ldots, \lambda_\ell)$, where we write~$\ell(\lambda)$ for the \emph{length}~$\ell$ of the partition~$\lambda$ and $|\lambda|=\lambda_1+\ldots+\lambda_\ell$ for the \emph{size}.
    \item Infinite sequences $(\lambda_1, \lambda_2,\dots)$ with $\lambda_1 \geq \lambda_2 \geq \dots$ and $\lambda_j = 0$ for all but finitely many~$j$;
    \item (\emph{Stanley's multi-rectangular coordinates}) Two sequences $\bf{r}$ and $\bf m$ of non-negative integers, of the same length~$d$, and of which $\bf m$ is strictly decreasing. These sequences correspond to the partition
    \[{\bf r} \times {\bf m} = (\underbrace{m_1,\ldots,m_1}_{r_1},\ldots,\underbrace{m_d,\ldots,m_d}_{r_d})\]
    in the first definition. This representation is unique if the elements of~$\bf r$ and~$\bf m$ are positive. Often, given $\lambda\in\PP$, we write ${\bf r}(\lambda)=(r_1,\dots,r_d)$ and ${\bf m}(\lambda)=(m_1,\dots,m_d)$ for such sequences. The integer~$d$ is called the \emph{depth} of the partition.
    \item Multisets of integers, in which the integer~$m$ has
    \[r_m(\lambda) = \# \{ j \mid \lambda_j = m \}\,\]
    appearances.
\end{enumerate}
\begin{example}
The smallest partition is the empty partition (of the integer~$0$), written as $(), (0,0,\ldots), ()\times(), \emptyset$ respectively. The Stanley coordinates of the partition $\lambda=(6,4,4,3,2,2,2,1)$ can be read of by the decomposition $\lambda=(1,2,1,3,1)\times(6,4,3,2,1)$.
\end{example}

%An element $\lambda \in \PP$ can be written as $\lambda = (\lambda_1, \lambda_2,\dots)$ with $\lambda_1 \geq \lambda_2 \geq \dots$ and $\lambda_j = 0$ for all but finitely many $j$. Write $\ell(\lambda)$ for the number of (non-zero) parts of $\lambda$, i.e., $\ell(\lambda) = \#\{j|\lambda_j \neq 0\}$. For an $m\in \Z_{\geq 1}$ we denote by $r_m(\lambda)$ the number of parts of $\lambda$ being equal to $m$, i.e.
There are two  elementary operations on the space of functions ${f,g\in \q^\PP}$:
\begin{enumerate}[(i),leftmargin=*]
\item \emph{pointwise multiplication}, i.e.\ the multiplication $(f \odot g)(\lambda)=f(\lambda)\,g(\lambda)$ induced by the multiplication on $\q$,
\item \emph{conjugation}, i.e.\ $\omega(f)(\lambda) = f(\overline\lambda)$, where $\overline\lambda$ denotes the transpose of the partition~$\lambda$. 
\end{enumerate}
We adapt both operations to make them equivariant with respect to the $q$-bracket~\eqref{eq:qbrac}. In order to do so, we introduce the $\vec{u}$-bracket \cite[Definition 3.2.1]{vI}.
\begin{definition}\label{defn:ubrac}
The vector space isomorphism $\langle \ \rangle_{\vec{u}}:\q^\PP \to \q\llbracket u_1,u_2,\ldots\rrbracket$ is given by
\[\langle f \rangle_{\vec{u}} \defis \frac{ \sum_{\lambda \in \PP} f(\lambda)\, u_{\lambda}}{\sum_{\lambda \in \PP}  u_{\lambda}}\qquad (u_\lambda = u_{\lambda_1}\,u_{\lambda_2}\cdots,\ u_0=1).\]
For $f\in \q^\PP$ we call $\langle f\rangle_{\vec{u}}$ the \emph{$\vec{u}$-bracket} of $f$. For all $\lambda \in \PP$ we write $a_\lambda(f)$ to denote the coefficient of $u_\lambda$ in $\langle f \rangle_{\vec{u}}\mspace{1mu}$, i.e., $\langle f \rangle_{\vec{u}} = \sum_{\lambda \in \PP} a_\lambda(f) \, u_{\lambda}\mspace{1mu}$. 
%We denote by $\Phi$ the inverse of the $\vec{u}$-bracket, i.e.,$\Phi\langle f \rangle_{\vec{u}}=f$. 
\end{definition}
Note that the~$\vec{u}$-bracket reduces to the~$q$-bracket~\eqref{eq:qbrac} by specializing $u_i=q^i$ for all integers~$i$. 

The $\vec{u}$-bracket is \emph{not} an algebra homomorphism with respect to the pointwise product on $\q^\PP$. Therefore, we introduce the \emph{harmonic product} on $\q^\PP$, making the $\vec{u}$-bracket into an algebra homomorphism. Even so, we introduce a conjugation~$\iota$ making the $\vec{u}$-bracket equivariant.
%different from~$\omega$ for which $\langle \iota(f)\rangle_{\vec{u}}$ corresponds to taking transposing partitions in $\langle f \rangle_{\vec{u}}\mspace{1mu}$.
\begin{definition}\label{def:products}
Given $F,G\in \q\llbracket u_1,u_2,\ldots\rrbracket$, we define
\begin{enumerate}[(i),leftmargin=*]\itemsep3pt
\item the \emph{harmonic product} as the multiplication $F\ost G = FG$, where $FG$ denotes the standard product of $F$ and $G$ in $\q\llbracket u_1,u_2,\ldots\rrbracket$,
\item the \emph{conjugation} by $\iota(F) = \sum_{\lambda \in \PP} a_\lambda u_{\overline\lambda}$\ , where $F=\sum_{\lambda \in \PP} a_\lambda u_\lambda$ with $a_\lambda \in \q$ and where $\overline\lambda$ denotes the transpose of the partition~$\lambda$,
\item the \emph{shuffle product} as the multiplication $F\osh G = \iota(\iota(F)\ost \iota(G))$,
\item the \emph{derivative} of $F=\sum_{\lambda \in \PP} a_\lambda u_\lambda$ by $DF=\sum_{\lambda \in \PP} a_\lambda |\lambda| \,u_\lambda\mspace{1mu}$.
\end{enumerate}
We extend these definitions to $\q^\PP$ by the isomorphism given by the $\vec{u}$-bracket, i.e.\ for $f,g \in \q^\PP$ we define 
\begin{enumerate}[(i),leftmargin=*]\itemsep3pt
\item the \emph{harmonic product} by $\langle f \ost g\rangle_{\vec{u}} =\langle f\rangle_{\vec{u}} \langle g\rangle_{\vec{u}}\mspace{1mu}$,
\item the \emph{conjugation} by $\langle\iota(f)\rangle_{\vec{u}} = \iota\langle f \rangle_{\vec{u}}\mspace{1mu}$,
\item the \emph{shuffle product} as the multiplication $\langle f\osh g \rangle_{\vec{u}} = \langle f \rangle_{\vec{u}} \osh \langle g \rangle_{\vec{u}}\mspace{1mu}$,
\item the \emph{derivative} of $f$ by $\langle Df\rangle_{\vec{u}} = D\langle f \rangle_{\vec{u}}\mspace{1mu}$.
\end{enumerate}
\end{definition}
\begin{remark} %For $f,g \in \q^\PP$ we have $\langle f \ost g\rangle_{\vec{u}} =\langle f\rangle_{\vec{u}} \langle g\rangle_{\vec{u}}\,$.
In \cite{vI} the harmonic product was called the \emph{induced product}, as it is induced from the product on $\q\llbracket u_1,u_2,\ldots\rrbracket$. In the context of the present work, the name \emph{harmonic product} is more appropriate, as it will be indicated in Example~\ref{eq:dshanalog}.
\end{remark}

\begin{proposition}\label{prop:doubleshuffle} For all $f,g\in \q^\PP$
\[ \langle \iota(f) \rangle_q = \langle f \rangle_q \,,\quad  \langle f \ost g \rangle_q = \langle f\rangle_q \,\langle g\rangle_q = \langle f  \osh g \rangle_q \quad \text{and} \quad q\frac{\partial}{\partial q}\langle f \rangle_q = \langle Df\rangle_q\,.\]
\end{proposition}
\begin{proof}
The first equality follows directly by noting that a partition~$\lambda$ and its conjugate~$\overline{\lambda}$ are of the same size, so that substituting $q^i$ for $u_i$ has the same effect on $u_\lambda$ and $u_{\overline{\lambda}}\mspace{1mu}$. By definition of the shuffle product, the second identity follows from the first. The last equality follows directly by noting that $\langle f \rangle_{\vec{u}}|_{u_i=q^i}=\langle f \rangle_q\mspace{1mu}$.
\end{proof}
The harmonic product, conjugation, shuffle product and derivative can be given by explicit formulas on the level of functions in $\q^\PP$. For this, recall that a \emph{strict partition} is a partition for which all (non-zero) parts are distinct. Let the M\"obius function $\mu:\PP\to \{\pm 1\}$ be given by
\[ \mu(\lambda) = \begin{cases} (-1)^{\ell(\lambda)} & \lambda \text{ is strict} \\ 0 & \text{else.}\end{cases} \]
Denote the convolution product of $f,g\in \q^\PP$ by $\star$ (we reserve the symbol~$*$ for the harmonic product), i.e.,
\[ (f \star g)(\lambda) = \sum_{\alpha \cup \beta=\lambda} f(\alpha)\,g(\beta),\]
where in the summation we take the union of partitions considered as multisets. Write~$1:\PP\to \q$ for the inverse of the M\"obius  function under convolution, i.e.\ $1(\lambda)=1$. 
Then, by a direct computation, for all $f,g\in \q^\PP$, we have (see also \cite[Proposition~3.2.3]{vI})
\begin{align}
 \label{eq:harmonic}  f\ost g &\= f\star g\star \mu% \sum_{\alpha\cup\beta\cup\gamma=\lambda} f(\alpha)\,g(\beta)\,\mu(\gamma)
\\ \iota(f)&\= \omega(f\star \mu)\star 1%=\sum_{(\overline{\alpha\cup\delta})\cup\gamma=\lambda} f(\alpha)\,\mu(\delta), 
\\ f\osh g &\= \omega(\omega(f\star \mu)\star \omega(g\star \mu))\star 1,%= \sum_{\overline{\overline{(\alpha\cup\delta})\cup\overline{(\beta\cup\epsilon})}\cup\gamma=\lambda} f(\alpha)\,g(\beta)\,\mu(\delta)\,\mu(\epsilon)
 \end{align}
where we recall $\omega(f)(\lambda)=f(\overline{\lambda})$. 
Also, by \cite[Proposition~5.1.1]{vI} for all $f\in \q^\PP$ we have
\[ Df(\lambda)\= f(\lambda)|\lambda| - (f\ost |\cdot|)(\lambda).%= fS_2 - f\star S_2\star \mu
\]

Two subspaces of $\q^\PP$ are particularly well behaved with respect to these operations. 
\begin{definition}\label{def:HandJ}  We define  
\begin{align*}
\mathscr{H}&\defis\{f\in \q^\PP \mid f(\lambda)=f(\rho) \text{ if } {\bf r}(\lambda) = {\bf r}(\rho) \text{ for all } \lambda,\rho \in \PP \} \\
\mathscr{J}&\defis\{f\in \q^\PP \mid (f \star  \mu)(\lambda)=(f \star  \mu)(\rho) \text{ if } {\bf m}(\lambda) = {\bf m}(\rho) \text{ for all } \lambda,\rho \in \PP \},
\end{align*}
where ${\bf r}(\lambda)$ and ${\bf m}(\lambda)$ are the Stanley coordinates for~$\lambda$. 
\end{definition}
For example, the M\"obius function~$\mu$ is in $\mathscr{H}$, and for any $m\geq 1$ the function $\lambda \mapsto r_m(\lambda)$ is an element of $\mathscr{J}$. 
\begin{lemma}\label{lem:handjclosed}
The space $\mathscr{H}$ is closed under the harmonic product~$\ost$ and $\mathscr{J}$ is closed under the shuffle product~$\oshh$. In fact the spaces are conjugate, i.e.\ $\iota\mathscr{H}=\mathscr{J}$.
\end{lemma}
\begin{proof}
Suppose $\lambda = {\bf r \times m}$ and $\rho = {\bf r \times s}$ for some sequences $\bf r, m, s$. Given $\alpha\subset \lambda$ (where we consider $\alpha$ and $\lambda$ to be multisets), there is some sequence $\bf a$ such that $\alpha = {\bf a \times m}$. Then, ${\bf a \times s}\subset \rho$. This gives a bijection between subsets of~$\lambda$ and of~$\rho$. Moreover, if $f\in \mathscr{H}$, then $f({\bf a \times m})=f({\bf a \times s})$. Hence, by the expression in Eq.~\eqref{eq:harmonic}, we conclude that $\mathscr{H}$ is closed under the harmonic product. 

Next, we show that $\mathscr{H}$ and $\mathscr{J}$ are conjugate under~$\iota$. By the same argument as above, if $f\in \mathscr{H}$, then also $f\star\mu$. Note that $a_\lambda(f) = (f\star\mu)(\lambda)$, where $a_\lambda(f)$ is defined by Definition~\ref{defn:ubrac}.
Hence, if $f\in \mathscr{H}$ or $f\in \mathscr{J}$, then $a_\lambda(f)$ and $a_\rho(f)$ agree whenever ${\bf r}(\lambda) = {\bf r}(\rho)$ or ${\bf m}(\lambda) = {\bf m}(\rho)$ respectively. The statement now follows from the definition of $\iota$ using the fact that the conjugate of $(r_1,\ldots, r_d)\times (m_1,\ldots, m_d)$ is given by $(m_1, m_{2}-m_1,\ldots,m_d-m_{d-1})\times (r_1+\ldots+r_d,\ldots,r_1)$.
\end{proof}
\begin{remark}
A more explicit definition of the space $\mathscr{J}$ is as follows. Given $\lambda\in \PP$  of depth~$d$, write $\lambda={\bf r}\times {\bf m}$  and for $I\subset [d]:=\{1,\ldots,d\}$, let ${\bf m}_I$ be the strict partition with parts $m_i$ for $i\in I$. Then $f\in \mathscr{J}$ precisely if $f(\lambda)$ is determined by the value of~$f$ on all strict partitions contained in~$\lambda$ in the following way: 
\[ f(\lambda) \= \sum_{I\cup J=[d]}(-1)^{|J|} \prod_{i\in I} r_i \prod_{j \in J}(r_j-1) \, f({\bf m}_I).\]
\end{remark}

\subsection{Degree and limits} 
%To every $f\in \q^\PP$ we associate a degree. In the next section we introduce polynomial functions on partitions, in which case the degree corresponds to the degree of the corresponding polynomials. Moreover, in that case the corresponding limit converges to a multiple zeta value. 
Recall from~\eqref{eq:qbrac} and \eqref{eq:degree} the definition of the $q$-bracket, degree $\deg(f)$ and limit $\Zdegree(f)$ of a function~$f$ on partitions.
% To power series $F\in \q\llbracket q \rrbracket$ we associate a degree and a corresponding limit, in such a way that if $F=\langle f\rangle_q$ for some $f:\PP\to \q$, then $\deg(F)=\deg\langle f \rangle_q$ and $\Zdegree(F)=\Zdegree(f)$ (see Eq.~\eqref{eq:degree}). 
% \begin{definition} Define the \emph{degree} $\degree(F)$ of a power series $F\in\q\llbracket q \rrbracket$ to be \[\degree(F) = \inf_{a\in \R}\{\lim_{q\to 1}(1-q)^a F(q) \text{ converges}\}.\] Also, whenever defined let $\Zdegree\in \mathbb{R}\cup\{\pm\infty\}$ be the value of the corresponding limit, i.e.\
% \[\Zdegree(F) = \lim_{q\to 1} (1-q)^{\degree(F)}F(q).\]
%Define the degree and degree limit of a function $f:\PP\to \q$ to be the degree and degree limit of $\langle f \rangle_q\,.$ 
%\end{definition}
\begin{example}\label{ex:limit} Let us consider several examples of degree limits of interesting functions on partitions:
\begin{enumerate}[(i), leftmargin=*]
\item Let $d(\lambda)$ equal the number of different parts of $\lambda$, i.e., $d$ is the depth of $\lambda$. As
\[\sum_{\lambda \in \PP} d(\lambda) \,q^{|\lambda|} \= \frac{q}{1-q} \Bigl(\sum_{\lambda \in \PP} q^{|\lambda|}\Bigr),\]
we have $\langle d \rangle_q = \frac{q}{1-q}$. Hence, $\degree(d)=1$ and $\Zdegree(d)=1$. 
\item Next, consider $f(\lambda) = \ell(\lambda)$. Then $\langle f \rangle_q = \sum_{m\geq 1} \frac{q^m}{1-q^m}$ (see, e.g., \cite[Proposition~3.1.4]{vI}). 
Hence, for all $\epsilon\geq0$ one has \[\lim_{q\to 1}(1-q)^{1+\epsilon}\langle f \rangle_q \= \lim_{q\to 1}(1-q)^\epsilon \sum_{m\geq 1}\frac{(1-q)\,q^m}{1-q^m}.\] Note that $\lim_{q\to 1} \frac{(1-q)\,q^m}{1-q^m} = \frac{1}{m}$, so that $\sum_{m\geq 1}\frac{(1-q)\,q^m}{1-q^m}$ diverges as $q\to 1$ at logarithmic rate (see, also, \cite{P}). Hence, $\lim_{q\to 1}(1-q)^{1+\epsilon}\langle f \rangle_q$ converges to $0$ for $\epsilon>0$ and diverges to $\infty$ for $\epsilon=0$. In other words, the degree of $f$ is~$1$ and $\Zdegree(f)=\infty$. 
\item\label{it:altMZV} Thirdly, let $f(\lambda)$ be equal to the number of even parts in $\lambda$ minus the number of odd parts. That is, $f(\lambda) = \sum_{m} (-1)^m\, r_m(\lambda)$. Then, (see again \cite[Proposition~3.1.4]{vI})
\[(1-q)\langle f \rangle_q \= \sum_{m\geq 1}(-1)^m\frac{(1-q)q^m}{1-q^m} \,\xrightarrow[q\to 1]{}\,\sum_{m\geq 1} \frac{(-1)^m}{m} \= \log(2). \]
Hence, in this example, the corresponding degree limit, conjecturally, is not a multiple zeta value. 
%In Section~\ref{sec:limits} we generalize this example to other functions which are polynomial in the Stanley coordinates.
\item The degree may be any real number $x\in \R$. Namely, let \[f(\lambda) \= \sum_{m\in \bf m(\lambda)} (-1)^m\,\binom{x}{m} \= \sum_{m\geq 1} (-1)^m\,\binom{x}{m} \, \delta_{r_m(\lambda)\geq 1}\mspace{1mu}.\]
Then (also by \cite[Proposition~3.1.4]{vI})
\[\langle f \rangle_q \= \sum_{m\geq 0} (-1)^m\,\binom{x}{m}\, q^{m} \= (1-q)^x.\]
\item It may happen that $\deg(f)=-\infty$, as is the case for the M\"obius function:
\[\langle \mu\rangle_q \= \Bigl(\sum_{\lambda \in \PP} q^{|\lambda|}\Bigr)^{\! -2} \= \prod_{m\geq 1} (1-q^m)^2.\]
\item If $f$ is such that $g(\tau)=\langle f \rangle_q$ is a cusp form of weight $k$ (with $q=e^{2\pi i\tau}$), the modular transformation $g(\tau) = \tau^{-k}g\bigl(-\frac{1}{\tau}\bigr)$ implies that
\[ \lim_{q\to 1} (1-q)^k \langle f \rangle_q \= \lim_{\tau \to 0} \frac{(1-e^{2\pi i\tau})^k}{\tau^k} g\Bigl(-\frac{1}{\tau}\Bigr) \= \lim_{q\to 0} (-2\pi i)^k \langle f\rangle_q = 0.\]
This illustrates the fact that if $\langle f\rangle_q$ is a quasimodular form of weight~$k$, the degree of $f$ is $k$ and the limit can be computed using the quasimodular transformation (in fact, for quasimodular forms one can recover the full asymptotic expansion as $q\to 1$ using \emph{growth polynomials}; see \cite[Section~9]{CMZ}).

\end{enumerate}
% More generally, if $\langle f\rangle_q$ is a cusp form, then the degree of $f$ is $-\infty$ (this are examples of cases where the degree is not a very interesting notion.)
% The degree of $\ell(\lambda)$ is 1, with limit $\zeta(1)=\infty$. The degree of $|\lambda|$ is $2$ with limit $\zeta(2)$. The degree of the depth is 1, with limit~$1$. 

% \hen{If $H(n)=\sum_{m=1}^n \frac{(-1)^m}{m}$, then $f(\lambda) = \sum_{m>0} H(r_m(\lambda))$ gives $-\log(2)$ (with degree 1). Maybe one can give a general statement saying that for a q-series $g(q) = \sum_{n>0} a_n q^n$ where $q\rightarrow 1$ exists and is non-zero one can construct a function $f$ on partitions by $f(\lambda) = \sum_{m> 0} \sum_{n=1}^{r_m(\lambda)} a_n$ such that $\deg(f) = 1$ and $\Zdegree(f) = g(1)$ (?) }
\end{example}
\begin{remark}
Given $F=\sum_{n\geq 0} a_n q^n\in \q\llbracket q \rrbracket$ there are infinitely many functions $f:\PP\to \q$ for which $\langle f \rangle_q = F$ (the only condition on~$f$ is that $\sum_{|\lambda|=n} f(\lambda) = \sum_{m\geq 0} a_m \,p(n-m)$ with $p(i)$ the number of partitions of $i$). Hence, one could define the degree and the degree limit of $F$ as the degree and degree limit of~$f$ for one such~$f$. However, in the generality of this section, one does not discover a lot of structure in the values $\Zdegree(F)$. In the next section we see how this situation alters when one restricts to a certain polynomial subspace of $\q^\PP$.
\end{remark}
 
\section{Partition analogues of Multiple Zeta Values}
\subsection{Polynomial functions on partitions}
% \jw{Introduce the functions $S$, provide examples/motivation/qMZV}
% \jw{Give the quasishuffle setup}

The spaces of polynomials and modular forms are graded algebras with the property that  after fixing the degree, or weight and a congruence subgroup respectively, the corresponding vector spaces are finite dimensional.  Similarly, the algebra~$\MS$ %of \emph{polynomial functions on partitions} 
admits a weight filtration such that the vector subspace of elements of a fixed weight is finite dimensional. 

This weight filtration is most naturally introduced after giving an equivalent definition for $\MS$. That is, let $\Phi$ be the composition of the $\vec{u}$-bracket (see Definition~\ref{defn:ubrac}), and the mapping $\Psi$ from Definition~\ref{def:MS}. Then, $\Phi$ denotes the linear map
\[\Phi: \bigoplus_{n\geq 0} \q[x_1,\ldots,x_n,y_1,\ldots,y_n] \to \q\llbracket u_1,u_2,\ldots\rrbracket \]
uniquely determined by
\[ g(x_1,\ldots,x_n,y_1,\ldots,y_n)\mapsto \sum_{\substack{m_1>\ldots>m_n>0\\r_1,\ldots,r_n\geq 1}} g(m_1,\ldots,m_n,r_1,\ldots,r_n)\, u_{m_1}^{r_1}\cdots u_{m_n}^{r_n}\]
(see Proposition~\ref{prop:ubracketofs}).
Now, $\mathrm{Im}\Phi$ is an algebra with respect to the natural product on $\q\llbracket u_1,u_2,\ldots\rrbracket$, and $\Phi$ becomes an algebra homomorphism if one defines a corresponding product on the domain. Even more, this domain admits a weight and a depth filtration, determined by assigning to $g(x_1,\ldots,x_n,y_1,\ldots,y_n)$  weight $\deg g + n$ and depth $n$. The main advantage of $\Phi$ over $\Psi$ is that the natural product on the codomain of $\Phi$ (in contrast to the natural product on the codomain of $\Psi$) behaves well with respect to the $q$-bracket (see Proposition~\ref{prop:doubleshuffle}). This yields the following natural definition for the weight and depth filtration on $\MS$. 
\begin{definition}
We call a function $f\in\q^\PP$ a polynomial function on partitions %if $\langle f \rangle_{\vec{u}}\in \mathrm{Im} \Phi,$ and we say $f$ is
of \emph{weight} $\leq k$ and \emph{depth} $\leq p$ if $\langle f \rangle_{\vec{u}}=\Phi(g)$ for some $g$ of weight $\leq k$ and depth~${\leq p}$ (where we recall that $g(x_1,\ldots,x_n,y_1,\ldots,y_n)$ has weight $\leq \deg g + n$ and depth~${\leq n}$).
\end{definition}
Note that %this weight filtration makes $\MS$ into a filtered algebra, and that there 
the vector space of polynomial functions on partitions of bounded weight is finite dimensional. Recall, we denote the $\Q$-vector space of all polynomial function on partitions by
\begin{align*}
   \MS \= \langle f \in \q^\PP \mid f \text{ is a polynomial function} \rangle_\Q
\end{align*}
and write for the subspace of all polynomial functions of weight $\leq k$
\begin{align*}
    \MS_{\leq k} \defis \langle f \in \MS  \mid f \text{ is of weight} \leq k \rangle_\Q\,.
\end{align*}

In the next section we provide a basis for~$\MS$ (or, in fact, several), which one can easily find because of the following two results. Note that this contrasts the situation for multiple zeta values, as well as for $q$-analogues of multiple zeta value, where so far it has not been possible to prove that a certain generating set actually forms a basis. 
\begin{proposition}\label{prop:inj}
The linear map~$\Phi$ is injective.
\end{proposition}
\begin{proof}
Suppose $\Phi(g)=0$ and write $g=(g_0,g_1,\ldots)$ with $g_n\in \q[x_1,\ldots,x_n,y_1,\ldots,y_n]$. Suppose there exists a minimal~$n$ such that $g_n\not\equiv 0$. Then, there are integers ${m_1>\ldots>m_n>0}$ and $r_1,\ldots,r_n\geq 1$ such that
\[g_n(m_1,\ldots,m_n,r_1,\ldots,r_n)\neq 0.\]
Hence, the coefficient of $u_{m_1}^{r_1}\cdots u_{m_n}^{r_n}$ in $\Phi(g)$ is non-zero, contradicting our assumption. 
\end{proof}

%We now give a third equivalent definition. The equivalence of this definition with the previous will be clear when we, in Section~\ref{sec:ubracpolfunc} compute the $\vec{u}$-bracket of polynomial function on partitions. 
%\begin{definition}\label{defn:polfunc} 
\begin{corollary}
Let $f\in\q^\PP$. Then $f\in \MS$ precisely if there exist a $p_0\in \Q$ and for $n\geq 1$  polynomials $p_n \in y_1\dots y_n \q[x_1,\dots,x_n,y_1,\dots,y_{n}]$ with $p_n\equiv0$ for all but finitely many $n$, such that for any partition $\lambda$ we have
\begin{align}\label{eq:representationofpolfunc}
    f(\lambda) \=  p_0 \+ \sum_{n\geq 1}\sum_{m_1>\dots>m_n>0} p_n(m_1,\dots,m_n,r_{m_1}(\lambda),\dots,r_{m_n}(\lambda))\,.
\end{align}
Moreover, the function~$f$
\begin{enumerate}[{\upshape(i)}, leftmargin=*]\itemsep3pt
\item is of \emph{weight} $\leq k$ if $\deg(p_n) \leq k$ for all $n\geq 0$.
\item is of \emph{depth} $\leq r$ if $p_n \equiv 0$ for $n>r$, 
\item %By Proposition~\ref{prop:inj} 
uniquely determines the polynomials~$p_n\mspace{1mu}$.
\end{enumerate}
\end{corollary}
\begin{proof}
Here, we use that $\sum_{r=1}^R r^d$ is a polynomial which is of degree $d+1$ in $R$ and divisible by $R$. For~(iii) we use the previous proposition. 
\end{proof}

%\item By $\MS_{\leq k}$ we denote the space of polynomial functions of weight $\leq k$.

%\end{definition}
% \begin{definition} We call a function $f\in\q^\PP$ a \emph{monomial function on partitions} if for some $n\geq 0$ there exist polynomials $p \in \q[x_1,\dots,x_n]$ and  $q \in x_1 \dots x_n \q[x_1,\dots,x_n]$ such that for any partition $\lambda$ we have
% \begin{align*}
%     f(\lambda) =  \sum_{m_1>\dots>m_n>0} p(m_1,\dots,m_n) \, q(r_{m_1}(\lambda),\dots,r_{m_n}(\lambda))\,.
% \end{align*}
% By $\MS$ we denote the $\Q$-vector space spanned by all monomials on partitions and we refer to any element in $f\in \MS$ as a \emph{polynomial function on partitions}.

% In case $f$ can be written as a linear combination of monomial function on partitions with $n\leq p$ or $\deg pq\leq k$, we say $f$ is of depth $\leq p$ or weight $\leq k$ respectively. 
%\end{definition}

\begin{remark}
The function
\begin{align}\label{eq:k_2=0}
f(\lambda) = \sum_{m_1>m_2>0} m_1 m_2 \, r_{m_1}(\lambda)
\end{align}
is a polynomial function on partitions, although $x_1 x_2 y_1\not \in y_1 y_2\q[x_1,x_2,y_1,y_2]$. Namely, observe that
\[f(\lambda) \= \sum_{m_1>0} \tfrac{1}{2}m_1^2(m_1-1)\, r_{m_1}(\lambda).\]
Observe that~$f$ is of weight $\leq 4$ and depth $\leq 1$; a fact one could not easily read of from the expression~\eqref{eq:k_2=0}. 

More generally, one could relax the condition $p_n\in y_1\cdots y_n\q[x_1,\ldots,x_n,y_1,\ldots,y_n]$ %in Definition~\ref{defn:polfunc} 
by $p_n\in y_1\q[x_1,\ldots,x_n,y_1,\ldots,y_n]$, at the cost of breaking the uniqueness of the representation~\eqref{eq:representationofpolfunc}---as we will see in Corollary~\ref{cor:basis}---and making it harder to read of the degree and depth.
\end{remark}

\begin{remark}
There exist natural extensions $\MS(N)$ of~$\MS$ to higher levels, i.e., such that $\MM(N)$ is the subspace of $\MS(N)$ for which the $q$-brackets are quasimodular forms of level~$N$ and for which the corresponding limits as $q$ tends to an $N$th root of unity are multiple polylogarithms at root of unity (sometimes called \emph{colored multiple zeta values}). For example, the alternating multiple zeta value $\log 2$ can be obtained in this way: see Example~\ref{ex:limit}\ref{it:altMZV}. More concretely, one would take \[\bigoplus_{n\geq 0}\q[x_1,\ldots,x_n,y_1,\ldots,y_n,\zeta^{x_1},\ldots,\zeta^{x_n},\zeta^{y_1},\ldots,\zeta^{y_n} \mid \zeta^N=1]\]
as the domain in the definition of~$\Phi$.
We do not pursue to work out all details in the current work, but instead refer to \cite{I2}, where this has been worked out in a similar setting. 
\end{remark}

\subsection{Bases for the space of polynomial functions on partitions}
%To be more concrete, w
We define a basis of $\MS$, or rather different bases. These bases depend on a polynomial sequence 
%(a sequence of polynomials indexed by the non-negative integers in which each index is equal to the degree of the corresponding polynomial) 
in the following way. 

% \begin{example}
% We define the following choices for well-normalized families of polynomials
% \def\arraystretch{1.5}%
% \begin{table}[ht!]
% \begin{tabular}{l l l}
% Abbr. & Name & Definition of $f_k$ for $k\geq 1$ \\\hline
% $\mathrm{m}$ & monomial powers & $f_k(x)=x^k$ \\
% $\mathrm{s}$ & Bernoulli--Seki polynomials & $f_k(x)-f_k(x-1) = x^{k-1}$ and $f_k(0)=0$ \\
% $\mathrm{b}$ & binomial coefficients & $f_k(x) = \binom{x}{k}$ \\
% $\mathrm{b}^+$ &  shifted binomial coefficients & $f_k(x) = \binom{x+1}{k}-\delta_{k,1}$
% \end{tabular}
% \end{table}
% \end{example}

\begin{definition}
Let $\FF=\{f_k\}_{k=0}^\infty$ be a polynomial sequence (i.e., $f_k\in \q[x]$ is of degree $k$) such that $f_0=1$ and $f_k(0)=0$ for $k\geq 1$. For $r\geq 0$, $k_1\geq 1, k_2,\dots,k_r \geq 0$ and $d_1,\dots,d_r \geq 0$, we define the map 
 \[  \ms[\FF]{k_1, \dots , k_r}{d_1,\dots,d_r}  : \PP \rightarrow \Q \]
 by 
\begin{align*}
\msl[\FF]{k_1, \dots , k_r}{d_1,\dots,d_r}{\lambda} \defis \sum_{m_1 > \dots > m_r > 0}\, \prod_{j=1}^{r} m_j^{d_j} f_{k_j}(r_{m_j}(\lambda))
\end{align*}
and by $P_{\FF}(\lambda)=1$ if $r=0$. 
%Often we omit the subscript $\FF$. 
We write
\begin{align*}
\hm[\FF]{k_1, \dots , k_r} \=  \ms[\FF]{k_1, \dots , k_r}{0,\dots,0}, \quad \jm[\FF]{d_1, \dots , d_r} \=  \ms[\FF]{1, \dots , 1}{d_1,\dots,d_r}.
\end{align*}
\end{definition}
Note that $\mssmall[\FF]{k_1, \dots , k_r}{d_1,\dots,d_r} \in \MS$, $\hm[\FF]{k_1, \dots , k_r} \in \MS\cap \mathscr{H}$ and $ \jm[\FF]{d_1, \dots , d_r}  \in \MS\cap \mathscr{J}$, where $\mathscr{H}$ and $\mathscr{J}$ were defined in Definition~\ref{def:HandJ}. 
%We refer to the functions $P_\FF$ as \emph{partition analogues of multiple zeta values}.

\begin{definition} A polynomial sequence $\FF=\{f_k\}_{k=0}^\infty$ with $f_0=1$ and $f_k(0)=0$ for $k\geq 1$ is called \emph{well-normalized} if the leading coefficient of $f_k$ equals $\frac{1}{k!}$. 
In Table~\ref{table:polmodels}, we define four families of polynomials of which the last three are well-normalized.%\clearpage
    %\vspace{-0.3cm}
\def\arraystretch{1.5}%
\begin{table}[ht!]
\begin{tabular}{l | l | l}
% Abbr. & Name & Definition of $f_k$ for $k\geq 1$ \\\hline
$\FF$ & Name & Definition of $f_k$ for $k\geq 1$ \\\hline
$\mathrm{m}$ & monomial model & $f_k(x)=x^k$ \\
$\mathrm{s}$ & Bernoulli--Seki model\tablefootnote{The corresponding polynomials are often referred to as Faulhaber polynomials, but here we use the name which we believe is historically more correct.} & $f_k(x)-f_k(x-1) = \frac{1}{(k-1)!}x^{k-1}$ and $f_k(0)=0$ \\
$\mathrm{b}$ & binomial model & $f_k(x) = \binom{x}{k}$ \\
$\mathrm{b}^+$ &  shifted binomial model & $f_k(x) = \binom{x+1}{k}-\delta_{k,1}$
\end{tabular}\vspace{8pt}
\caption{Overview of different models of partition analogues.}
\label{table:polmodels}
\vspace{-10pt}
\end{table}
\end{definition}

% A main result is the fact that the monomial functions $P$ (for any choice of ${\FF}$) form a basis of $\MS$, rather than a generating set. Note that this contrasts the situation for multiple zeta values, and for $q$-analogues of multiple zeta value, where so far it has not been possible to prove that a certain generating set actually forms a basis. 
\begin{corollary}[of Proposition~\ref{prop:inj}]\label{cor:basis}
For all polynomial sequences~$\FF$, a basis for $\MS_{\leq k}$ is given by the functions $P_{\FF}\bigl(\ar{k_1, \dots , k_r}{d_1,\dots,d_r} \bigr)$ for all $r\geq 0$, $k_1,\dots,k_r \geq 1\,,$ and $d_1,\dots,d_r \geq 0$ such that $k_1+\ldots+k_r+d_1+\ldots+d_r \leq k$. 
\end{corollary}

%\item Denote by $\MS$ the following $\Q$-vector space
%\begin{align*}
%\MS = \big\langle  \ms{k_1, \dots , k_r}{d_1,\dots,d_r} \Big| \,r\geq 0\,, k_1,\dots,k_r \geq 1\,, d_1,\dots,d_r \geq 0 \big\rangle_\Q
%\end{align*}
%and denote its subspaces by
%\begin{align*}
%\MH(n) &= \big\langle  \hm{k_1, \dots , k_r} \Big| \,r\geq 0\,, k_1\geq 1 \text{ and } k_1,k_2\dots,k_r \geq n\big\rangle_\Q \subset \MS \\
%\MJ &= \big\langle  \jm{d_1, \dots , d_r} \Big| \,r\geq 0\,, d_1,\dots,d_r \geq 0  \big\rangle_\Q \subset \MS \,.
%\end{align*}
%\item In addition, denote by $\MS^\text{set}$ the underlying set of generators, i.e.\
%\begin{align*}
%\MS^\text{set} = \big\{  \ms{k_1, \dots , k_r}{d_1,\dots,d_r} \Big| \,r\geq 0\,, k_1,\dots,k_r \geq 1\,, d_1,\dots,d_r \geq 0 \big\}
%\end{align*}
%and similarly for its subspaces.

%\item In addition we define the subspace
%\begin{align*}
%\MH^{+} = \big\langle  \hm{k_1, \dots , k_r} \Big| \,r\geq 0\,, k_1,\dots,k_r \geq 2\,,  \big\rangle_\Q \subset \MH \,.
%\end{align*}
%\end{enumerate}

%\begin{remark}
%Equivalently, we could have defined
%\[\MS=\big\langle  \ms{k_1, \dots , k_r}{d_1,\dots,d_r} \Big| \,r\geq 0\,, k_1\geq 1,k_2,\dots,k_r \geq 0\,, d_1,\dots,d_r \geq 0 \big\rangle_\Q.\]
%However, by defining $\MS$ as above it turns out that the generators are linearly independent. 
%%This implies that $\MH(0)\subset \MS$. 
%\end{remark}

\subsection{Multiplication, conjugations and brackets of partition analogues}\label{sec:ubracpolfunc}
In Section~\ref{sec:operations} we introduced the $\vec{u}$-bracket, two natural conjugation operations, and three natural product operations on the space of all functions on partitions. We will now explain how the space of \emph{polynomial} functions on partitions behaves under these operations.  

First of all, the $\vec{u}$-bracket of elements of $\MS$ is given in Proposition~\ref{prop:ubracketofs} below. As a corollary, the $q$-bracket of such an element can be seen as a $q$-analogue of multiple zeta value as introduced in \cite{B} and further discussed in \cite{BK2}. In \cite{BK2} the authors denote by $\mzv_q$ the space spanned by all $q$-series of the form
\begin{align}\label{eq:defqmzv}
    \sum_{m_1 > \dots > m_r > 0} \prod_{j=1}^r \frac{Q_{k_j}(q^{m_j})}{(1-q^{m_j})^{k_j}}\,,
\end{align}
for $k_1,\dots,k_r\geq 1$ and polynomials $Q_{k_j}(X) \in \Q[X]$ with $\deg Q_{k_j} \leq k_j$ and ${Q_{k_1}(0)=0}$. There it was shown  (\cite[Theorem 1]{BK2}) that $\mzv_q$ is spanned by the $q$-series \eqref{eq:faulhaberqseries}, whose corresponding polynomial functions on partition also span our space $\MS$ (see Example~\ref{ex:qbrac}\ref{it:bibrac}).
That is, 
$ \langle \MS  \rangle_q$ equals the space of $q$-analogues $\mzv_q\mspace{1mu}$.
% \begin{align*}
%     \mzv_q \defis \langle \MS  \rangle_q\,.
% \end{align*}
%It is easy to check this is exactly the same space as introduced in \cite{BK2}, e.g., by making use of the Proposition~\ref{prop:ubracketofs} below. 
In Section~\ref{sec:limits}, we show that the (regularized) limit of $q\rightarrow 1$ of the elements in $\mzv_q$ are always given by (regularized) multiple zeta values. 

Given a polynomial $f$, for an integer $n$, denote by $\partial f(n)=f(n)-f(n-1)$ if $n\geq 1$ and let $\partial f(0)=f(0)$. Then \cite[Proposition~3.1.4 and Lemma~3.5.2]{vI} yields:
\begin{proposition}\label{prop:ubracketofs}
The $\vec{u}$-bracket of $\mssmall[\FF]{k_1,\ldots,k_r}{d_1,\ldots,d_r}$ is given by
\[\left\langle \ms[\FF]{k_1, \dots , k_r}{d_1,\dots,d_r}\right\rangle_{\!\!\vec{u}} \=\sum_{m_1 > \dots > m_r > 0} \prod_{j=1}^r\Bigl(m_j^{d_j} \sum_{n\geq 0}  \partial f_{k_j}(n)\, u_{m_j}^{n}\Bigr).\]
% Given $k_1,\ldots, k_r$, denote $J=\{j\mid k_j\neq 0\}.$ The $\vec{u}$-bracket of the multiple symmetric polynomials are given by
% \[\left\langle \ms[\FF]{k_1, \dots , k_r}{d_1,\dots,d_r}\right\rangle_{\vec{u}} =\sum_{m_1 > \dots > m_r > 0} \prod_{i=1}^r m_i^{d_i}\prod_{j\in J}^r \left(\sum_{n_j>0}\partial f_{k_j}(n_j) u_{m_i}^{n_i}\right).\]
\end{proposition}
Recall that $f_0=1$, so that $\partial f_0(n)=\delta_{n,0}\mspace{1mu}$. Moreover, for $k\geq 1$, we have $f_k(0)=0$, so that $\partial f_k(0)=0$. 

\begin{ex}\label{ex:qbrac} For some particular choices of $\FF$, the $q$-bracket of partition analogues of multiple zeta values gives rise to several models for $q$-analogues of multiple zeta values:
\begin{enumerate}[(i), leftmargin=*]\itemsep3pt
\item \label{it:bibrac} The Bernoulli--Seki model $\FF=\mathrm{s}$ corresponds, up to a factor, to the \emph{bi-brackets} introduced by the first author in \cite{B}
\begin{align}\label{eq:faulhaberqseries}
\Bigl\langle\ms[\mathrm{s}]{k_1, \dots , k_r}{d_1,\dots,d_r}\Bigr\rangle_q &= %\left[\begin{matrix} k_1 & \dots & k_r \\ d_1 & \dots & d_r\end{matrix}\right] :=
\sum_{\substack{m_1>\ldots>m_r>0 \\ n_1,\ldots,n_r>0}} \prod_{j=1}^r m_j^{d_j}\frac{n_j^{k_j-1}}{(k_j-1)!} q^{m_jn_j} & (k_i\geq 1).
%\langle\hm[\mathrm{s}]{k_1,\dots,k_r}\rangle_q &= \left[\begin{matrix}k_1 & \dots & k_r\end{matrix}\right] :=
%\sum_{\substack{m_1>\ldots>m_r>0 \\ n_1,\ldots,n_r>0}} \prod_{j=1}^r n_j^{k_j-1} q^{m_jn_j}, & (k_i\geq 1)\\
\end{align}
For $r=1$, $d_1=0$ and even $k_1\geq 2$ these are, up to the constant term, the Eisenstein series $G_{k_1} = - \frac{B_{k_1}}{2{k_1}!} + \bigl\langle\mssmall[\mathrm{s}]{k_1}{0}\bigr\rangle_{\!q}\mspace{1mu}$, defined in~\eqref{eq:eisenstein}. 
In the case $r=1, k_1>1, d_1>0$ and $k_1+d_1$ even, they are essentially the $d_1$-th derivative of $G_{k_1-d_1}$. For example for $k>d>0$ we have
\begin{align}\label{eq:derivofeisensteinseries}
		\Bigl\langle\ms[\mathrm{s}]{k}{d}\Bigr\rangle_{\!q} &\= \frac{(k-d-1)!}{(k-1)!} \left(q \frac{d}{dq}\right)^d G_{k-d}\,.
\end{align}
Together with 
\begin{align*}
\Bigl\langle\ms[\mathrm{s}]{1}{1}\Bigr\rangle_{\!q}  \= \Bigl\langle\ms[\mathrm{s}]{2}{0}\Bigr\rangle_{\!q}  \= G_2 +\frac{1}{24}
\end{align*}
and 
\begin{align*}\Bigl\langle\ms[\mathrm{s}]{k}{d}\Bigr\rangle_{\! q}  \=   \frac{d!}{(k-1)!} \Bigl\langle\ms[\mathrm{s}]{d+1}{k-1}\Bigr\rangle_q 
\end{align*}
we see that all $q$-brackets of $\ms[\mathrm{s}]{k}{d}$ with $k+d$ even are in the ring~$\QM$ of quasimodular forms.
\item The binomial model $\mathcal{F}=\mathrm{b}$ corresponds to the so-called Schlesinger--Zudilin  $q$-analogues of multiple zeta values \cite{S,Zu} given by 
\begin{align}
\langle\hm[\mathrm{b}]{k_1, \dots , k_r}\rangle_q &\= %\zeta_q^{\mathrm{SZ}}(k_1,\ldots,k_r):=
\sum_{m_1>\cdots>m_r>0}\prod_{j=1}^r\frac{q^{m_jk_j}}{(1-q^{m_j})^{k_j}} \qquad (k_1\geq 1, k_i\geq 0).
\end{align}
\item Shifted binomial polynomials correspond to the $q$-analogues of multiple zeta values introduced by Bradley \cite{Bra} and Zhao given by 
\begin{align}
\langle\hm[\mathrm{b}^+]{k_1, \dots , k_r}\rangle_q &\= %\zeta_q^{\mathrm{BZ}}(k_1,\ldots,k_r):=
\sum_{m_1>\cdots>m_r>0}\prod_{j=1}^r\frac{q^{m_j(k_j-1)}}{(1-q^{m_j})^{k_j}} & (k_1\geq 2, k_i\geq 1).
\end{align}
%\langle\jm[\mathrm{m}]{d_1, \dots , d_r}\rangle_q &=\jw{???}:=\sum_{m_1>\cdots>m_r}\prod_{j=1}^r m_j^{d_j}\frac{q^{m_j}}{1-q^{m_j}}. 
%\end{align}
\end{enumerate}
\end{ex}

The three products on the space of all functions on partitions naturally reduce to~$\MS$. That is, we obtain three commutative $\Q$-algebras $(\MS, \ost)$, $(\MS, \oshh)$ and $(\MS, \odot)$ with a natural involution, as stated in the following result.

%\filbreak
\begin{proposition} \label{prop:polfctonpartprop}The space $\MS$ of polynomial function on partitions
\begin{enumerate}[{\upshape (i)}, leftmargin=*]\itemsep3pt
\item  is closed under the conjugation $\iota$.
\item is closed under the harmonic, shuffle and pointwise products $\ost$, $\oshh$ and $\odot$.
\item has $D$ (see Definition~\ref{def:products}) as a derivation on both $(\MS,\ost)$ and $(\MS,\oshh)$.
    \item admits the algebra homomorphisms given in the following commutative diagram of differential $\Q$-algebras (with the derivation $q\frac{d}{dq}$ on $\mzv_q$)
\[\begin{tikzcd}[row sep=tiny]
& (\MS,\ost) \ar[dr,"\langle\ \rangle_q"] \ar[dd, "\iota", leftrightarrow]
& \\
&
& \mzv_q \\
& (\MS,\oshh) \ar[ur,"\langle\ \rangle_q"']
&
\end{tikzcd}\raisebox{-3pt}{.}\]
% \begin{align*}
% (\MS, \bullet) &\longrightarrow \mzv_q\\
% f &\longmapsto \langle f \rangle_q\,,
% \end{align*}
\end{enumerate} 
\end{proposition}

\begin{proof}
The first statement follows since the conjugate of $(r_1,\ldots, r_d)\times (m_1,\ldots, m_d)$ is given by $(m_1, m_{2}-m_1,\ldots,m_d-m_{d-1})\times (r_1+\ldots+r_d,\ldots,r_1)$ and therefore the conjugate of a polynomial in Stanley coordinates is again a polynomial in Stanley coordinates. Statement (ii) follows from Proposition~\ref{prop:ubracketofs} together with Corollary~\ref{cor:basis}, since the harmonic, as well as the pointwise product, of two $\vec{u}$-brackets of $\ms[\FF]{\dots}{\dots}$ can again be expressed as a linear combination of $\vec{u}$-brackets of $\ms[\FF]{\dots}{\dots}$. We will see this explicitly in Proposition~\ref{prop:alghoms}. Together with (i) this then also shows that the space is closed under~$\oshh$. Since $D$ is the derivation on $\q\llbracket u_1,u_2,\ldots\rrbracket$, defined on the generators by $D u_n =n u_n\mspace{1mu}$, (iii) follows immediately by the fact that both~$\ost$ and~$\oshh$ are defined via the $\vec{u}$-bracket. 
Finally (iv) follows from (i)--(iii) and the definition of~$\ost$ by the $\vec{u}$-bracket after setting $u_n = q^n$.
\end{proof}
\begin{remark}
The space $\MM \subset \MS$ is also closed under $\iota$ and the products $\ost$, $\oshh$. Moreover, since the space of quasimodular forms is closed under $q \frac{d}{dq}$, the space $\MM$ is closed under $D$ as well, i.e. $\MM$ satisfies almost all properties of Proposition~\ref{prop:polfctonpartprop}. But one can check that $\MM$ is not closed under $\odot$.
In general, it is not clear how $\MM$ can be described explicitly as a subspace of $\MS$. Since it contains the kernel of the $q$-bracket, it might be difficult to give an explicit description. In this note, we will give several examples of explicit elements in $\MM$, but expect that there are much more (see, e.g., Example~\ref{ex:pi}). 
Another question is if there are other interesting subalgebras of $\MS$, which satisfy (some of) the properties of Proposition~\ref{prop:polfctonpartprop}.
\end{remark}

\begin{remark} 
The space $\MS$ is not closed under the conjugation $\omega$. For example, for $\FF$ well-normalized
$\omega(\ms[\FF]{1}{0})(\lambda) = \lambda_1$
equals the biggest part of the partition~$\lambda$. 
It is a small computation to verify that
\[ \langle \lambda_1 \rangle_{\vec{u}} \= - \sum_{\lambda\in \PP} \mu(\lambda)\, \lambda_\ell \, u_{\lambda} \,\not\in \,\langle \MS\rangle_{\vec{u}}\qquad (\ell=\ell(\lambda) \text{ is the \emph{length} of } \lambda).\]
The subspace~$\Lambda^*$ of $\MS$, defined in the introduction, however, has the remarkable property that it is closed under~$\omega$ but not under $\iota$. 
\end{remark}
\begin{remark}\label{rk:iotaexplicit}
For the Bernoulli--Seki model $\FF=\mathrm{s}$ the involution~$\iota$ is explicitly given by
\begin{align}\label{eq:iotaexplicit}
\iota\ \ms[\mathrm{s}]{k_1, \dots , k_r}{d_1,\dots,d_r} \= 
\sum_{\substack{a_1,\ldots,a_r\geq 1 \\ |\vec{a}|=|\vec{d}|+r}}\sum_{\substack{b_1,\ldots,b_r\geq 0 \\ |\vec{b}|=|\vec{k}|-r}}
C^{\vec{a},\vec{k}}_{\vec{b},\vec{d}}\, \ms[\mathrm{s}]{a_1,\ldots,a_r}{b_1,\ldots, b_r},
\end{align}
% \[\hen{\iota\ \ms[\mathrm{s}]{k_1, \dots , k_r}{d_1,\dots,d_r} = 
% \sum_{\substack{a_1,\ldots,a_r\geq 1 \\ |\vec{a}|=|\vec{d}|+r}}\sum_{\substack{b_1,\ldots,b_r\geq 0 \\ |\vec{b}|=|\vec{k}|-r}}
% C^{\vec{a},\vec{k}}_{\vec{b},\vec{d}} \ms[\mathrm{s}]{a_1,\ldots,a_r}{b_1,\ldots, b_r},}
% \]
where for $\vec{a},\vec{b},\vec{d},\vec{k}\in \Z^r$ the constant $C^{\vec{a},\vec{k}}_{\vec{b},\vec{d}}$ is given by
\[C^{\vec{a},\vec{k}}_{\vec{b},\vec{d}} \=(-1)^{|\vec{b}|}\prod_{j=1}^r\textstyle\binom{s_j(\vec{d})-s^j(\vec{a})+j-1}{a_{r-j+1}-1}\binom{k_{r-j+1}-1}{s_j(\vec{b})-s^j(\vec{k})+j-1} \frac{(a_j-1)!}{(k_j-1)!}(-1)^{s_j(\vec{k})+s^j(\vec{b})+j}
\]
and where for $\vec{\ell}\in \z^r$ and $j=1,\ldots,r$, we write
\[ s_j(\vec{\ell}) = \sum_{i=1}^j \ell_i, \qquad s^j(\vec{\ell}) = \sum_{i=r-j+2}^r \ell_i\,.\]
% \[C^{\vec{a},\vec{k}}_{\vec{b},\vec{d}}=\prod_{j=1}^r\textstyle\binom{d_1+\ldots+d_j-a_r-\ldots-a_{r-j+2}+j-1}{a_{r-j+1}-1}\binom{k_{r-j+1}-1}{b_1+\ldots+b_j-k_r-\ldots-k_{r-j+2}+j-1}(-1)^{k_1+\ldots+k_{r-j+1}-b_1-\ldots-b_j-j}
% \]
%\[\hen{C^{\vec{a},\vec{k}}_{\vec{b},\vec{d}}=\prod_{j=1}^r\textstyle\binom{d_1+\ldots+d_j-a_r-\ldots-a_{r-j+2}+j-1}{a_{r-j+1}-1}\binom{k_{r-j+1}-1}{b_1+\ldots+b_j-k_r-\ldots-k_{r-j+2}+j-1}(-1)^{|\vec{k}|-k_1-\dots-k_j-b_1-\ldots-b_j-j}}
%\]
\end{remark}

Since the $q$-bracket is $\iota$-invariant \eqref{eq:iotaexplicit} gives explicit linear relations among the $q$-brackets of the Bernoulli--Seki model \eqref{eq:faulhaberqseries}. This corresponds to the partition relation in \cite[Theorem 2.3]{B}. There it was also conjectured that these families of relations together with the stuffle product are enough to write any $q$-series \eqref{eq:faulhaberqseries} as a linear combination of those with either $d_1=\dots=d_r=0$ or $k_1=\dots=k_r=1$. In our setup here this translates to the following model independent conjecture:
\begin{conjecture}We have 
\begin{align*}
  \mzv_q = \langle \MS  \rangle_q =  \langle \MS  \cap \mathscr{H}\rangle_q =  \langle \MS  \cap \!\mathscr{J}\rangle_q\,.
\end{align*}
\end{conjecture}
Notice that the first equation here is known as explained in the remark at the beginning of Section~\ref{sec:ubracpolfunc} and the last equation is clear since $\iota(\MS  \cap \mathscr{H}) = \MS  \cap \mathscr{J}$. The functions in $\MS  \cap \mathscr{H}$ are exactly those polynomial functions on partitions where the polynomials in Definition~\ref{def:MS} are independent of the variables $x_i$. This conjecture was also discussed in more detail for the Bernoulli--Seki model in \cite{BK2} and \cite{Bri2}.

\subsection{Some remarks about the M\"oller transform}
In contrast to the conjugation, products and derivation above, some natural operations on functions on partitions do not leave the space of polynomial functions invariant. One of the most interesting such operators is the M\"oller transform $\mathcal{M}:\q^\PP\to\q^\PP$ in %discovered by M\"oller in
\cite{CMZ}:
\begin{definition} Given $f:\PP\to \q$, its \emph{M\"oller transform} $\mathcal{M}f:\PP\to\q$ is the function given on a partition~$\lambda$ of size~$n$ by
\[\mathcal{M}f(\lambda) \= \frac{1}{n!} \sum_{\mu \in \PP(n)} |C_\mu|\,\chi_\lambda(\mu)^2 \, f(\mu),\]
where $|C_\mu|$ denotes the size of the conjugacy class associated to $\mu$ in the symmetric group $\mathfrak{S}_n\mspace{1mu}$, and $\chi_\lambda(\mu)$ denotes value of the character of the irreducible representation of $\mathfrak{S}_n$ associated to $\lambda$ at any element in the conjugacy class associated to $\mu$.
\end{definition}
The key property of the M\"oller transform (which follows directly from the second orthogonality relation for the characters of the symmetric group) is that for all $f:\PP\to \q$ one has
\[\langle \mathcal{M}f\rangle_q = \langle f \rangle_q\,.\]
For the following subvectorspace of $\MS$, the M\"oller transform is known to preserve the space of polynomial functions: 
\begin{proposition}\label{prop:moller}
Let $f\in \MS$. Then $\mathcal{M}(f)\in \MS$ if $f$ is in the vector space generated by the constant function~$1$ and
\[ \ms{1}{k-1} \odot \underbrace{\ms{1}{1} \odot \cdots \odot \ms{1}{1}}_{m}\]
for any $k\geq 2$ and $m\geq 0$.
\end{proposition}
\begin{remark}
As the statement of this proposition does not depend on the choice of $\mathcal{F}$, we also omit it from the notation.
\end{remark}
\begin{proof}
For all $k\geq 1$, the M\"oller transform $\mathcal{M}\ms{1}{k-1}$ equals the hook-length moment~$T_k\mspace{1mu}$, given by
\[ T_k(\lambda)\defis \sum_{\xi \in Y_\lambda} h(\xi)^{k-2},\]
where $Y_\lambda$ denotes the Young diagram of $\lambda$, $\xi$ a cell in this diagram and $h(\xi)$ the hook-length of this cell. Note that for even $k$ this function lies in the space of shifted symmetric functions defined in the introduction, as is shown in \cite[Thm.~13.5]{CMZ}. For all $k\geq 2$, the results of Section~\ref{sec:armbein} imply that $\mathcal{M}\ms{1}{k-1}$ is a polynomial function on partitions.

Also, any polynomial in $|\lambda|$, that is, any polynomial in $\ms{1}{1}$ with respect to $\odot$ is invariant under the M\"oller transform---as follows easily from the first orthogonality relation. In fact, $\mathcal{M}(f\odot \ms{1}{1})=\mathcal{M}(f)\odot \ms{1}{1}$ from which the result follows.
\end{proof}
\begin{remark}
For $\mathcal{M}\ms{1}{0}$, the situation is totally different. In this case
\[ \mathcal{M}\ms{1}{0}(\lambda) \= \sum_{\xi \in Y_\lambda} h(\xi)^{-1}.\]
As the $q$-bracket of a function and its M\"oller transform agree, we would expect this function to be of weight~$1$ as well, if it were a polynomial function on partitions. It is a matter of linear algebra to check that $\mathcal{M}\ms{1}{0}$ is not a polynomial function on partitions of weight $\leq 1$. 
\end{remark}

In fact, the 233-dimensional vector space of polynomial functions on partitions of weight $\leq 6$ only has a $10$-dimensional subspace for which the M\"oller transform is again a polynomial function on partitions---this subspace is generated by $1$, $\ms{1}{k-1}$ for $k=2,\ldots,6$, $\ms{1}{1}\odot\ms{1}{1}, \ms{1}{2}\odot\ms{1}{1}, \ms{1}{1}\odot\ms{1}{1}\odot\ms{1}{1}$ and $\ms{1}{3}\odot\ms{1}{1}$.
\begin{lemma}
For $f\in \MS_{\leq 6}$ the converse in Proposition~\ref{prop:moller} holds.
\end{lemma}
Note that this lemma can be proven by finding the kernel of a matrix containing at least $2\cdot 233+1$ values of all these polynomial functions as well as of their M\"oller transforms, as we did in Pari/GP. Hence, it seems the converse of Proposition~\ref{prop:moller} holds for all $f\in \MS$.
% \begin{proposition} \label{prop:derivcommute}For every $k\geq 1$, the following diagram commutes
% \begin{center}
% \begin{tikzcd}
%   \de_k \arrow[r, "\rho"] \arrow[d,"\partial_k",swap]
%     & \widetilde{\mathcal{M}}_k \arrow[d, "q\frac{d}{dq}"] \\
%   \de_{k+2} \arrow[r, "\rho^\Kenz_{k+2}",swap]
% &  \widetilde{\mathcal{M}}_{k+2} 
% \end{tikzcd}
% \end{center}
% where $\partial_k$ is the map given in Proposition~\ref{prop:partialk}.
% \end{proposition}

\section{From partitions to Multiple Zeta Values}\label{sec:parttomzv}
\subsection{Weight and degree limits of partition analogues}\label{sec:limits}
For functions ${f\in \MS}$, recall $\Zdegree(f)$ and $\mathrm{deg}(f)$ are defined such that
\[ (1-q)^{\mathrm{deg}(f)}\langle f \rangle_q \= \Zdegree(f) \+ O(1-q) \]
asymptotically for \emph{real} $q$. We will see that the degree is a non-negative \emph{integer} bounded above by the weight of $f$. In this section we prove Theorem~\ref{thm:main}. 
That is, we compute this degree, and show that $\Zdegree(f)$ is a multiple zeta value, which justifies to call the elements of $\MS$ partition analogues of multiple zeta values. From now on, we omit the \emph{well-normalized} family~$\FF$ from the notation and write $P$ instead of $P_\FF$, unless the results depend on the choice of~$\FF$. 

To get started we discuss two examples. 
First of all, we consider the degree of derivatives of polynomial functions on partitions. 
\begin{lemma}\label{lem:limder}
Given $f\in \MS$, we have
\[\deg(Df) \= \deg f +1, \qquad \Zdegree(Df) \= \deg f \cdot \Zdegree(f).\]
\end{lemma}
\begin{proof}
This follows directly from L'Hôpital's rule, as for all $k\neq 0$ one has
\begin{align}\lim_{q\to 1} (1-q)^k \langle f\rangle_q &\= \lim_{q\to 1} \frac{1}{k} (1-q)^{k+1} q\frac{\partial}{\partial q}\langle f\rangle_q \=  \lim_{q\to 1} \frac{1}{k} (1-q)^{k+1} \langle Df\rangle_q\,. \qedhere\end{align}
\end{proof}

As another example, we consider polynomial functions on partitions of depth $\leq 1$, corresponding to \emph{single} zeta values. In order to calculate degree limits explicitly we define for $k\geq 1$ the (slightly modified) Eulerian polynomials~$E_k$ by
\begin{align}\label{eq:Eulerian} \frac{E_k(X)}{(1-X)^k} \defis \sum_{d\geq 1} \frac{d^{k-1}}{(k-1)!} X^d .\end{align}
Also, write $[n]_q=\frac{1-q^n}{1-q} = 1+q+\dots+q^{n-1}$. 
\begin{lemma}\label{lem:lim1} Let $k\geq 1, d\geq 0$. Then, 
\[\degree \ms{k}{d} = \max(k,d+1), \quad \Zdegree\ms{k}{d} = 
\begin{cases} \zeta(k-d) & k\geq d+1 \\
\frac{d!}{(k-1)!} \zeta(d-k+2) & k\leq d+1.
\end{cases}\]
\end{lemma}
\begin{proof}
Consider the limit $q\to 1$ of the function
$(1-q)^a \sum_{m,n} \frac{m^{k-1}}{(k-1)!}n^{d} q^{mn}$, which is the only term contributing to the degree limit for any well-normalized family. 
For $a=k$, this equals
\begin{align}\label{eq:lem41}\sum_{n\geq 1} n^d \,\frac{E_{k}(q^n)}{[n]_q^k}.\end{align}
%\jw{OLD:
%\begin{align}\label{eq:lem41}\sum_{n}  \frac{(1-q)^a}{(k-1)!} D^{k-1} n^{-k+1} \frac{q^n}{1-q^n}.\end{align}}
By Proposition~\ref{prop:review} below the terms in the sum~\eqref{eq:lem41} are bounded by $(k+1)n^{d-k}$ and therefore the sum converges absolutely if $k-d>1$.  
% \jw{OLD:
% We will show that one can interchange sum and limit in this expression for $a=k$. Write
% \[ 
% f_{k,n}(q) \defis \frac{(1-q)^{k}}{n^{k-1}\,(k-1)!}D^{k-1} \frac{q^n}{1-q^n}.
% \]
% By \cite[Lemma~6.6]{BK} one has %$f_{1,n}(q) \leq \frac{1}{n}$
% %and
% $f_{k,n}(q) \leq \frac{1}{n^2}$ if $k\geq 2.$
% Hence, for $k\geq 2$ every summand in \eqref{eq:lem41} is bounded between $0$ and $n^{-2}$, i.e., the sum converges absolutely. }
As $\lim_{q\to 1} \frac{E_k(q^n)}{[n]_q^k} = \frac{1}{n^k}$, in that case it follows that
\[\lim_{q\to 1} (1-q)^k\, \ms{k}{d} \= \zeta(k-d).\] 
Conversely, one obtains the case $k-d<1$ by the same argument interchanging the roles of $k$ and $d$.
Finally, in case $k-d=1$, a more careful analysis as in \cite{P}, shows the limit diverges---in accordance with the divergence of $\zeta(1)$. 
Also, after letting $a=k+\epsilon$ and $k-\epsilon$ for some $\epsilon>0$, the limit converges to $0$ and diverges to $+\infty$ respectively. 
%
% The cases with $k=1$ and $d\geq 0$ are obtained in the same way, as the $q$-bracket of $\ms{d+1}{0}$ and $\ms{1}{d}$ agree up to a constant. The cases with $k>1$ and $d>0$ then follow by (iteratively) applying the previous lemma by making use of the relation $D\mssmall[\mathrm{s}]{k}{d} = k\mssmall[\mathrm{s}]{k+1}{d+1}$.
\end{proof}
\begin{remark}
Note that Lemma~\ref{lem:lim1} and Lemma~\ref{lem:limder} are in agreement with the relation $D\mssmall[\mathrm{s}]{k}{d} = k\mssmall[\mathrm{s}]{k+1}{d+1}$.
\end{remark}

In the previous lemma, the proof of absolute convergence of the sum~\eqref{eq:lem41} was postponed. Note that by applying the techniques for asymptotics of sums of the form $\sum_{n} f(nt)$ in \cite{Zag}, as was previously observed in \cite[Lemma~1]{Zu2}, one obtains the following lemma (where the Eulerian polynomials~$E_k$ are defined by~\eqref{eq:Eulerian}).
\begin{lemma}\label{lem:ekbound} For $k,n\geq 1$ 
\begin{align*}
    \frac{E_k(q^n)}{[n]^k_q} \= \frac{1}{n^k}\+O(1-q), \qquad (q\to 1^{-}).
\end{align*}
%where $[n]_q=\frac{1-q^n}{1-q} = 1+q+\dots+q^{n-1}$. %Equality only holds in the case $q=1$.
\end{lemma}
We will give an upper bound for the left-hand side for all $q\in [0,1]$ in the following proposition. The special cases $k=1,2$ were obtained before in \cite[Lemma~6.6(ii)]{BK}.
\begin{proposition}
\label{prop:review}
For $k,n\geq 1$ and $q\in [0,1]$ we have
\[
\frac{E_{k}(q^{n})}{[n]_{q}^{k}}
\,\le\,
\frac{k+1}{n^{k}}.
\]
\end{proposition}
%%%%%%%%%%%%%%%%%%%%%%%%%%%%%%%%%%%%%%%%
\begin{proof}
By continuity, we may assume $q\in(0,1)$.
By definition, we have
\[
\frac{E_{k}(q^{n})}{[n]_{q}^{k}}
\=
\frac{E_{k}(q^{n})}{(1-q^{n})^{k}}\,(1-q)^{k}
\=
\biggl(\sum_{d\ge 1}\frac{d^{k-1}}{(k-1)!}q^{dn}\biggr)(1-q)^{k}.
\]
Thus, it suffices to show
\[
\biggl(\sum_{d\ge 1}\frac{d^{k-1}}{(k-1)!}q^{dn}\biggr)\,(1-q)^{k}
\,\le\,
\frac{k+1}{n^{k}} \qquad (k,n\geq 1,  q\in(0,1)).
\]

%%%%%
We have
\begin{align}
\sum_{d\ge 1}{d^{k-1}}\,q^{dn}
&\=
\biggl(\log\frac{1}{q}\biggr)\,
\sum_{d\ge 1}{d^{k-1}}\int_{dn}^{\infty}q^{u}\,du\\
&\=
\biggl(\log\frac{1}{q}\biggr)\,
\int_{n}^{\infty}
\biggl(\sum_{1\le d\le u/n}{d^{k-1}}\biggr)\,
q^{u}\,du, 
\end{align}
where interchanging the sum and integral is allowed since the sum $\sum_{d\ge 1}{d^{k-1}}q^{dn}$ converges (absolutely) for $q\in (0,1)$. 
For $u\ge n$, we now have
\begin{align}
\sum_{1\le d\le u/n}{d^{k-1}}
\,&=
\sum_{1\le d\le\lfloor u/n\rfloor-1}\!\!{d^{k-1}}
\+
\biggl(\frac{u}{n}\biggr)^{k-1}\\
&\le
\sum_{1\le d\le\lfloor u/n\rfloor-1}\int_{d}^{d+1}{y^{k-1}}\,dy
\+
\biggl(\frac{u}{n}\biggr)^{k-1}\\
&\le\,
\int_{1}^{u/n}{y^{k-1}}\,dy
\+
\biggl(\frac{u}{n}\biggr)^{k-1}\\
&\le\,
%\frac{1}{k}\biggl(\frac{u}{n}\biggr)^{k}
%\+
%\biggl(\frac{u}{n}\biggr)^{k-1}\\
%&\le\,
\frac{k+1}{k}\biggl(\frac{u}{n}\biggr)^{k}.
\end{align}
We thus have
\begin{equation}
\label{prop:HenrikQ:prelim}
\begin{aligned}
\sum_{d\ge 1}\frac{d^{k-1}}{(k-1)!}q^{dn}
\,&\le\,
\frac{k+1}{k!}\frac{1}{n^{k}}
\biggl(\log\frac{1}{q}\biggr)
\int_{n}^{\infty}
q^{u}\,u^{k}\,du\\
&=\,
\frac{k+1}{k!}\frac{1}{n^{k}}
\biggl(\log\frac{1}{q}\biggr)^{-k}
\int_{n\log\frac{1}{q}}^{\infty}
e^{-u}\,u^{k}\,du\\
&\le\,
\frac{k+1}{k!}\frac{1}{n^{k}}
\biggl(\log\frac{1}{q}\biggr)^{-k}
\int_{0}^{\infty}
e^{-u}\,u^{k}\,du\\
&=\,
\frac{k+1}{n^{k}}
\biggl(\log\frac{1}{q}\biggr)^{-k}.
\end{aligned}
\end{equation}
By using the bound
\begin{equation}
\label{prop:HenrikQ:log}
\log\frac{1}{q}
\=
\int_{q}^{1}\frac{du}{u}
\,\ge\,
1-q,
\end{equation}
we get
\[
\biggl(\sum_{d\ge 1}\frac{d^{k-1}}{(k-1)!}q^{dn}\biggr)(1-q)^{k}
\,\le\,
\frac{k+1}{n^{k}}\,
\biggl(\frac{1-q}{\log\frac{1}{q}}\biggr)^{k}
\,\le\,
\frac{k+1}{n^{k}}. \qedhere
\]
%This completes the proof.
\end{proof}
\begin{remark}\label{rem:inequcounterex} The stronger bound
\begin{align}\label{eq:failing}
\frac{E_{k}(q^{n})}{[n]_{q}^{k}}
\,\le\,
\frac{1}{n^{k}}
\end{align}
fails for $k$ divisible by $4$, large $n$ and $q$ close to $1$. Namely, setting $e^{-t}=q^n$, one finds that this inequality is equivalent to
\[\sum_{d\geq 1} \frac{d^{k-1}}{(k-1)!}\,e^{-dt}\,<\,\frac{1}{(n(1-e^{-t/n}))^k}.\]
Note, $(n(1-e^{-t/n}))$ is bounded from below by $t$ and converges to $t$ as $n\to \infty$. As $t\downarrow 0$, one would deduce from this inequality that $\zeta(1-k) \,<\, 0$, which fails if $4\mid k$. Concretely, $k=4, q=e^{-10^{-4}}$ and $n=10^4$ provides a counterexample to~\eqref{eq:failing}. 
\end{remark}

%%%%%%%%%%%%%%%

In order to determine the degree of a polynomial function on partitions, we start with two estimates on this degree. Together, these statements are an improvement of the statement that the degree is bounded by the weight. In fact, in most cases, the degree is strictly smaller than the weight, so that the weight limit, defined below, vanishes in most cases.
\begin{definition}\label{def:wtlim}
For $f\in \MS$ of weight $\leq k$ which is not of weight $\leq k-1$, the \emph{weight limit}~$Z(f)$ is given by
\begin{align}
    Z(f) \defis \lim_{q\to 1}(1-q)^{k}\langle f \rangle_q.
\end{align}
\end{definition}
Notice that $Z$ is not a linear map.    

\begin{lemma}\label{lem:liminq} For all $\vec{k}$ and $\vec{d}$ one has
\begin{enumerate}[{\upshape (i)}]\itemsep5pt
\item\label{it:lim1} $\displaystyle \degree\ms{k_1,\ldots,k_r}{d_1,\ldots,d_r} \,\leq\, \sum_i \max(k_i,d_i+1);$
\item\label{it:lim2}  $\displaystyle \degree\ms{k_1,\ldots,k_{i+1},\ldots,k_r}{d_1,\ldots,d_{i+1}+1,\ldots,d_r} \,\leq\, \degree\ms{k_1,\ldots,k_i,\ldots,k_r}{d_1,\ldots,d_i+1,\ldots,d_r};$
%\item  $\displaystyle \degree\ms[\mathrm{s}]{k_1,\ldots,k_r}{d_1,\ldots,d_r} \leq |\vec{k}| + \max(0,d_1-k_1+1+\max(0,d_2-k_2+1+\ldots))$
\end{enumerate}\vspace{5pt}
Moreover, if there is no index~$t$ such that $k_i=1$ for all $i\leq t$ and $d_i=0$ for all $i>t$, one has
%\begin{enumerate}
%\item[{\upshape(iii)}]\label{it:lim3} 
\[ Z\ms{k_1, \dots , k_r}{d_1,\dots,d_r} =0.\]
%\end{enumerate}
\end{lemma}
\begin{proof}
Given $F\in \Q\llbracket q\rrbracket$, write $[q^n] F$ for the coefficient of $q^n$ in $F$. Forgetting the inequalities $m_1>m_2>\ldots>m_r\mspace{1mu}$, one obtains %for the coefficients of $q^n$ of the corresponding $q$-brackets
\[ [q^n]\,\Bigl\langle \ms{k_1,\ldots,k_r}{d_1,\ldots d_r}\Bigr\rangle_{\! q} \leq [q^n]\,\Bigl\langle \ms{k_1}{d_1}\Bigr\rangle_{\! q} \cdots \Bigl\langle \ms{k_r}{d_r}\Bigr\rangle_{\! q}\]
for all $n\in\N$, from which the first part of the statement follows using Lemma~\ref{lem:lim1}. The second statement follows by estimating $m_{i+1}^{d_{i+1}+1}\leq m_{i+1}^{d_{i+1}}m_{i}\mspace{1mu}$. 
Finally, if such a $t$ does not exists, then the first two properties imply that the degree of the function at hand is strictly less than the weight, so that the weight limit vanishes.
\end{proof}

We now prove Theorem~\ref{thm:main}, which determines the degree of a polynomial function
\begin{align}\label{eq:ms} f=\ms{k_1,\ldots,k_r}{d_1,\ldots,d_r}.\end{align}
That is,
\begin{equation}\label{eq:deg2} \mathrm{deg}(f) \= \max_{j\in \{0,\ldots,r\}}\biggl\{\sum_{i\leq j}(d_i+1)\+\sum_{i> j}k_i\biggr\}.\end{equation}
Recall that if $r=1$, by Lemma~\ref{lem:lim1} the degree limit diverges if $k_1=d_1+1$. In general, %we will see later that 
this is the case when $k_a+\ldots+k_b = (d_a+1)+\ldots+(d_b+1)$ for some indices $a\leq b$. If such indices do not exist, the maximum in \eqref{eq:deg2} is unique and  we show
%the following proposition determines the degree, and shows that 
the corresponding degree limit is a sum of multiple zeta values. %Then, in the next section, we prove Theorem~\ref{thm:main} by extending the following proposition to the case where the degree limit diverges (i.e., the maximum in $v(f)$ below is not unique).  
% \begin{proposition}\label{prop:admissiblecase} Let $f$ as above~\eqref{eq:ms}. 
% If the maximum in
% \begin{equation}\label{eq:deg} v(f)\defis \max_{j\in \{1,\ldots,n+1\}}\sum_{i\leq j}(d_i+1)+\sum_{i> j}k_i\end{equation}
% is unique, then $\degree(f)=v(f)$ and 
% $\Zdegree(f) \in \mzv_{\leq \degree(f)}.$
% \end{proposition}
% In fact, this limit can be written as  in Lemma~\ref{lem:estimatezeta} below.
For this we first show that multiple zeta values are also well-defined for certain indices with negative integers (cf., \cite[Theorem~3]{M}).
\begin{lemma}\label{lem:estimatezeta}
For $r\geq 1$ and 
$\kappa_1,\ldots,\kappa_r \in \z$ with $\sum_{i=1}^j \kappa_i>j$ for all $j\geq 1$, \[ \zeta(\kappa_1,\ldots,\kappa_r) \defis \sum_{m_1>\cdots>m_r>0 }\frac{1}{m_1^{\kappa_1}\cdots m_r^{\kappa_r}}\]
is a well-defined real number. Moreover, it is an element of $\mzv_{\leq |\kappa|}$.
\end{lemma}
\begin{proof}
First we estimate $m_r^{-\kappa_r}$ by $m_{r-1}^{-\kappa_r+1}m_r^{-1}$ if $\kappa_r\leq 0$ and obtain
\begin{align}0\,<\,\zeta(\kappa_1,\ldots,\kappa_r) \,\leq\, \zeta(\kappa_1,\kappa_2,\ldots,\kappa_{r-1}+\kappa_r-1,1)\end{align}
if $\kappa_r\leq 0$. Iteratively (at most $r-1$ times) performing this procedure by estimating $m_i^{-\kappa_i'}$ by $m_{i-1}^{-\kappa_i'+1}m_i^{-1}$ for the largest index~$i$ for which the exponent~$\kappa_i'$ of $m_i$ is non-positive, we find
\begin{align}\label{eq:estimatezeta}0<\zeta(\kappa_1,\ldots,\kappa_r)\, \leq \,\zeta(\kappa_1+\kappa_2+\ldots+\kappa_j-j+1,\ldots) \, <\,\infty\end{align}
for some $j\geq 1$, 
where by the nature of this iterative process only the first entry may be non-positive. By our assumption $\sum_{i=1}^j \kappa_i>j$, we find that the first entry is at least equal to $2$, from which we indeed conclude that $\zeta(\kappa_1,\ldots,\kappa_r)$ is a well-defined real number. 

By iteratively using that $\sum_{m=a}^b m^{-k}$ for $k\leq 0$ is a polynomial in $a$ and $b$, it is not hard to show that $\zeta(\kappa_1,\ldots,\kappa_r)\in \mzv_{\leq |\vec{\kappa}|},$ with $|\vec{\kappa}|=\sum_{i}\kappa_i\mspace{1mu}$. Observe, here, again we make use of the assumption $\sum_{i=1}^j \kappa_j>j$ to ensure all multiple zeta values in this linear combination are convergent. 
\end{proof}

\begin{proof}[Proof of Theorem~\ref{thm:main}] 
Let $j=t$ be the smallest index for which the maximum in~\eqref{eq:deg2} is attained. Write 
$g =\mssmall{k_1,\ldots,k_t}{d_1,\ldots,d_t}$ and $h=\mssmall{k_{t+1},\ldots,k_r}{d_{t+1},\ldots,d_r}.$
We will show that $\Zdegree g$ and $\Zdegree h$ are multiple zeta values in case $j=t$ is the unique index for which the maximum in~\eqref{eq:deg2} is attained. In that case, we also show that $\Zdegree f = (\Zdegree g )( \Zdegree h).$ 

Since the degree limit only depends on the leading term we can choose any well-normalized family~$\mathcal{F}$. 
Observe that if we choose the Bernoulli--Seki model we have
\begin{align}\label{eq:h}\langle h\rangle_q %&=\biggl\langle\ms{k_{t+1},\ldots,k_r}{d_t,\ldots,d_r}\biggr\rangle_{\!q}
% &= \sum_{\substack{m_{t+1}>\ldots>m_n>0 \\ r_{t+1},\ldots,r_n>0}} \prod_{i} \frac{r_i^{k_i-1}}{(k_i-1)!} q^{m_i r_i} 
&\= \sum_{m_{t+1}>\ldots>m_r>0}  \prod_{i} m_i^{d_i} \frac{E_{k_i}(q^{m_i})}{(1-q^{m_i})^{k_i}}.
\end{align}
By Lemma \ref{lem:ekbound} we get for all $i$ 
\[\lim_{q\to 1} m_i^{d_i} (1-q)^{k_i}\frac{E_{k_i}(q^{m_i})}{(1-q^{m_i})^{k_i}} \= m_i^{d_i-k_i}.\]
% \jw{OLD: Observe that if we choose the Bernoulli--Seki model we get
% \begin{align}\langle h\rangle_q %&=\biggl\langle\ms{k_{t+1},\ldots,k_r}{d_t,\ldots,d_r}\biggr\rangle_{\!q}
% % &= \sum_{\substack{m_{t+1}>\ldots>m_n>0 \\ r_{t+1},\ldots,r_n>0}} \prod_{i} \frac{r_i^{k_i-1}}{(k_i-1)!} q^{m_i r_i} 
% &\= \sum_{m_{t+1}>\ldots>m_r>0}  \prod_{i} m_i^{d_i} \frac{E_{k_i}(q^{m_i})}{(1-q^{m_i})^{k_i}}, 
% \end{align}
% where $E_k$ is the $k$-th Eulerian polynomial ($k\geq 1$), defined by
% \[ \sum_{n\geq 1} \frac{n^{k-1}}{(k-1)!}x^n \= \frac{E_k(x)}{(1-x)^k}.\]
% The Eulerian polynomials satisfy $\lim_{x\to 1}E_k(x)=1.$ Hence, for all $i$ we have
% \[\lim_{q\to 1} m_i^{d_i} (1-q)^{k_i}\frac{E_{k_i}(q^{m_i})}{(1-q^{m_i})^{k_i}} \= m_i^{d_i-k_i}.\]}
First, assume $j=t$ is the unique index for which the maximum in~\eqref{eq:deg2} is attained. Then, we know that for all $t'> t$
\begin{align}\label{eq:ineqkd}\sum_{i=t+1}^{t'} k_i >\sum_{i=t+1}^{t'} (d_i+1).\end{align}
By Proposition~\ref{prop:review} we find that the sum in~\eqref{eq:h} is absolutely bounded by a constant (depending on the $k_i$) times  $\zeta(k_{t+1}-d_{t+1},\ldots,k_r-d_r)$. Using~\eqref{eq:ineqkd}, this multiple zeta value  is well-defined by Lemma~\ref{lem:estimatezeta}, 
so~\eqref{eq:h} converges absolutely. By interchanging the sum and limit, we conclude
\[\degree h = \sum_{i=t+1}^r k_i \quad \text{and} \quad \Zdegree h = \zeta(k_{t+1}-d_{t+1},\ldots,k_r-d_r). \]
%where the multiple zeta value on the right is well-defined by Lemma~\ref{lem:estimatezeta}. 
Hence, $\Zdegree h$ is a linear combination of multiple zeta values of weight $\leq \sum_{i=t+1}^r k_i\mspace{1mu}$. 

In case the index $j=t$ is not unique, let $\epsilon>0$. Then, as 
\[ \lim_{q\to 1} m_i^{d_i} (1-q)^{k_i+\epsilon} \frac{E_{k_i}(q^{m_i})}{(1-q^{m_i})^{k_i}} \:\to\: 0\]
if $q\to 1$, we see that $\deg h \leq \sum_{i=t+1}^r k_i$ in this case. Moreover, as
\[ \lim_{q\to 1} m_i^{d_i} (1-q)^{k_i} \frac{E_{k_i}(q^{m_i})}{(1-q^{m_i})^{k_i}} \= m_i^{d_i-k_i} \:>\:0,\]
we conclude that also in this case we have $\deg h = \sum_{i=t+1}^r k_i$.

Next, (again working with the Bernoulli--Seki model) observe that
\begin{align}
\langle g \rangle_q
%&=\biggl\langle\ms{k_1,\ldots,k_t}{d_1,\ldots,d_t}\biggr\rangle_{\!q}
&\= \sum_{\substack{m_{1}>\ldots>m_t>0 \\ r_{1},\ldots,r_t>0}} \prod_{i} m_i^{d_i}r_i^{k_i-1} q^{m_i r_i} \\
&\= \sum_{\substack{0<r_{1}'<\ldots<r_t' \\ m_{1}',\ldots,m_t'>0}} \prod_{i} (m_i'+\ldots+m_t')^{d_i} (r_i'-r_{i-1}')^{k_i-1} q^{m_i' r_i'},
\end{align}
where we set $r_0'=0$.
Given a `monomial' $m_1^{d_1'}\cdots m_t^{d_t'}$, we have
\begin{align}
\sum_{\substack{0<r_{1}<\ldots<r_t \\ m_{1},\ldots,m_t>0}} \prod_{i} m_i^{d_i'} q^{m_i r_i} 
&\= \sum_{0<r_{1}<\ldots<r_t} \prod_{i} d_i'! \frac{E_{d_i'+1}(q^{r_i})}{(1-q^{r_i})^{d_i'+1}}, 
\end{align}
which is of degree $\sum_{i} (d_i'+1)$ and (in case $d_t'>0$) with degree limit $\zeta(d_t'+1,\ldots,d_1'+1)$ (as follows from above). Hence, by expanding $\prod_{i}(m_i'+\ldots+m_t')^{d_i}$ in monomials, we see that \[%\degree\ms{k_1,\ldots,k_t}{d_1,\ldots,d_t}
\degree g\=\sum_{i}(d_i+1),\]
provided that the corresponding limit
\[%\Zdegree \ms{k_1,\ldots,k_t}{d_1,\ldots,d_t} 
L\=  \sum_{0<r_1<\ldots<r_t} \prod_i \frac{(r_i-r_{i-1})^{k_i-1}}{r_i}\cdot\Omega\biggl[\prod_{i=1}^t \Bigl(\frac{1}{r_i}+\ldots+\frac{1}{r_t}\Bigr)^{d_i}\biggr], \]
converges, where $\Omega:\Q[r_1^{-1},\ldots,r_t^{-1}]\to \Q[r_1^{-1},\ldots,r_t^{-1}]$ is the linear mapping~\eqref{def:omega}, i.e.,
\[\Omega\Bigl[\frac{1}{r_1^{l_1}\cdots r_t^{l_t}}\Bigr] \defis \frac{l_1!\cdots l_r!}{r_1^{l_1}\cdots r_t^{l_t}} \qquad \qquad (l_i\in \z_{\geq 0}).\]
Note that $\langle g\rangle_q$ is bounded be a constant (only depending on the $d_i'$) times $L$, which justifies interchanging sum and limit in this computation (if $L$ converges). 
We estimate
\[ r_i-r_{i-1}\leq r_i \qquad \text{and} \qquad \frac{1}{r_i}+\ldots+\frac{1}{r_t}\leq \frac{t}{r_i}\]
to obtain
%Note that, by definition of~$t$, we have
\[ L \:\leq\: \sum_{0<r_1<\ldots<r_t} \prod_i r_i^{k_i-d_i}  t^{d_i} d_i! \= \biggl(\prod_{i} t^{d_i}d_i!\biggr)\cdot \zeta(d_t-k_t,\ldots,d_1-k_1).\]
Note that if for some $s$ we have that $\sum_{i=s}^t (d_i-k_i+1)\leq 0$, then the sum~\eqref{eq:deg2} is also maximized for $j=s-1$, contradicting the definition of~$t$. Hence, Lemma~\ref{lem:estimatezeta} implies that $L$ is a well-defined real number.
Therefore, $L=\Zdegree g$, and $\Zdegree g$ is a well-defined linear combination of multiple zeta values of weight $\leq \sum_{i=1}^t (d_i+1)$ (the latter one reads of from the definition of $L$).%, due to the definition of $t$.

% and, if $t$ is unique,
% \[%\Zdegree \ms{k_1,\ldots,k_t}{d_1,\ldots,d_t} 
% \Zdegree g\=  \sum_{0<r_1<\ldots<r_t} \prod_i \frac{(r_i-r_{i-1})^{k_i-1}}{r_i}\cdot\Omega\biggl[\prod_{i=1}^t \Bigl(\frac{1}{r_i}+\ldots+\frac{1}{r_t}\Bigr)^{d_i}\biggr], \]
% where $\Omega:\Q[r_1^{-1},\ldots,r_t^{-1}]\to \Q[r_1^{-1},\ldots,r_t^{-1}]$ is the linear mapping~\eqref{def:omega}, i.e.,
% \[\Omega\Bigl[\frac{1}{r_1^{l_1}\cdots r_t^{l_t}}\Bigr] \defis \frac{l_1!\cdots l_r!}{r_1^{l_1}\cdots r_t^{l_t}} \qquad \qquad (l_i\in \z_{\geq 0}).\]
% This is a linear combination of multiple zeta values of weight $\leq \sum_{i=1}^t (d_i+1)$.

Finally, we relate the degree and limits of~$g$ and~$h$ to the degree and limits of~$f$. Consider the difference $f-g\ost h$, and observe that $\degree (g\ost h) = \degree g + \degree h$ (as $\langle g\ost h\rangle_q = \langle g\rangle_q \langle h\rangle_q$; see Proposition~\ref{prop:doubleshuffle}).
%\begin{align}\label{eq:concfunc}\ms[\mathrm{s}]{\hspace{7pt}1,\ldots,1,k_{t+1},\ldots,k_r}{d_1,\ldots,d_{t},\hspace{5pt}0,\ldots,0\hspace{11pt}} \end{align}
%and
%\begin{align}\label{eq:prodfunc}\ms[\mathrm{s}]{1, \dots , 1}{d_1,\dots,d_t} \ost \ms[\mathrm{s}]{k_{t+1}, \dots , k_r}{0,\dots,0}.\end{align}
%Writing this difference in terms of the basis elements of $\MS$, it contains three types of terms. We claim that all these terms are of degree smaller than the degree~\eqref{eq:deg} in the statement. 
First we assume again that $j=t$ is the unique index for which the maximum in~\eqref{eq:deg2} is attained. In that case, we claim that $f-g\ost h$ is of smaller degree than $g\ost h$. 
It follows from this claim that the function~$f$ has the same degree, as well as degree limit, as $g\ost h$. Hence, in case $j=t$ is the unique index for which the maximum in~\eqref{eq:deg2} is attained, it suffices to prove the claim. 

In order to prove the claim, we write this difference $f-g\ost h$ in terms of the basis elements of $\MS$. Then, we will see it contains two types of terms. We show that all these terms are of degree smaller than the degree~\eqref{eq:deg2} in the statement. 

% First of all, we consider terms in $f-g\ost h$ of lower weight than the weight of $f$. %Note that the degree~\eqref{eq:deg} equals the weight of $f$.
% By Lemma~\ref{lem:liminq}\ref{it:lim1} the degree is bounded by the weight, so the degree of these terms will be smaller than the weight~\eqref{eq:degHJ}. 

%\jw{I have rewritten the next to paragraphs for clarity}
First of all, there are terms in $f-g\ost h$ where the column with $k_t$ and $d_t$ is on the right of the column with $k_{t+1}$ and $d_{t+1}$ (note that possibly some other column is stuffled on the columns containing these values). We apply Lemma~\ref{lem:liminq}\ref{it:lim2} repeatedly, so that:
\begin{itemize}
    \item The values~$d_j$ are replaced by $d_j'$ such that $k_j\leq d_{j}'+1$ if $j\leq t$ and $k_t< d_{t}'+1$, and with $k_j\geq d_{j}'+1$ if $j>t$ and $k_{t+1}>d_{t+1}'+1$. That this is always possible follows from the construction of $t$. For the strict inequalities we use the uniqueness of the maximum defining $t$.  
    \item Additionally, the values $d_t'$ and $d_{t+1}'$ are replaced by $d_t'-1$ and~$d_{t+1}'+1$ respectively. 
\end{itemize}
Lemma~\ref{lem:liminq}\ref{it:lim1} now implies that these terms are of lower degree. 

Secondly, there are terms of depth $<r$, for which a certain column is given by $\binom{k_t+k_{t+1}-a}{d_t+d_{t+1}}$ for some $a\geq 0$. Then, similarly, Lemma~\ref{lem:liminq} also implies that these terms are of lower degree. 

Now, the case that $j=t$ is not the unique index for which the maximum in~\eqref{eq:deg2} is attained remains. By a similar argument as in Lemma~\ref{lem:liminq}\ref{it:lim1}, we have
\[ \deg f \leq \deg g  +  \deg h.\]
Namely, for this inequality we forget all inequalities $m_i>m_j$ if $i\leq t$ and $j>t$. 

%Next, observe that in case $f=h$, we are done. Hence, $t>0$. %Moreover, by definition of $t$ we have $d_t\neq 0$. 
Let $j=t'$ be the \emph{maximal} index for which the maximum in~\eqref{eq:deg2} is obtained, and write
\[ f'=\ms{k_1,\ldots,k_{t'-1}, \hspace{5pt} k_{t'}, \hspace{5pt} k_{t'+1},\ldots,k_r}{0,\ldots,0,d_1+\ldots+d_{t'},d_{t'+1},\ldots,d_r}.\]
Now, we apply Lemma~\ref{lem:liminq}\ref{it:lim2} for $i=1,\ldots,t'-1$, to obtain a lower bound for the degree of~$f$. That is, %we obtain
\begin{align} \deg f=\degree\ms{k_1,\ldots,k_r}{d_1,\ldots,d_r} &\:\geq\:  \degree f'
%\ms{k_1,\ldots,k_{t'-1}, \hspace{15pt} k_{t'}, \hspace{15pt} k_{t'+1},\ldots,k_r}{0,\ldots,0,d_1+\ldots+d_{t'},d_{t'+1},\ldots,d_r} \\
\= \sum_{j=1}^{t'} (d_j+1) + \sum_{j=t'+1}^r k_{j}.
\end{align}
Here, the last equality holds as for $f'$ the maximum in~\eqref{eq:deg2} is uniquely attained at $j=t'$, and we already computed the degree and corresponding limit in that case. Hence, $\deg f$ is bounded from both sides by the value \eqref{eq:deg2}, which completes the proof in the case  that $j=t$ is not the unique index for which the maximum in~\eqref{eq:deg2} is attained.
%
% Hence, 
% \[\Zdegree\ms[\mathrm{s}]{\hspace{7pt}1,\ldots,1,k_{t+1},\ldots,k_r}{d_1,\ldots,d_{t},\hspace{5pt}0,\ldots,0\hspace{11pt}} = \Zdegree\ms[\mathrm{s}]{1, \dots , 1}{d_1,\dots,d_t}\cdot\Zdegree\ms[\mathrm{s}]{k_{t+1}, \dots , k_r}{0,\dots,0} \]
\end{proof}

Recall that by Lemma~\ref{lem:liminq} we have
\[ Z\ms{k_1, \dots , k_r}{d_1,\dots,d_r} =0\]
 if there is no index~$t$ such that $k_i=1$ for all $i\leq t$ and $d_i=0$ for all $i>t$.
%\begin{enumerate}
%\item[{\upshape(iii)}]\label{it:lim3} 
By going through the proof above for those basis elements in $\MS$ for which there does exist such a~$t$, we find that the weight limits of such elements does not vanish:
\begin{corollary}\label{prop:limitfactors} For all $d_1,\ldots,d_{t-1}\geq 0,  d_t\geq 1, k_{t+1}\geq 2$ and $k_{t+2},\ldots,k_{r}\geq 1$, write 
\[ f \= \ms{\hspace{7pt}1,\ldots,1,k_{t+1},\ldots,k_r}{d_1,\ldots,d_{t},\hspace{5pt}0,\ldots,0\hspace{11pt}}.\]
We have
\begin{align}\label{eq:degHJ} \degree f \= \sum_{i=1}^t (d_i+1)\+\sum_{i=t+1}^r k_i \= \wt f\end{align}
and 
\[ \Zdegree f \= \xi(d_1,\dots,d_t) \, \zeta(k_{t+1}, \dots , k_r) \= Z f.\]
\end{corollary}

\subsection{Algebraic setup}\label{sec:algsetup} We will now introduce the algebraic setup for our space $\MS$ by generalizing the classical setup introduced by Hoffman in \cite{H} and used in \cite{IKZ} for the regularization of multiple zeta values. For each model $\FF=\{f_k\}_{k=0}^\infty$ of well-normalized polynomials we will describe the harmonic product as well as the shuffle product on the level of words in analogy to \cite{H}. For the Bernoulli--Seki model the algebraic setup described here coincides with the one described in \cite[Section 3]{B}.
Define the set $A$, also called the set of \emph{letters}, by
\begin{align*}
	A \defis \left\{ \ai{k}{d} \,\Bigl\vert\, k \geq 1,\,d \geq 0 \right\} \,.
\end{align*}
We will define a product $\qsf$ on the space~$\QA$ of non-commuting polynomials in~$A$ and we will call a monic monomial in $\QA$ a \emph{word}. This product on the space of words will depend on a product $\df$ on the space $\Q A$ of letters, which depends on 
%a well-normalized family of polynomials
$\FF$. Recall that as $\FF$ is well-normalized, the $f_k$ are polynomials of degree $k$ with leading coefficient $\frac{1}{k!}$. Then, for $k_1,k_2 \geq 1$ and $1 \leq j \leq k_1+k_2-1$ there exist rational numbers $\alpha_\FF(k_1,k_2,j) \in \q$ such that for all $n \geq 1$
\begin{align}\label{eq:defalpha}
\sum_{\substack{n_1+n_2=n\\n_1,n_2\geq 1}} \partial f_{k_1}(n_1)\, \partial f_{k_2}(n_2) \=   \partial f_{k_1+k_2}(n) \+\sum_{j=1}^{k_1+k_2-1} \alpha_\FF(k_1,k_2,j)\,  \partial f_{j}(n)\,,
\end{align}
where as before $\partial f(n)=f(n)-f(n-1)$ for integers $n\geq 1$.

%\filbreak %enable to prevent a page break between Example and (i).
\begin{example}\mbox{}\\[-12pt]
\begin{enumerate}[{\upshape(i)},leftmargin=*]\itemsep5pt
    \item If $\FF=\mathrm{b}$ is the binomial model, i.e., $f_k(x)=\binom{x}{k}$, we have $\partial f_k(x) = \binom{x-1}{k-1}$ for $k\geq 1$ and since 
\begin{align*}
    \sum_{\substack{n_1+n_2=n\\n_1,n_2\geq 1}} \binom{n_1-1}{k_1-1} \binom{n_2-1}{k_2-1} \= \binom{n-1}{k_1+k_2-1}
\end{align*}
we obtain $\alpha_\mathrm{b}(k_1,k_2,j)=0$. 
\item For the Bernoulli--Seki model $\FF=\mathrm{s}$ we find that~$\alpha_\mathrm{s}(k_1,k_2,j)$ equals
%{\small
\begin{align}\label{eq:deflambda}
-\left(\!(-1)^{k_1} \binom{k_1+k_2-1-j}{k_2 -j} + (-1)^{k_2} \binom{k_1+k_2-1-j}{k_1-j} \!\right) \frac{B_{k_1+k_2-j}}{(k_1+k_2-j)!} \,,
\end{align}
%}
which can be proven by using Bernoulli's/Seki's/Faulhaber's formula for the sum of powers (see, e.g., \cite[Lemma~6.1.2(ii)]{vI}).
\end{enumerate} 
\end{example}

On $\Q A$ we define the product $\df$ for $k_1, k_2 \geq 1$ and $d_1,d_2 \geq 0$ by 
\begin{align}\label{eq:diamoneq}
	\ai{k_1}{d_1} \diamond_\FF \ai{k_2}{d_2} \= 	\ai{k_1+k_2}{d_1+d_2} +\sum_{j=1}^{k_1+k_2-1} \alpha_
	\FF(k_1,k_2,j) \ai{j}{d_1+d_2} \,.
\end{align}
In the case $\FF=\mathrm{b}$ we just write $\diamond = \diamond_{\mathrm{b}}\mspace{1mu}$, which is given by
\begin{align*}
	\ai{k_1}{d_1} \diamond \ai{k_2}{d_2} = 	\ai{k_1+k_2}{d_1+d_2}\,.
\end{align*}

For each well-normalized family of polynomials $\FF$ this gives a commutative non-unital $\Q$-algebra $(\Q A,  \df)$. We will be interested in $\Q$-linear combinations of words in the letters of $A$ and we define 
\begin{align*}
    \ws = \QA
\end{align*}
In the following we will use for $k_1,\dots,k_r\geq 1$, $d_1,\dots,d_r\geq 0$ the following notation to write words in $\ws$:
\begin{align*}
	\ai{k_1,\dots,k_r}{d_1,\dots,d_r} \defis   	\ai{k_1}{d_1}\dots \ai{k_1}{d_1}\,,
\end{align*}
where the product on the right is the usual non-commutative product in $\QA$.
\begin{definition}\label{def:quasishuffleprod} For a well-normalized family of polynomials $\FF$ we define the \emph{quasi-shuffle product}  $\qsf$ on $\ws$ as the $\Q$-bilinear product, which satisfies $1 \qsf w = w \qsf 1 = w$ for any word $w\in \ws$ and
	\begin{align}\label{eq:qshdef}
		a w \qsf b v = a (w \qsf b v) + b (a w \qsf v) + (a \df b) (w \qsh  v) 
	\end{align}
	for any letters $a,b \in A$ and words $w, v \in  \ws$. 
\end{definition}
We obtain a commutative $\Q$-algebras $(\ws,\qsf)$, as shown in \cite{H} (see also \cite{HI}). In the case $\FF=\mathrm{b}$ we just write $\qsh = \qsh_{\mathrm{b}}\mspace{1mu}$. As an example for the product~\eqref{eq:qshdef} we have
\begin{align}\label{eq:qashindep2}
	\ai{k_1}{d_1} \qsh \ai{k_2}{d_2} = \ai{k_1,k_2}{d_1,d_2} + \ai{k_2,k_1}{d_2,d_1} +	\ai{k_1+k_2}{d_1+d_2}\,.
\end{align}

The algebra $(\ws,\qsh)$ is graded by \emph{weight}, where the weight is defined by 
\begin{align*}
	\wt\left( \ai{k_1,\dots,k_r}{d_1,\dots,d_r} \right) = k_1+\dots+k_r+d_1+\dots+d_r\,
\end{align*}
and it is filtered by \emph{depth}, which is defined by 
\begin{align*}
	\dep\left( \ai{k_1,\dots,k_r}{d_1,\dots,d_r} \right) = r\,.
\end{align*}
In general the algebra $(\ws,\qsf)$ is not graded but filtered by weight.  

\begin{proposition}\label{prop:alghoms}For any well-normalized family of polynomials $\FF$ the  maps 
\begin{align*}
   \siso_\FF: (\ws,\qsf) &\longrightarrow (\MS, \ost) \\
    \ai{k_1,\dots,k_r}{d_1,\dots,d_r} &\longmapsto  \ms[\FF]{k_1, \dots , k_r}{d_1,\dots,d_r}
\end{align*}
and
\begin{align*}
    (\ws,\qsh) &\longrightarrow (\MS, \odot) \\
    \ai{k_1,\dots,k_r}{d_1,\dots,d_r} &\longmapsto  \ms[\mathrm{m}]{k_1, \dots , k_r}{d_1,\dots,d_r}
\end{align*}
are algebra-isomorphisms of filtered algebras. Recall that $\mathrm{m}$ stands for the monomial model (see Table~\ref{table:polmodels}). 
\end{proposition}
\begin{proof} Assume $\FF$ is a well-normalized family of polynomials.  First notice that whenever one has a sum of the form 
\begin{align}\label{eq:harmonicsum}
    \sum_{m_1 > \dots > m_r > 0} \prod_{j=1}^r m_j^{d_j}\, h(m_j,k_j)
\end{align}
in a $\q$-algebra $(A,\times)$ with some function $h$ satisfying 
\begin{align}\label{eq:propofh}
    h(m,k_1)\times h(m,k_2) \= h(m,k_1+k_2) \+\sum_{j=1}^{k_1+k_2-1} \alpha_\FF(k_1,k_2,j)\,  h(m,j)\,,
\end{align}
then the linear map $\varphi$ defined by sending the generators $ \ai{k_1,\dots,k_r}{d_1,\dots,d_r}$ to~\eqref{eq:harmonicsum} satisfies $\varphi(w) \times \varphi(v)=\varphi(w \qsf v)$ for all $w,v\in \ws$.
% , assuming that the sum~\eqref{eq:harmonicsum} exists in $A$ \jw{This sentence should be rewritten, especially I do not understand what `assuming that this sum exists' means.}. 
This can be shown by truncating the sum over the $m_j$ by some $M$ and then do induction on $M$ together with the definition of~$\qsf$ (see for example \cite[Lemma~2.18]{B2}). 
To show the first statement we consider the $\vec{u}$-bracket of a generator of $\MS$, which by Proposition~\ref{prop:ubracketofs} is given by
\begin{align}\label{eq:ubrackeofsf}
\left\langle \ms[\FF]{k_1, \dots , k_r}{d_1,\dots,d_r}\right\rangle_{\!\!\vec{u}} \=\sum_{m_1 > \dots > m_r > 0} \prod_{j=1}^r\Bigl(m_j^{d_j} \sum_{n>0}  \partial f_{k_j}(n)\, u_{m_j}^{n}\Bigr).
\end{align}
Since the harmonic product~$\ost$ is defined by the $\vec{u}$-bracket (see Definition~\ref{def:products}) we see that $\siso_\FF$ is an algebra homomorphism, as~\eqref{eq:ubrackeofsf} is a sum of the form~\eqref{eq:harmonicsum} with $h(m,k) = \sum_{n>0}  \partial f_{k}(n)\, u_{m}^{n}$ and the property~\eqref{eq:propofh} is satisfied due to~\eqref{eq:defalpha}.
With a similar argument we see that also the second map is an algebra homomorphism, since
\begin{align*}
\msl[\mathrm{m}]{k_1, \dots , k_r}{d_1,\dots,d_r}{\lambda}\= \sum_{m_1 > \dots > m_n > 0} \prod_{j=1}^{r} m_j^{d_j} r_{m_j}(\lambda)^{k_j} \,,
\end{align*}
is a sum of the form~\eqref{eq:harmonicsum} with $h(m,k)=r_{m}(\lambda)^{k}$. All maps are also isomorphisms among the elements in $\ws$ and in $\MS$ (Corollary~\ref{cor:basis}).
\end{proof}

As a direct consequence of describing the product $\ost$ in terms of a quasi-shuffle product is the following.

\begin{proposition}\label{prop:pkkkisquasimodular}
For $k\geq 2, d\geq0$ with $k+d$ even we have $\mssmall[\mathrm{s}]{k, \dots , k}{d,\dots,d} \in \MM$, i.e., their $q$-brackets are quasimodular forms (of mixed weight).
\end{proposition}
\begin{proof}
Suppose a $\Q$-algebra~$R$ and an algebra homomorphism $\varphi:  (\ws,\qsf) \rightarrow R$ are given.
Using \cite[(32)]{HI} (see also \cite[(2.13)]{B2}) we have for $a\in A$ %and any algebra homomorphism $\varphi:  (\ws,\qsf) \rightarrow R$ in some $\Q$-algebra~$R$ 
the following equation in $R\llbracket X \rrbracket$:
\begin{align}\label{eq:qshexp}
\exp\left( \sum_{n=1}^{\infty} (-1)^{n-1} \varphi( a^{\diamond_\FF n}) \frac{X^n}{n} \right) \= 1+ \sum_{n=1}^{\infty} \varphi(a^n) X^n,
\end{align}
where $a^{\diamond_\FF n} = \overbrace{a \diamond_\FF \dots \diamond_\FF a }^n \in \Q A$ and $a^n = \overbrace{a\cdots a}^n \in \ws$. 

Now, let $R=\q\llbracket q \rrbracket$ and let $\varphi$ be the composition of the algebra homomorphism $\siso_\FF$ in Proposition~\ref{prop:alghoms} with the $q$-bracket, so that we obtain an algebra homomorphism $\varphi: (\ws,\qsf) \rightarrow \Q\llbracket q \rrbracket$ (see Proposition~\ref{prop:polfctonpartprop}). We call a letter $\ai{k}{d}\in A$ even if ${k\geq 2}$, ${d\geq 0}$ and $k+d$ is even. %In the case $\FF=\mathrm{s}$ one then checks using~\eqref{eq:deflambda} and~\eqref{eq:diamoneq} that for $k_1,k_2\geq 2, d_1,d_2\geq 0$ and $k_1+d_1$, $k_2 + d_2$ even, the product $\ai{k_1}{d_1} \diamond_\mathrm{s} \ai{k_2}{d_2}$ is again a linear combination of some $\ai{k'}{d'}$ with $k'\geq 2, d'\geq 0$ and $k'+d'$ even 
In the case $\FF=\mathrm{s}$ one then checks using~\eqref{eq:deflambda} and~\eqref{eq:diamoneq} that for two even letters $a,a'\in A$, the product $a \diamond_\mathrm{s} a'$ is again a linear combination of some even letters
(the Bernoulli numbers $B_k$ for $k\geq 3$ odd vanish). 
This inductively gives that for an even letter $a$ the element $a^{\diamond_\FF n}$ is also a linear combination of even letters. Then~\eqref{eq:qshexp} implies that the $q$-bracket of $\mssmall[\mathrm{s}]{k, \dots , k}{d,\dots,d}$, which is $\varphi(a^n)$ for $a=\ai{k}{d}$, can be written as a polynomial in the $q$-brackets of $\mssmall[\mathrm{s}]{k'}{d'}$ with $k'+d'$ even. By~\eqref{eq:derivofeisensteinseries} for $\FF=\mathrm{s}$ these are exactly given by derivatives of Eisenstein series and therefore we obtain a quasimodular form (of mixed weight). 
\end{proof}

\begin{example} For example, in the case $k=2$, $d=0$ we get in depth two
\begin{align*}
\biggl\langle \ms[\mathrm{s}]{2,2}{0,0} \biggr\rangle_{\!\! q} \= \frac{1}{2}G_2^2 - \frac{1}{2} G_4 + \frac{17}{120} G_2 +  \frac{31}{5760}\,,
\end{align*}
where we used $G_k = -\frac{B_k}{2 k!} + \langle \mssmall[\mathrm{s}]{k}{0} \rangle_q $.
\end{example}

The above algebraic setup can be seen as a generalization of the algebra $\wh^1$ defined in \cite{H}, as we explain now. Setting $z_k := \ai{k}{0}$ for $k\geq1$ we define 
\begin{align*}
    \wh^1 = \Q + \langle z_{k_1}\dots z_{k_r} \mid r\geq 1 , k_1,\dots,k_r \geq 1\rangle_\Q \,\subset\, \ws
\end{align*}
and its subspace of \emph{admissible words}
\begin{align*}
    \wh^0 = \Q + \langle z_{k_1}\dots z_{k_r} \in  \wh^1 \mid r\geq 1 , k_1 \geq 2\rangle_\Q\,.
\end{align*}
The quasi-shuffle product on $\ws$ reduces to a product on $\wh^1$ and on $\wh^0$. Moreover, the product $\ast$ reduces to the the classical harmonic product introduced in \cite{H}. One can then define the following $\Q$-linear map\footnote{By abuse of notation, we use $\zeta$ for the name of the map as well as for the name of the object. From the context, it should always be clear if we are talking about the map or the object.}
\begin{align}\label{eq:defzetamap}
    \zeta: \wh^0 &\longrightarrow \mzv\\
    z_{k_1}\dots z_{k_r} &\longmapsto \zeta(k_1,\dots,k_r)\,.
\end{align}
By the definition of multiple zeta values~\eqref{eq:defmzv} as nested sum one then checks that this map is an algebra homomorphism from $(\wh^0, \qsh)$ to the algebra of multiple zeta values $\mzv$. Since $\wh^1 = \wh^0[ z_1]$ (\cite[Proposition 1]{IKZ}) one can extend this algebra homomorphism uniquely to an algebra homomorphism $\zeta^\ast: (\wh^1, \qsh) \rightarrow \mzv[T]$ which satisfies $\zeta^\ast(z_1)=T$ and $\zeta^\ast_{\mid \wh^0} = \zeta$. This gives the definition of \emph{(harmonic) regularized multiple zeta values} for $k_1,\dots,k_r\geq 1$. %by $\zeta^\ast(k_1,\dots,k_r) := \zeta^\ast(z_{k_1}\dots z_{k_r}) \in \mzv[T]$.

In the following we will generalize this result and the goal is to obtain an algebra homomorphism $(\ws, \qsh) \rightarrow \mzv[T]$. First we define the analogue of admissible words in our larger space $\ws$.

\filbreak
\begin{definition}\mbox{}\\[-12pt] 
\begin{enumerate}[{\upshape(i)},leftmargin=*]\itemsep3pt
\item For $d\geq 0$ set $v_d:=\ai{1}{d}$. As an counterpart of $\wh^1$ and $\wh^0$ we define %the following subspaces of $\ws$
\begin{align*}
\wj^1 &\= \Q + \langle v_{d_1}\dots v_{d_r} \mid r\geq 1 , d_1,\dots,d_r \geq 0\rangle_\Q \,\subset\, \ws \\
\wj^0 &\= \Q + \langle v_{d_1}\dots v_{d_r} \in  \wj^1 \mid r\geq 1 , d_r \geq 1\rangle_\Q\,.
\end{align*}
% \begin{align*}
%   \wj^1 &= \Q + \left\langle  \ai{1,\dots,1}{d_1,\dots,d_r}\mid r\geq 1, d_1,\dots,d_r\geq 0  \right\rangle_\Q\,,\\   \wj^0 &= \Q + \left\langle  \ai{1,\dots,1}{d_1,\dots,d_r}\mid r\geq 1, d_1,\dots,d_{r-1}\geq 0, d_r\geq 1  \right\rangle_\Q \subset \wj^1\,.
% \end{align*}
    \item We set $\ws^1 := \wj^1 \wh^1$, which is the space spanned by all words of the form
	\begin{align}\label{eq:wordins1}
	w = \ai{1,\dots,1}{d_1,\dots,d_m}  \ai{1}{0}^{\! j} \ai{k_1,\dots,k_r}{0,\dots,0}
	\end{align}
	for some $j \geq 0$, $m,r \geq 0$, $d_m\geq 1$ (if $m\geq 1$), and $k_1\geq 2$ (if $r\geq 1$).
	\item 	A word $w\in \ws^1$ of the form~\eqref{eq:wordins1} is called \emph{admissible} if $j=0$. The subspace of admissible words will be denoted by 
\begin{align*}
    \ws^0 \defis \wj^0 \wh^0 = \langle w \in \ws^1 \mid w \text{ is admissible} \rangle_\Q\,.
\end{align*}
\end{enumerate}  

\end{definition}
Notice that this notion of admissibility coincides with the classical notion of admissibility on the subspace $\wh^1$, namely we have $\wh^0 = \wh^1 \cap \ws^0$. However, in contrast to the classical setup, the spaces $\ws^1$ and $\ws^0$ are not closed under $\qsh$ (or any $\qsf$) and therefore they are not subalgebras of $\ws$. 

\begin{lemma}\label{eq:nisideal} Let $\wsn$ be the subspace of $\ws$ spanned by all words in $\ws \setminus \ws^1$. Then $\wsn$ is an ideal in $(\ws, \qsh)$.
\end{lemma}
\begin{proof}
We need to show that for any $w \in \ws$ which is not of the form~\eqref{eq:wordins1} the product $w \qsh v$ is in $\wsn$ for any $v \in \ws$. If $w \in \ws$  is not of the form~\eqref{eq:wordins1} then $w$ either contains a letter $\ai{k}{d}$ or it contains a letter~$\ai{k}{0}$ which is on the left of a letter~$\ai{1}{d}$ with $k\geq 2$ and $d\geq 1$ in both cases. In the first case the product of $w$ with any element~$v$ will be a sum of words which also contain either a letter~$\ai{k}{d}$ or a $\ai{k}{d} \diamond b$ with some other letter $b$ of $v$. Since $\diamond$ is the component-wise addition we see that each word therefore also contains a letter with top entry $\geq 2$ and bottom entry $\geq 1$. In the second case all words either also have the a letter $\ai{k}{0}$ which is on the left of a letter $\ai{1}{d}$, or a a letter $\ai{k}{0} \diamond b$ on the left of a letter $\ai{1}{d} \diamond c$, where $b,c$ are letters of $v$, etc.. In all cases the words are not of the form~\eqref{eq:wordins1} and are therefore elements in $\wsn$. 
\end{proof}

As a generalization of $\wh^1 = \wh^0[ z_1]$, we will now show that also any element in~$\ws^1$ can be written as a polynomial with respect to the quasi-shuffle product~$\qsh$ with coefficients in $\ws^0$ up to elements in the ideal $\wsn$. We write
\[\ai{1}{0}^{\qsh j} = \underbrace{ \ai{1}{0} \qsh \cdots \qsh \ai{1}{0}}_j\,.\]
\begin{proposition} \label{prop:reg} For any word $w \in \ws^1$ of length $r$ there exist unique $w_j \in \ws^0$ such that
	\begin{align}\label{eq:wasregpol}
	w \:\equiv\: \sum_{j=0}^r w_j \qsh \ai{1}{0}^{\qsh j} \mod \wsn.
	\end{align}
	Moreover, if $\wt(w)=k$ then the $w_j$ are linear combinations of words of weight $k-j$.
\end{proposition}
\begin{proof}
%We just need to show the statement for  words of the form 
A word in $w\in\ws^1$ can be written as
\begin{align*}
w =  u \ai{1}{0}^j v \qquad\qquad (j\geq 0, u \in \wj^0, v \in \wh^0).
\end{align*}
Write $u = \ai{1,\dots,1}{d_1,\dots,d_m} $ and $v=\ai{k_1,\dots,k_r}{0,\dots,0}$ with $d_m>0$ and $k_1>1$. % i.e. $u \in \wj^0$ and $v \in \wh^0$.
%since all other words are admissible.
By definition of the quasi-shuffle product~$\qsh$ for $j\geq 1$ we have that $u \ai{1}{0}^{j-1} v \, \qsh\, \ai{1}{0}$ equals
{\footnotesize \allowdisplaybreaks
\begin{align*}
%&u \ai{1}{0}^{j-1} v \, \qsh\, \ai{1}{0} \=  
& j\, w \+  \sum_{i=1}^{m-1} \ai{1,\dots,1}{d_1,\dots,d_i}\ai{1}{0} \ai{1,\dots,1}{d_{i+1},\dots,d_m} \ai{1}{0}^{j-1}v \+ \sum_{i=1}^{r} u \ai{1}{0}^{j-1} \ai{k_1,\dots,k_i}{0,\dots,0}  \ai{1}{0} \ai{k_{i+1},\dots,k_r}{0,\dots,0} \\
 &
 \+ \sum_{i=1}^m \ai{1,\dots,1}{d_1,\dots,d_{i-1}} \biggl( \ai{1}{0}  \diamond \ai{1}{d_i}  \biggr) \ai{1,\dots,1}{d_{i+1},\dots,d_m} \ai{1}{0}^{j-1} \! v \+\sum_{i=1}^{j-1} u \ai{1}{0}^{i-1} \biggl( \ai{1}{0}  \diamond \ai{1}{0}  \biggr) \ai{1}{0}^{j-1-i} \! v \\
 &
 \+ \sum_{i=1}^r u \ai{1}{0}^{j-1} \ai{k_1,\dots,k_{i-1}}{0,\dots,0}  \biggl( \ai{1}{0} \diamond  \ai{k_i}{0} \biggr)\ai{k_{i+1},\dots,k_r}{0,\dots,0}\,.
\end{align*}}

Here every displayed term, except for $j w$, is either an element in $\wsn$ or is of the form $u' \ai{1}{0}^{j'}  v'$, where $u'\in \wj^0$ and $v'\in \wh^0$ and $j'\leq j-1$.
By induction on~$j$ we therefore see that $w$ can be written as in~\eqref{eq:wasregpol}. 

For the uniqueness assume that $  \sum_{j=0}^r w_j \qsh \ai{1}{0}^{\qsh j} \equiv 0 \mod \wsn$ for some $w_j \in \ws^0$. Then the term  $w_r \qsh \ai{1}{0}^{\qsh r}$ is the only summand which contains classes of the form~\eqref{eq:wordins1} with $j=r$. Since there are no relations among elements in $\ws^1$ modulo $\wsn$ we immediately obtain $w_r = 0$ and therefore $w_j=0$ for all $j=1,\dots,r$.
\end{proof}

\begin{example} %For example 
For $k\geq 2,d\geq 1$ we have 
\begin{align*}
    \ai{1,1,k}{d,0,0} \= \underbrace{\ai{1,k}{d,0}}_{w_1} \qsh \ai{1}{0} \meno \underbrace{\ai{1,k,1}{d,0,0} - \ai{1,1,k}{0,d,0} - \ai{1,k+1}{d,0}}_{w_0} \meno \underbrace{\ai{2,k}{d,0}}_{\in \,\wsn}\,. 
\end{align*}
Here $w_0,w_1 \in \ws^0$ and the last term is in $\wsn$ due to the $2$ in the top left entry.
\end{example}

% \begin{proof}[Proof of Theorem~\ref{thm:main}]
% \jw{to be written!}
% Iso: Proposition \ref{prop:alghoms}
% reg: Proposition \ref{prop:reg}

% Proposition \ref{prop:admissiblecase}
% \end{proof}

We end this subsection by defining an analogue of the involution $\iota$ on the space~$\ws$. Since Proposition~\ref{prop:alghoms} gives as for any well-normalized family of polynomials $\FF$ an isomorphism  $ \siso_\FF: (\ws,\qsf) \longrightarrow (\MS, \ost) $ we give the following definition.

\begin{definition}For any well-normalized family of polynomials $\FF$ we define the linear map $\iota_\FF$ by 
\begin{align*}
    \iota_\FF: \ws &\longrightarrow \ws\\
    w&\longmapsto (\siso_\FF^{-1} \circ \iota \circ \siso_\FF)(w)\,.
\end{align*}
\end{definition}

When we choose the Bernoulli--Seki model $\FF = \mathrm{s}$, then the involution~$\iota_\mathrm{s}$ can be described nicely using generating series. That is, given a depth $r\geq 1$ we write~$\gena$ for the formal power series in $\ws\llbracket X_1,Y_1,\dots,X_r,Y_r\rrbracket$ given by
\begin{align}\label{eq:gs}
	\gena \bi{X_1,\dots,X_r}{Y_1,\dots,Y_r} \defis \sum_{\substack{k_1,\dots,k_r\geq 1\\d_1,\dots,d_r \geq 0}} 	\ai{k_1,\dots,k_r}{d_1,\dots,d_r} X_1^{k_1-1} \dots X_r^{k_r-1} \frac{Y_1^{d_1}}{d_1!} \dots \frac{Y_r^{d_r}}{d_r!}\,.
\end{align}
%With this the action of $ \iota_\mathrm{s}$ can be described as follows:
\begin{proposition}\label{prop:iotags} We have 
	\begin{align}\label{eq:iotafaulhabergenseries}
		\iota_\mathrm{s} \left(	\gena \bi{X_1,\dots,X_r}{Y_1,\dots,Y_r}  \right) \= \gena\bi{Y_1 + \dots + Y_r,\dots,Y_1+Y_2,Y_1}{X_r,X_{r-1}-X_r,\dots,X_1-X_2}\,, 
	\end{align}
	where the involution~$\iota_\mathrm{s}$ on the left is applied coefficient-wise. 
\end{proposition}
\begin{proof}
This can be proven by using the same change of variables as it was done in \cite[Theorem 2.3]{B} by replacing the $q$-series there with the $\vec{u}$-bracket of $P_\mathrm{s}$ and then using the explicit description of $\iota$ mentioned in the proof of Lemma~\ref{lem:handjclosed}.
\end{proof}

% \jw{We have to discuss this paragraph. I find it hard to parse this.}\hen{Yea I agree}
% The involution~$\iota_\mathrm{s}$ is of particular interest when one is interested in describing relations among modular forms. More precisely, as it will be shown in \cite{BI}, one is interested in  $\iota_\mathrm{s}$-invariant algebra homomorphism $(\ws,\ast) \longrightarrow A$ for some $\Q$-algebra $A$. Since then the images of the words $\ai{k}{d}$ satisfy, up to a factor, the same algebraic relations as the $d$-th derivatives of the Eisenstein series $G_{k-d}\mspace{1mu}$.  In particular, the cases of holomorphic functions in the upper half-plane $A=\mathcal{O}(\mathbb{H})$ and $A=\Q\llbracket q \rrbracket$ are of interest if the depth one images are exactly the (derivatives of) Eisensteins series. 
% A candidate for the depth two part of such a map was constructed in in \cite{BKM} and for the case $A=\Q\llbracket q \rrbracket$ a general construction will be content of \cite{BB}. In the next section, we will use the results from Section~\ref{sec:limits} to give for any  well-normalized family of polynomials $\FF$ an $\iota_\FF$-invariant algebra homomorphism $(\ws,\ast) \longrightarrow \R[T]$, which generalizes the harmonic regularized multiple zeta values. This map can be seen as providing the constant term for a realization of the formal multiple Eisenstein series space defined in \cite{BI}.

\subsection{Bi-multiple zeta values} We now combine the results of Section~\ref{sec:limits} and~\ref{sec:algsetup} to define a bi-variant of multiple zeta values. These can be seen as regularized limits of polynomial functions on partitions, since they are defined for arbitrary words $w \in \ws$.
\begin{definition}\label{def:bimzv} We define the linear map 
\begin{align*}
    \zeta: \ws &\longrightarrow \mzv[T]\\
    w = 	\ai{k_1,\dots,k_r}{d_1,\dots,d_r} &\longmapsto \zeta(w) = \zebi{k_1,\dots,k_r}{d_1,\dots,d_r}
\end{align*}
as follows:
\begin{enumerate}[(i), leftmargin=*]\itemsep3pt
    \item \label{it:bizetai} For $w \in \wsn$ we set $\zeta(w) = 0$.
    \item For $w\in \ws^0$ we write $w = \ai{\hspace{3pt}1,\dots,1,\hspace{3pt} k_1,\dots,k_r}{d_1,\dots,d_m,\hspace{3pt}0,\dots,0\hspace{5pt}}$ and set
    \begin{align*}
        \zebi{\hspace{7pt} 1,\dots,1,k_1,\dots,k_r}{d_1,\dots,d_m,0,\dots,0\hspace{7pt}} \= \xi(d_1,\ldots,d_m)\, \zeta(k_1,\dots,k_r)\,,
    \end{align*}
    where $k_1\geq 2$, $d_m\geq 1$ and where $\xi$ is the conjugated multiple zeta value defined in Definition~\ref{def:conjzeta}.
    \item\label{it:bizetaiii} For $w = \sum_{j=0}^r w_j \qsh \ai{1}{0}^{\qsh j}$ with $w_j \in \ws^0$ we set 
    \begin{align*}
        \zeta(w) \= \sum_{j=0}^r \zeta(w_j) \,T^j\,.
    \end{align*}
\end{enumerate}
\end{definition}
Notice that due to Proposition~\ref{prop:reg} and part~\ref{it:bizetai} the value $\zeta(w)$ in \ref{it:bizetaiii} is well-defined.

% \begin{lemma} For $w_1,w_2 \in \ws^0$ we have $w_1 \ast w_2 \equiv w_1 \ast_{\mathrm{s}} w_2  \mod lower degree$.
% \end{lemma}

\begin{theorem}\label{thm:bimzv} For any  well-normalized family of polynomials $\FF$ the map $$\zeta: (\ws,\qsh) \longrightarrow \mzv[T]$$ is an $\iota_\FF$-invariant algebra homomorphism.
\end{theorem}
\begin{proof}
First we show that $\zeta$ is an algebra homomorphism. Since $\wsn$ is an ideal (Lemma~\ref{eq:nisideal}) and  $\zeta(w) = 0$ for $w \in \wsn$, we can restrict to the case when $w \notin \wsn$, i.e., $w \in \ws^1$. Moreover we can restrict to the elements in $\ws^0$ due to the definition in (iii). 
If $w \in \ws^0$ we can assume that we can write
$w = \ai{1,\dots,1,k_1,\dots,k_r}{d_1,\dots,d_m,0,\dots,0}$ with $d_m\geq 1$ (or $m=0$) and $k_1\geq 2$ (or $r=0$). 
Then we have by Corollary~\ref{prop:limitfactors} that 
\begin{align}\label{eq:zetaofws0}
    \zeta(w) \=  \Zdegree \ms[\mathrm{s}]{\hspace{7pt}1,\ldots,1,k_{1},\ldots,k_r}{d_1,\ldots,d_{m},0,\ldots,0\hspace{7pt}}  \= Z\siso_\mathrm{s}(w)\,,
\end{align}
where the weight limit~$Z$ is introduced in Definition~\ref{def:wtlim}.
Now for $k\geq 1$ and $f \in \MS$ set
\begin{align}
    Z_k(f) = \lim_{q\to 1}(1-q)^{k}\langle f \rangle_q\,,
\end{align}
i.e., if $k$ is the weight of $f$ then $Z_k(f)$ is exactly the weight limit of~$f$.

Let $w_1,w_2 \in \ws^0$ be given and write $k_1=\wt(w_1)$, $k_2= \wt(w_2)$ and $k=k_1+k_2$. Then, we have by Proposition~\ref{prop:alghoms} that $\siso_\mathrm{s}(w_1 \ast_{\mathrm{s}} w_2) = \siso_\mathrm{s}(w_1) \ost  \siso_\mathrm{s}(w_2)$ and $\siso_\mathrm{b}(w_1 \ast w_2) = \siso_\mathrm{b}(w_1) \ost  \siso_\mathrm{b}(w_2)$. Comparing the definition of $\ast =\ast_\mathrm{b}$ and $\ast_\mathrm{s}\mspace{1mu}$, we see that they are the same up to lower weight terms. In particular, their image under $Z_k$ is the same. Moreover, since $Z_k$ is defined via the $q$-bracket, which is an algebra homomorphism with respect to $\ost$ (Proposition~\ref{prop:polfctonpartprop}), we get
\begin{align*}
    \zeta(w_1) \zeta(w_2) &\= Z_{k_1}\siso_\mathrm{s}(w_1)\, Z_{k_2}\siso_\mathrm{s}(w_2) \= Z_k\left( \siso_\mathrm{s}(w_1)  \ost \siso_\mathrm{s}(w_2) \right)\\
    &\= Z_k\left( \siso_\mathrm{s}(w_1 \ast_{\mathrm{s}} w_2)\right) \= Z_k\left( \siso_\mathrm{s}(w_1 \ast w_2) \right).
\end{align*}
But this also equals $\zeta(w_1 \ast w_2)$, since $w_1 \ast w_2$ is a linear combination of words which are either in $\ws^0$ (i.e., they can be written as in~\eqref{eq:zetaofws0}) or which are in $\wsn$. For the latter the weight limit, and therefore the image under $Z_k\mspace{1mu}$, vanishes as seen in Lemma~\ref{lem:liminq}.

It remains to show that $\zeta$ is $\iota_\FF$-invariant for any  well-normalized family of polynomials $\FF$. Let $w = \sum_{j=0}^r w_j \qsh \ai{1}{0}^{\qsh j}$ with $w_j \in \ws^0$ and define the $\tilde{w}_j \in  \ws^0$ by $\iota_\FF(w) =:  \sum_{j=0}^r \tilde{w}_j \qsh \ai{1}{0}^{\qsh j}$. We need to show that $\zeta(w_j) = \zeta(\tilde{w}_j)$ for all $j=0,\dots,r$. Since the $q$-bracket is $\iota$ invariant we have $\langle \siso_\FF(w) \rangle_q = \langle \siso_\FF( \iota_\FF(w)) \rangle_q\mspace{1mu}$, i.e., we get, up to $q$-series which vanish under the weight limit (elements in $\langle \iota_\FF(\wsn) \rangle_q$), that
\begin{align*}
     \sum_{j=0}^r \langle \siso_\FF(w_j)  \rangle_q \,   \biggl\langle \siso_\FF \ai{1}{0}  \biggr\rangle_{\!\!q}^{ j} \,\equiv\,   \sum_{j=0}^r \langle \siso_\FF(\tilde{w}_j)  \rangle_q  \,  \biggl\langle \siso_\FF \ai{1}{0}  \biggr\rangle_{\!\!q}^{ j} \,.
\end{align*}
For any  well-normalized family of polynomials $\FF$ we have 
\begin{align*}
    \biggl\langle \siso_\FF \ai{1}{0}  \biggr\rangle_{\!\!q} \= \sum_{n>0} \frac{q^n}{1-q^n}\,.
\end{align*}
Since\footnote{Here $f(q) \asymp g(q)$ means there exist constants $c_1,c_2$ with $c_1 |g(q)| \leq |f(q)| \leq c_2 |g(q)|$ as  $q\rightarrow 1$.} $\sum_{n>0} \frac{q^n}{1-q^n} \asymp -\frac{\log(1-q)}{(1-q)}$ (see \cite{P}) and $ \langle \siso_\FF(w_j) \rangle_q   \asymp \frac{1}{(1-q)^{\wt(w_j)}}$ (by Section~\ref{sec:limits}) as $q\rightarrow 1$, we see that 
%for $j>0$ we actually have $\langle \siso_\FF(w_j) \rangle_q = \langle \siso_\FF(\tilde{w_j}) \rangle_q$ and for $j=0$ the weight limits of $\langle \siso_\FF(w_0)  \rangle_q $ and $\langle \siso_\FF(\tilde{w_0})  \rangle_q $ equal 
for $j\geq 0$ the weight limits of $\langle \siso_\FF(w_j)  \rangle_q $ and $\langle \siso_\FF(\tilde{w_j})  \rangle_q $ are equal. By the discussion above applying $\zeta$ to words in $\ws^0$ exactly yields the weight limits, which are also independent of the choice of (well-normalized) $\FF$, from which we deduce $\zeta(w_j) = \zeta(\tilde{w}_j)$ for all $j=1,\dots,r$.
\end{proof}

\section{Induced relations among Multiple Zeta Values}\label{sec:mzv}
Here we interpret and provide relations among MZV's from our partition analogues. 

%\subsection{Derivative of relations???}

\subsection{Double shuffle relations} In this section, we want to explain why the relations $ \langle f  \osh g -  f \ost g\rangle_q=0$ for any $f,g\in \q^\PP$ (see Proposition~\ref{prop:doubleshuffle}) can be seen as an analogue of the double shuffle relations for multiple zeta values.
For this, we first try to indicate why the product~$\oshh$ can be seen as the correct analogue of the shuffle product of multiple zeta values. We first recall the classical shuffle product. 
Setting $\wh=\Q\langle x,y\rangle$ the spaces $\wh^1$ and $\wh^0$ from the previous section can be identified with the subspaces $\Q+\wh y$ and $\Q + x \wh y$ of $\wh$ via $z_k = x^{k-1}y$. On these spaces one can then define the shuffle product  as the $\Q$-bilinear product, which satisfies $1 \shuffle w = w \shuffle 1 = w$ for any word $w\in \wh$ and
\begin{align*}
a_1 w_1 \shuffle a_2 w_2 = a_1 (w_1 \shuffle a_2 w_2) + a_2 (a_1 w_1 \shuffle w_2)
\end{align*}
for any letters $a_1, a_2 \in \{x,y\}$ and words $w_1,w_2 \in \wh$. The spaces $\wh^1$ and $\wh^0$ are closed under this product and a classical result is (see \cite[Corollary 2.10]{B2}) that the  $\Q$-linear map $\zeta: \wh^0 \rightarrow \mzv$ defined in \eqref{eq:defzetamap} is also an algebra homomorphism with respect to the shuffle product $\shuffle$. For $w,v\in \wh^0$ this then implies the (finite) double shuffle relations $\zeta(w \shuffle v - w \ast v)=0$.

By the definition of $\ost$ together with  Proposition~\ref{prop:alghoms} and Theorem~\ref{thm:bimzv}, it becomes clear why $\ost$ can be seen as the analogue of the stuffle product. On the other hand the definition of~$\oshh$ in Definition~\ref{def:products} as $F\osh G = \iota(\iota(F)\ost \iota(G))$ is completely different to the definition of $\shuffle$ above and it is not obvious why this is the correct analogue. 
This was recently shown by Brindle in \cite{Bri1}, in a slightly different context, by showing that the involution~$\iota$ (which is called the partition relation) corresponds to a certain duality introduced by Singer in \cite{Sin}. With this one can then show (see \cite[Theorem 3.46]{Bri1}) that the product defined by combining this duality together with the stuffle product indeed gives the correct analogue of the shuffle product. 

In \cite{BI} it will be shown that one not just only gets the finite double shuffle relations, but also the extended double shuffle relations, i.e., conjecturally all relations among multiple zeta values. 

We will leave out the details here and just indicate this with two examples. 
\begin{example} \label{eq:dshanalog} The harmonic product of $z_2$ and $z_3$ is given by $z_2 \ast z_3 = z_2 z_3 + z_3 z_2 + z_3$, i.e. $\zeta(2) \zeta(3) = \zeta(2,3)+\zeta(3,2)+\zeta(5)$.  The harmonic product~$\ost$ of $\ms[\mathrm{s}]{2}{0}$ and $\ms[\mathrm{s}]{3}{0}$ is given by
\begin{align*}
   \ms[\mathrm{s}]{2}{0} \ost \ms[\mathrm{s}]{3}{0} &=    \ms[\mathrm{s}]{2,3}{0,0}+  \ms[\mathrm{s}]{3,2}{0,0}+  \ms[\mathrm{s}]{5}{0}-\frac{1}{12}  \ms[\mathrm{s}]{3}{0}\,,
\end{align*}
which gives exactly the harmonic product among multiple zeta values after applying the weight $5$ limit (defined by Definition~\ref{def:wtlim}), since the last term vanishes. 
The shuffle product of $z_2$ and $z_3$ is given by $z_2 \shuffle z_3 = xy \shuffle xxy =  xyxxy + 3 xxyxy + 6 xxxyy =  z_2z_3 + 3 z_3 z_2 + 6 z_4 z_1$. In particular, applying the weight $5$ limit, this implies the shuffle product for multiple zeta values 
\begin{align}\label{eq:shufflez2z3}
    \zeta(2) \zeta(3) = \zeta(2,3) + 3\zeta(3,2) + 6 \zeta(4,1)\,.
\end{align}
The shuffle product~$\oshh$ of $\ms[\mathrm{s}]{2}{0}$ and $\ms[\mathrm{s}]{3}{0}$ is given by
\begin{align*}
\ms[\mathrm{s}]{2}{0} \osh \ms[\mathrm{s}]{3}{0} &= \iota\left( \iota\ms[\mathrm{s}]{2}{0} \ost \iota\ms[\mathrm{s}]{3}{0} \right) = \iota\left(  \frac{1}{2} \ms[\mathrm{s}]{1}{1}   \ost  \ms[\mathrm{s}]{1}{2} \right) , \\
% \end{align*}
% which equals
% \begin{align*}
	&= \frac{1}{2} \iota\left( \ms[\mathrm{s}]{1,1}{1,2}  + \ms[\mathrm{s}]{1,1}{2,1} + \ms[\mathrm{s}]{2}{3}  - \ms[\mathrm{s}]{1}{3}   \right) \\
	&=  \ms[\mathrm{s}]{2,3}{0,0} +  3 \ms[\mathrm{s}]{3,2}{0,0} + 6 \ms[\mathrm{s}]{4,1}{0,0}+ 3 \ms[\mathrm{s}]{4}{1}  - 3 \ms[\mathrm{s}]{4}{0}.
\end{align*}
From the results in Section~\ref{sec:parttomzv} we see that this exactly implies the shuffle product~\eqref{eq:shufflez2z3} by taking the weight $5$ limit (observe that the weight $5$ limit of $\mssmall[\mathrm{s}]{4}{1}$ and of $\mssmall[\mathrm{s}]{4}{0}$ vanish).
\end{example}
\begin{example}\label{ex:pi}
%A natural question is which polynomial functions of partitions have quasimodular $q$-brackets. If $f,g \in \MS$ have quasimodular $q$-brackets then clearly the $q$-bracket of $f\osh g = f \ost g$ is also quasimodular. 
Recall $\Zdegree(f)$ is a rational multiple of an even power of $\pi$ for $f\in \MM$. 
%The limit of functions with quasimodular $q$-bracket is always a rational multiple of a power of $\pi$. 
Therefore, having diagram~\eqref{fig:overview} in mind, given a linear combination of multiple zeta values which is a rational multiple of $\pi^k$, one might ask if there exist a `natural lift' to a polynomial function in $\MS$ whose $q$-bracket is quasimodular.
One such family is given by a special case of Proposition~\ref{prop:pkkkisquasimodular}, since $\mssmall[\mathrm{s}]{2m,\dots,2m}{0,\dots,0}$ is quasimodular for $m\geq 1$ with limit $\zeta(2m,\dotsm,2m) =\zeta(\{2m\}^n)\in \Q \pi^{2mn}$. Another famous family of multiple zeta values which evaluates to a rational multiple of $\pi^k$ are given by the $3$-$1$-formula, which states that for any $n\geq 1$ we have 
\[ \zeta(3,1,\dots,3,1) = \zeta(\{3,1\}^n) = \frac{2 \pi^{4n}}{(4n+2)!}.\] 
The proof of this fact which was first conjectured by Don Zagier, then proven in \cite{BBBL1}, and later reduced to the following argument (\cite{BBBL2}) using the shuffle product. One can show that one has the following identity in the algebra $(\wh^0,\shuffle)$
	\begin{align}\label{eq:3131}
	\sum_{j=-n}^n (-1)^j z_2^{n-j} \shuffle z_2^{n+j} = 4^n (z_3 z_1)^n\,.
	\end{align}
The statement then follows after using $\zeta(2,\dots,2) = \frac{\pi^{2n}}{(2n+1)!}$. We can consider the left-hand side of this expression in our setup and get, by using again Proposition~\ref{prop:pkkkisquasimodular}, that for any $n\geq 1$ the $q$-bracket of 
	\begin{align*}
T(n) \defis	\sum_{j=-n}^n (-1)^j\,  P_\mathrm{s}\bi{\overbrace{2,\dots,2}^{n-j}}{0,\dots,0}  \osh P_\mathrm{s}\bi{\overbrace{2,\dots,2}^{n+j}}{0,\dots,0} \,\in \, \mathbb{M}
	\end{align*}
is quasimodular. Similar to \eqref{eq:3131} this sum evaluates to $4^n P_\mathrm{s}\bi{3,1,\dots,3,1}{0,\dots,0} $ plus some extra terms with vanishing weight $4n$ limit. For example, we have 
{\small 
\begin{align*}
    T(1) \= & 4 P_\mathrm{s}\bi{3,1}{0,0} + 2 P_\mathrm{s}\bi{3}{1} - 2 P_\mathrm{s}\bi{3}{0}\,, \\
% \end{align*}
% and
% \begin{align*}
T(2) \= &  16 P_\mathrm{s}\bi{3,1,3,1}{0,0,0,0} + 8 P_\mathrm{s}\bi{3,3,1}{1,0,0} - 8 P_\mathrm{s}\bi{3,3,1}{0,1,0} - 8 P_\mathrm{s}\bi{3,3,1}{0,0,0} \+\\
&+ 8P_\mathrm{s}\bi{3,1,3}{1,0,1} - 8P_\mathrm{s}\bi{3,1,3}{1,0,0} - 4 P_\mathrm{s}\bi{3,3}{0,2} + 4 P_\mathrm{s}\bi{3,3}{1,1} \+\\
&- 4 P_\mathrm{s}\bi{3,3}{1,0} + 4 P_\mathrm{s}\bi{3,3}{0,0}\,.
\end{align*}}
% n=3 would be this:
% -32*w([3 3 1 3 1], [0 1 0 0 0])-32*w(3,3,1,3,1)+16*w([3 3 3 1], [0 1 1 0])+16*w([3 3 3 1], [0 0 1 0])+16*w(3,3,3,1)-32*w([3 1 3 3 1], [0 0 0 1 0])-32*w(3,1,3,3,1)-16*w([3 3 3 1], [1 0 1 0])-16*w([3 3 3 1], [1 0 0 0])+64*w(3,1,3,1,3,1)+32*w([3 3 1 3 1], [1 0 0 0 0])-16*w([3 3 1 3], [0 1 0 1])+16*w([3 3 1 3], [0 1 0 0])+16*w(3,3,1,3)+32*w([3 1 3 3 1], [0 0 1 0 0])+32*w([3 1 3 1 3], [0 0 0 0 1])-32*w(3,1,3,1,3)-16*w([3 3 1 3], [0 0 0 1])-16*w([3 3 3 1], [0 2 0 0])+16*w([3 3 3 1], [1 1 0 0])+16*w([3 3 1 3], [1 0 0 1])-16*w([3 1 3 3], [0 0 0 2])+16*w([3 1 3 3], [0 0 1 1])+8*w([3 3 3], [0 1 2])-8*w([3 3 3], [0 2 1])-8*w([3 3 3], [1 0 2])+8*w([3 3 3], [1 1 1])+8*w([3 3 3], [0 0 2])-16*w([3 3 1 3], [1 0 0 0])-16*w([3 1 3 3], [0 0 1 0])-8*w([3 3 3], [0 1 1])+8*w([3 3 3], [0 2 0])-8*w([3 3 3], [1 1 0])+16*w(3,1,3,3)+8*w([3 3 3], [1 0 0])-8*w(3,3,3)
\end{example}

\subsection{Bloch--Okounkov relations}\label{sec:BO}
% \jw{Introduce B--O thm., and the idea of using this to find relation. To be written together with the introduction of the paper.}

The exponential isomorphism, introduced in \cite{H2} (see also \cite{HI}), maps any quasi-shuffle algebra to a shuffle algebra on the same words. We give an analogous definition of an exponential map on our space of polynomial functions on partitions. It turns out that the corresponding basis is convenient for expressing shifted symmetric functions as elements of $\MS$ (see, also, \cite{PV}). 

\begin{definition}\label{def:exp}
Given $\vec{m}\in \Z^n$, let $\Aut(\vec{m})=\prod_{j\in \Z} r_j(\vec{m})!$ with $r_j(\vec{m})$ the number of indices $i$ with $m_i=j$. This allows us to define $\exp:\MS\to \MS$ by 
\[\msexpl[\FF]{k_1, \dots , k_r}{d_1,\dots,d_r}{\lambda}= 
\sum_{m_1 \geq \dots \geq m_r > 0} \frac{1}{\Aut(\vec{m})}\prod_{j=1}^{r} m_j^{d_j} f_{k_j}(r_{m_j}(\lambda))\,.\]
\end{definition}
Recall the \emph{Bernoulli polynomial}~$B_k$ ($k\geq 0$) is the unique polynomial such that
\[ \int_{x}^{x+1} B_k(u) \,du \= x^k.\]
Then, $B_k$ is of degree $k$ with as constant term $B_k(0)=B_k$ the $k$-th Bernoulli number.

\begin{proposition}\label{prop:QkinStanleyCoordinates} 
For all $k\geq 2$ we have
\[Q_k \= \frac{B_{k}(\tfrac{1}{2})}{k!}+\sum_{i=0}^{k-2}\sum_{j=0}^{i}\frac{(-1)^{i+j}}{(k-i-1)!} \frac{B_j(\tfrac{1}{2})}{j!}
\msexp[\FF]{1,\ldots,1}{\underbrace{0,\ldots,0}_{i-j},k-1-i}\, \in \MM \,,\]
where $Q_k$ is a shifted symmetric function defined in~\eqref{eq:defQk}.
\end{proposition}
\begin{proof}
Given $m>0$, abbreviate $\sum_{m'>m}r_{m'}(\lambda)$ by $r_{>m}\mspace{1mu}$. If $\lambda_i=m$, then $r_{>m}< i \leq r_{>m-1}\mspace{1mu}$, hence
\begin{align}
p_k(\lambda) &\defis \sum_{i\geq 1} \bigl( (\lambda_i-i+\tfrac{1}{2})^k-(-i+\tfrac{1}{2})^k\bigr) \\
&\= \sum_{m=1}^\infty \sum_{j=1}^{r_m(\lambda)}\bigl((m-r_{>m}-j+\tfrac{1}{2})^{k}-(-r_{>m}-j+\tfrac{1}{2})^k\bigr).
\end{align}
Therefore, letting
\[t_b(m,\lambda) \defis \sum_{j=1}^{r_m(\lambda)}(r_{>m}+j-\tfrac{1}{2})^b,\]
one finds
\[p_k(\lambda) \= \sum_{m=1}^\infty\sum_{i=0}^{k-1}\binom{k}{i}m^{k-i}(-1)^i\, t_i(m,\lambda).\]
Expanding $t_b(m,\lambda)$ we have
\begin{align}
t_b(m,\lambda) \= \sum_{j=1}^{r_m(\lambda)}\sum_{l=0}^b\binom{b}{l}(r_{>m})^l(j-\tfrac{1}{2})^{b-l}.
\end{align}
By computing the generating series $\sum_{m\geq 0} \sum_{j=1}^n (j-\tfrac{1}{2})^{m}\frac{z^m}{m!}$, one finds that
\[\sum_{j=1}^n (j-\tfrac{1}{2})^{m} \= \frac{1}{m+1} \sum_{j=0}^m (-1)^j\binom{m+1}{j}B_j(\tfrac{1}{2})\,n^{m+1-j},\]
so that
\begin{align}
t_b(m,\lambda) &\= \sum_{l=0}^b\sum_{j=0}^{b-l}\binom{b}{l}(r_{>m})^l\frac{1}{b-l+1}(-1)^j\binom{b-l+1}{j}B_j(\tfrac{1}{2})r_m(\lambda)^{b-l+1-j}\\
&\= \sum_{l=0}^b\sum_{j=0}^{b-l}\frac{b!}{l!(b-l+1-j)!j!}(-1)^jB_j(\tfrac{1}{2})\left(\sum_{m_1,\ldots,m_l>m}r_{m_1}(\lambda)\cdots r_{m_l}(\lambda)\!\right)r_m(\lambda)^{b-l+1-j}.
\end{align}
Write $\vec{m}=(m_1,\ldots,m_{b-j+1})$. Then,\
\begin{align}
t_b(m,\lambda) &\= \sum_{l=0}^b\sum_{j=0}^{b-l}\frac{b!}{j!}(-1)^jB_j(\tfrac{1}{2})\left(\sum_{\substack{m_1\geq\ldots\geq m_l>m\\m_{l+1}=\ldots=m_{b-j+1}=m}}\!\frac{r_{m_1}(\lambda)\cdots r_{m_{b-j+1}}(\lambda)}{\Aut(\vec{m})}\right)\\
&\= \sum_{j=0}^{b}\frac{b!}{j!}(-1)^jB_j(\tfrac{1}{2})\left(\sum_{\substack{m_1\geq\ldots\geq m_{b-j+1}=m}}\frac{r_{m_1}(\lambda)\cdots r_{m_{b-j+1}}(\lambda)}{\Aut(\vec{m})}\right),\label{eq:t}
\end{align}
where we recall $\Aut(\vec{m})=\prod_{j\in \Z} r_j(\vec{m})!$.
%where $\vec{m}=(m_1,\ldots,m_{b-j+1})$. 
Hence,
\[p_k(\lambda) \= \sum_{i=0}^{k-1}\sum_{j=0}^{i}(-1)^{i+j}\frac{k!}{(k-i)!} \frac{B_j(\tfrac{1}{2})}{j!}
%S\bigg(\ar{\overline{1,\ldots,1}}{\underbrace{0,\ldots,0}_{i-j},k-i}\bigg).
\msexp[\FF]{1,\ldots,1}{\underbrace{0,\ldots,0}_{i-j},k-i}.\]
By definition $Q_k(\lambda)=\beta_k+\frac{p_{k-1}(\lambda)}{(k-1)!}$, which finishes the proof. 
\end{proof}

%\begin{corollary} The difference  
%\[Q_{k}-\sum_{i=0}^{k-2}\frac{(-1)^{i}}{(k-1-i)!} \msexp[\FF]{1,\ldots,1}{\underbrace{0,\ldots,0}_{i-j},k-1-i}\]
% is of weight $\leq k-1$. 
% \[Q_{\vec{k}}-\sum_{i_1=0}^{k_1-2}\cdots \sum_{i_r=0}^{k_r-2}\prod_{j=1}^{r}\frac{(-1)^{i_j}}{(k_j-i_j-1)!}S\bigg(\ar{\overline{1,\ldots,1}}{\underbrace{0,\ldots,0}_{i_j},k_j-1-i_j}\bigg)\]
% is of weight $\leq |\vec{k} |-1$. 
%\end{corollary}

\begin{corollary}\label{cor:bo} We have 
\begin{align}
\sum_{k\geq 2}\sum_{i=0}^{k-2}\frac{(-1)^{i}}{(k-1-i)!} \,\xi({\underbrace{0,\ldots,0}_{i},k-1-i})\,z^k  & \= %\frac{1}{z^2}\left(
\exp\left(\sum_{n\geq 2} \zeta(n)\frac{z^n+(-z)^n}{n}\right)-1
%\right)
\\
&\= 
\sum_{k\geq 2}\sum_{i=0}^{k-2}(-1)^{i}\,\zeta(k-i,\underbrace{1,\ldots,1}_{i})\,z^k.
\end{align}
\end{corollary}
\begin{proof}
The first equality follows directly from Proposition~\ref{prop:QkinStanleyCoordinates} by taking the weight limit. Note that for all $f\in \MS$ we have $Z\exp f=Z f$. 
Moreover, for computing the weight limit of $Q_k$ we make use of the Bloch--Okounkov theorem, which expresses the generating series of the~$Q_k$ as a certain Jacobi form %of which one easily computes the constant term 
\cite[Theorem~6.1]{BO}, i.e.,
\[
\sum_{k\geq 0} \langle Q_k\rangle_q\, z^{k-1} \= \frac{1}{\Theta(q,z)} \defis \frac{1}{z}\exp\biggl(2\sum_{\substack{k\geq 2 \\ k \text{ even}}}  G_k(q) \,\frac{z^k}{k} \biggr).
\]
Recall that the Eisenstein series in this work are normalized such that the weight limit of~$G_k$ equals~$\zeta(k)$. 

For the second equality, we combine Proposition~\ref{prop:QkinStanleyCoordinates} with the explicit formula for~$\iota$ in Remark~\ref{rk:iotaexplicit} in order to find that
\begin{align}\label{eq:iotaQk}\iota(Q_k) \= \frac{B_{k}(\tfrac{1}{2})}{k!}\+\sum_{i=0}^{k-2}\sum_{j=0}^{i} \frac{(-1)^{i+j} B_j(\tfrac{1}{2})}{j!}\,
\msexp[\mathrm{s}]{k-i,1,\ldots,1}{0,\underbrace{0,\ldots,0}_{i-j}}
%S\bigg(\ar{\overline{k-i,1,\ldots,1}}{0,\underbrace{0,\ldots,0}_{i-j}}\bigg)
.
\end{align}
Then, similarly, the second equality follows by taking the weight limit.
%Hence, the second equality follows from the Bloch--Okounkov theorem after taking the $q$-bracket, multiplying by $(1-q)^k$ and taking the limit $q\to 1$.
\end{proof}

\begin{remark}\mbox{}\\[-12pt] 
\begin{enumerate}[(i), leftmargin=*]\itemsep3pt
    \item The second equality in the above formula is a special case of the Ohno--Zagier relations \cite[Theorem~1]{OZ} after setting $s=1$ and $x=-y$, which goes back to unpublished results of Zagier from 1995. 
    \item  By calculating the $q$-bracket (instead of the weight limit) of both sides of the equality in Proposition~\ref{prop:QkinStanleyCoordinates} one obtains Ohno--Zagier relations for $q$-analogues of multiple zeta values. These might give special cases of the relations proven in \cite{OT}. 
\end{enumerate}
\end{remark}

The relation in the corollary above follows from exploiting the facts that $\langle Q_k\rangle_q$ is a quasimodular form, as well as  that $Q_k$ is a polynomial function on partitions. The same holds true for any shifted symmetric function, i.e., any polynomial in the generators~$Q_k\mspace{1mu}$. We illustrate this by considering $Q_4Q_3$:
\begin{example}
%We have
%\[ Q_3 = -\ms{1,1}{0,1}-\frac{1}{2}\ms{2}{1}+\frac{1}{2}\ms{1}{2}. \]
Recall that by computing the weight limit of $Q_3$ and $Q_4\mspace{1mu}$, cf.~Corollary~\ref{cor:bo}, we obtain $\zeta(3)-\zeta(2,1)=0$ and
$\zeta(4)-\zeta(3,1)+\zeta(2,1,1)=0$ respectively.
If we would proceed in the same way computing the weight limit of $Q_3Q_4\mspace{1mu}$, we find (the trivial statement) that the shuffle product of $\zeta(3)-\zeta(2,1)$ and $\zeta(4)-\zeta(3,1)+\zeta(2,1,1)$ vanishes as well. We can, however, do more. 

Instead, consider
\[ Q_4Q_3-Q_4\ost Q_3\,. \]
The $q$-bracket of this function vanishes, as can be seen from the fact that $Q_3(\lambda) = -\omega(Q_3)(\lambda)$.  The degree is equal to $6$, whereas the weight is $7$. Computing the degree-$6$ limit we obtain
\begin{align} -10&\zeta(5)+\zeta(2,3)+3\zeta(3,2)+16\zeta(4,1)+\\
&-16\zeta(3,1,1)-3\zeta(2,2,1)-\zeta(2,1,2)+10\zeta(2,1,1,1)=0,\end{align}
a formula which easily follows from the duality relations for multiple zeta values. Hence, we see that not only the weight-limits, but also the degree-limits of polynomial functions on partitions give rise to interesting (families of) relations between multiple zeta values. 
\end{example}

\subsection{Zagier's arm-leg moments}\label{sec:armbein}
\newcommand{\ua}{\underline{a}}
\newcommand{\ub}{\underline{b}}

In \cite[Theorem 8]{Z} it is shown that for any even polynomial $g \in \Q[x,y]$ the $q$-bracket of the function $ \mathcal{A}_g: \PP \rightarrow \Q$ defined by 
\begin{equation}\label{eq:defAP}
    \mathcal{A}_g(\lambda) \defis \sum_{\xi\in Y_\lambda} g( \ua(\xi), \ub(\xi) ) 
\end{equation}
has a quasimodular $q$-bracket. Here the sum is over all cells $\xi$ of the Young diagram of $\lambda$ and $\ua(\xi)=a(\xi)+\frac{1}{2}$, $\ub(\xi)=b(\xi)+\frac{1}{2}$, where $a(\xi)$ and $b(\xi)$ denote the arm- and leg-lengths of the cell $\xi$. 

We now show that in fact the function $\mathcal{A}_g\in \MM$, and that by computing its weight limits we obtain the sum formula for multiple zeta values. 
\begin{theorem} For all polynomials $g\in \Q[x,y]$, we have $\mathcal{A}_g\in \MM$ and the weight of~$\mathcal{A}_g$ is $\leq \deg g+2$. 
For $g(x,y)=\frac{x^a y^b}{a!b!}$ the difference
\begin{align}\label{eq:AP}\mathcal{A}_g \meno \sum_{i=0}^a (-1)^{a-i} \frac{1}{i!(a+1-i)!}
\msexp[\FF]{1,\ldots,1}{i,\underbrace{0,\ldots,0}_{b-1},a-i+1}
%P\bigg(\ar{\overline{1,\ldots,1}}{i,\underbrace{0,\ldots,0}_{b-1},a-i+1}\bigg) 
\end{align}
is of weight $\leq \deg g + 1$. Here, $\exp:\MS\to\MS$ is defined in Definition~\ref{def:exp}.
\end{theorem}
\begin{proof}
Suppose $\xi=(x,y)\in Y_\lambda\mspace{1mu}$. Then, we write $m=\lambda_x$ and $j=1-x+r_{\geq m}(\lambda)$, where we denote $r_{\geq \ell}(\lambda) = \sum_{\ell'\geq \ell} r_{\ell'}(\lambda)$. Note that $1\leq j \leq r_{m}(\lambda)$.
We can rewrite~\eqref{eq:defAP} as
\begin{align}
    \mathcal{A}_g(\lambda) \= \sum_{m \geq y \geq 1} \sum_{j=1}^{r_{m}(\lambda)} g\bigl(m - y + \tfrac{1}{2}, r_{\geq y}(\lambda)-r_{\geq m}(\lambda) + j - \tfrac{1}{2} \bigr)\,.
\end{align}
In particular we see that for any polynomial $g$ we have $\mathcal{A}_g \in \MS$. 

Analogous to \eqref{eq:t}, for $b\geq 0$ we have
\begin{multline}
\sum_{j=1}^{r_m(\lambda)}(r_{\geq y}(\lambda)-r_{\geq m}(\lambda) + j - \tfrac{1}{2})^b \\
% &= \sum_{j=1}^{r_m(\lambda)}\sum_{\ell=0}^b \binom{b}{l}(r_{\geq y}-r_{\geq m})^l (j-\tfrac{1}{2})^{b-l} \\
% &=\sum_{\ell=0}^b\sum_{j=0}^{b-l}\binom{b}{l}(r_{\geq y}-r_{\geq m})^l\frac{1}{b-\ell+1} (-1)^{j}\binom{b-\ell+1}{j} B_j(\tfrac{1}{2})r_m(\lambda)^{b-\ell+1-j} \\
% &= \sum_{\ell=0}^b \sum_{j=0}^{b-l} \frac{b!}{\ell!(b-\ell+1-j)!j!}(-1)^{+j}B_j(\tfrac{1}{2})\left(\sum_{m>m_1,\ldots,m_\ell\geq y} r_{m_1}(\lambda)\cdots r_{m_\ell}(\lambda)\right) r_m(\lambda)^{b-\ell+1-j}\\
% &= \sum_{\ell=0}^b \sum_{j=0}^{b-\ell}\frac{b!}{j!}(-1)^{j}B_j(\tfrac{1}{2})
% \left(\sum_{\substack{m>m_1\geq\ldots\geq m_l\geq y\\m_{l+1}=\ldots=m_{b-j+1}=m}}\frac{r_{m_1}(\lambda)\cdots r_{m_{b-j+1}}(\lambda)}{\Aut(\vec{m})}\right)\\
\= \sum_{j=0}^{b}\frac{b!}{j!}(-1)^jB_j(\tfrac{1}{2})\left(\sum_{\substack{m=m_1\geq\ldots\geq m_{b-j+1}\geq y}}\frac{r_{m_1}(\lambda)\cdots r_{m_{b-j+1}}(\lambda)}{\Aut(\vec{m})}\right).
\end{multline}
Hence, in the special case $g(x,y)=\frac{x^ay^b}{a!b!}$ we get
\begin{align}\mathcal{A}_g &= 
\sum_{j=0}^{b}\sum_{\substack{m_1\geq\ldots\geq m_{b-j+1}\geq y\geq 1}} \frac{(-1)^j}{a!j!}B_j(\tfrac{1}{2}) (m_1-y+\tfrac{1}{2})^a \frac{r_{m_1}\cdots r_{m_{b-j+1}}}{\Aut(\vec{m})} \\
&= \sum_{j=0}^{b}\sum_{\substack{m_1\geq\ldots\geq m_{b-j+1}\geq 1}} \frac{(-1)^j}{a!j!}B_j(\tfrac{1}{2})(\tilde{\mathrm{s}}_{a+1}(m_1)-\tilde{\mathrm{s}}_{a+1}(m_1-m_{b-j+1})) \frac{r_{m_1}\cdots r_{m_{b-j+1}}}{\Aut(\vec{m})}, \label{eq:ApinMS}
\end{align}
where $\tilde{\mathrm{s}}_{a+1}(n) = \sum_{i=1}^n (i-\tfrac{1}{2})^a.$ As $\tilde{\mathrm{s}}_{a+1}(n)$ is a polynomial in $n$ of degree $a+1$ and with leading coefficient $\frac{1}{a+1}$, we have that
\begin{align}\label{eq:highdegree}\tilde{\mathrm{s}}_{a+1}(u)-\tilde{\mathrm{s}}_{a+1}(u-v) &\= \frac{1}{a+1}\sum_{i=0}^a \binom{a+1}{i} (-1)^{a-i} u^i v^{a-i+1} + \ldots,\end{align}
where the ommited terms are of total degree at most $a$. Substituting~\eqref{eq:highdegree} in~\eqref{eq:ApinMS} for $u=m_1$ and $v=m_b$ gives the desired result. 
\end{proof}

\begin{corollary}[Sum formula] For all $a,b\geq 0$, we have
\begin{align}\label{eq:sumformula}
\sum_{i=0}^a \frac{(-1)^{a-i}}{i!(a+1-i)!}\,
\xi(i,\underbrace{0,\ldots,0}_{b-1},a-i+1)&\=\zeta(a+b+2) \\
&\=\!\sum_{\substack{k_1+\ldots+k_{b+1}=a+b+2\\ k_1\geq 2,\,k_i\geq 1}}\!\! \zeta(k_1,\ldots,k_{b+1}).
\end{align}
\end{corollary}
\begin{proof}
We compute the weight limit of~\eqref{eq:AP}. By \cite[p.~367]{Z} we have
\[ \langle \mathcal{A}_g \rangle_q \= \sum_{n=1}^\infty \sum_{i=0}^{n-1} g(i+\tfrac{1}{2},n-i-\tfrac{1}{2})\, \frac{q^n}{1-q^n}\]
for all $g\in \Q[x,y]$. Now let $g(x,y)=\frac{x^ay^b}{a!b!}$. Then, as
\[\sum_{i=0}^{n-1}g(i+\tfrac{1}{2},n-i-\tfrac{1}{2}) \= \frac{1}{(a+b+1)!}n^{a+b+1}+ O(n^{a+b}),\]
we find $Z(\mathcal{A}_g) = \zeta(a+b+2),$ which proves the first equality. 

The second equality follows by applying the involution $\iota$. Namely, 
the left hand side of~\eqref{eq:sumformula} is the coefficient of $X^{a+1}$ in 
\begin{align*}
    \zeta\left( \gena \bi{0,0,\dots,0}{X,0,\dots,0} -	\gena \bi{0,0,\dots,0,0}{X,0,\dots,0,-X} \right),
\end{align*}
where $\gena$ is defined by~\eqref{eq:gs}. By Proposition~\ref{prop:iotags}, together with Theorem~\ref{thm:bimzv} by which we know that $\zeta$ is $\iota_{\rm{s}}$-invariant, this equals
\begin{align*}
    \zeta\left( \gena \bi{X,\dots,X}{0,0,\dots,0} -	\gena \bi{0,X,\dots,X}{0,\dots,0} \right) \,.
\end{align*}
The coefficients of $X^{a+1}$ in the latter expression is the sum over all admissible multiple zeta values in weight ${a+b+2}$.
\end{proof}

% \begin{lemma}
% Let $p(x) \= \sum_{i\geq 0} a_i\,x^i \in \R[x]$ be a polynomial such that
% \begin{align}\label{eq:ass} \sum_{i=0}^j a_i \geq 0 \qquad \text{for all }j.\end{align}
% Then, $p(x)\geq 0$ for $x\in [0,1]$. 
% \end{lemma}
% \begin{proof}
% We prove the result by induction on the degree of the polynomial. If the degree is zero there is nothing to prove. Now, let $p(x)$ as in the statement be given, and say $p$ is of degree $d$. Note that if $a_d$ is positive, then
% \[ p(x) \= (p(x)-a_d x^d) \+ a_d x^d \geq 0 \qquad (x\in [0,1]),\]
% since $p(x)-a_d\,x^d$ is a polynomial of degree $d$ satisfying the assumption~\eqref{eq:ass}. 
% Now, observe that also
% $p(x) \+ a_d x^{d-1} \meno a_d x^d$
% is of degree $d-1$ and satisfies the hypothesis~\eqref{eq:ass}. Hence,
% \[ p(x) \= (p(x) + a_d x^{d-1} - a_d x^d) \+ (-a_d x^{d-1} + a_d x^d) \,\geq \, 0 \qquad (x\in [0,1])\]
% if $a_d<0$.
% \end{proof}


\begin{thebibliography}{BBBL2}

\bibitem[AF]{PV}  P.~Alexandersson, V.~F\'{e}ray
\newblock {\itshape Shifted symmetric functions and multirectangular coordinates of {Y}oung diagrams},
\newblock J. Algebra {\bf 483} (2017), 262--305.

\bibitem[B]{B} H.~Bachmann:
\newblock {\itshape The algebra of bi-brackets and regularized multiple Eisenstein series},
\newblock J. Number Theory, {\bf 200} (2019), 260--294.

\bibitem[B2]{B2}  H.~Bachmann:
\newblock {\itshape Multiple zeta values and modular forms}, 
\newblock Lecture notes (Ver.~5.4), Nagoya University, 2020. Available online at \url{https://henrikbachmann.com/mzv2020.html}.

\bibitem[BI]{BI} H.~Bachmann, J.-W.~van~Ittersum:
\newblock {\itshape Formal multiple Eisenstein series and their derivations}. \newblock In preparation.

\bibitem[BK]{BK}  H.~Bachmann, U.~K\"uhn:
\newblock {\itshape The algebra of generating functions for multiple divisor sums and applications to multiple zeta values},
\newblock Ramanujan J. {\bf 40} (2016), 605--648. 

\bibitem[BK2]{BK2}  H.~Bachmann, U.~K\"uhn:
\newblock {\itshape A dimension conjecture for $q$-analogues of multiple zeta values}, 
\newblock Periods in Quantum Field Theory and Arithmetic, Springer Proc. Math. Stat. {\bf 314} (2020),  237--258.

\bibitem[BKM]{BKM}  H.~Bachmann, U.~K\"uhn, N.~Matthes:
\newblock {\itshape Realizations of the formal double Eisenstein space}, \newblock preprint, \href{https://arxiv.org/abs/2109.04267}{arXiv:2109.04267} (2021), 16 pp.

\bibitem[BO]{BO} S.~Bloch, A.~Okounkov:
\newblock {\itshape The character of the infinite wedge representation},
\newblock  Adv. Math. {\bf 149}:1 (2000), 1--60

\bibitem[BBBL]{BBBL1} J. M. Borwein, D. M. Bradley, D. J. Broadhurst and P. Lisonek: 
\newblock {\itshape Special Values of Multidimensional Polylogarithms}, \newblock Trans. Amer. Math. Soc. {\bf 353} (2001), 907--941.

\bibitem[BBBL2]{BBBL2} J. M. Borwein, D. M. Bradley, D. J. Broadhurst and P. Lisonek:
\newblock {\itshape  Combinatorial aspects of multiple zeta values},
\newblock Electronic Journal of Combinatorics {\bf 5} (1998), 12 pp.

\bibitem[Bra]{Bra} D. M.~Bradley:
\newblock {\it Multiple q-zeta values}, 
\newblock J. Algebra {\bf 283} (2005), 752--798.
		
\bibitem[Bri]{Bri1} B. Brindle: \newblock {\itshape{Dualities of $q$-analogues of multiple zeta values}},
\newblock Master thesis, Universit\"at Hamburg, \url{https://sites.google.com/view/benjamin-brindle/}, 2021.

\bibitem[Bri2]{Bri2} B. Brindle: \newblock {\itshape{A unified approach to qMZVs}}, 
\newblock preprint, 	\href{https://arxiv.org/abs/2111.00051}{arXiv:2111.00051} (2021), 37 pp.

\bibitem[CMZ]{CMZ}
D.~Chen, M.~M\"{o}ller, and D.~Zagier.
\newblock {\itshape Quasimodularity and large genus limits of {S}iegel-{V}eech constants},
\newblock {\em J. Amer. Math. Soc.}, {\bf 31}:4 (2018), 1059--1163.

\bibitem[D]{Dij} R.~Dijkgraaf.
\newblock {\itshape 
Mirror symmetry and elliptic curves},
\newblock In {\em The moduli space of curves (Texel Island, 1994)}, 149-–163,
Progr. Math., 129, Birkhäuser Boston, Boston, MA, 1995.

\bibitem[DGZ]{DGZ}
V.~Delecroix, E.~Goujard, and P.~Zograf.
\newblock {\itshape Contribution of one-cylinder square-tiled surfaces to {M}asur-{V}eech volumes}, with an appendix by P.~Engel.
\newblock {\em Ast\'{e}risque}, {\bf 415}: 223--274, 2020.
% Quelques aspects de la th\'{e}orie des syst\`emes              dynamiques: un hommage \`a Jean-Christophe Yoccoz. I,

% \bibitem[GKZ]{GKZ} H.~Gangl, M.~Kaneko, D.~Zagier:
% \newblock {\it Double zeta values and modular forms},
% \newblock in "Automorphic forms and zeta functions", World Sci. Publ., Hackensack, NJ (2006), 71--106.

\bibitem[H]{H} M.E.~Hoffman: 
\newblock {\itshape The algebra of multiple harmonic series},
\newblock J. Algebra {\bf 194} (1997), 477--495.

\bibitem[H2]{H2} M.E.~Hoffman:
\newblock {\itshape Quasi-shuffle products}, 
\newblock J. Algebraic Combin. {\bf 11} (2000), 49--68.

\bibitem[HI]{HI} M.E.~Hoffman, K.~Ihara: 
\newblock {\itshape Quasi-shuffle products revisited}. 
\newblock J. Algebra {\bf 481} (2017), 293--326.

\bibitem[IKZ]{IKZ} K.~Ihara, M.~Kaneko and  D.~Zagier:
\newblock {\itshape Derivation and double shuffle relations for multiple zeta values}, 
\newblock Compositio Math. {\bf 142} (2006), 307--338.

\bibitem[I]{vI} J.W.M.~v.~Ittersum: 
\newblock {\itshape A symmetric Bloch--Okounkov Theorem}, \newblock Res. Math. Sci. {\bf 8}, 19 (2021).

\bibitem[I2]{I2} J.W.M.~v.~Ittersum:
\newblock {\itshape The Bloch--Okounkov theorem for congruence subgroups and Taylor coefficients of quasi‑Jacobi forms}, 
\newblock preprint, \href{https://arxiv.org/abs/2102.12964}{arXiv:2102.12964} (2021), 49 pp.

\bibitem[KZ]{KZ95} M.~Kaneko and D.~Zagier:
\newblock {\itshape A generalized Jacobi theta function and quasimodular forms},
\newblock In {\em The moduli space of curves (Texel Island, 1994)}, 149--163,
Progr. Math., 129, Birkhäuser Boston, Boston, MA, 1995.

\bibitem[M]{M} K.~Matsumoto: 
\newblock {\itshape On analytic continuation of various multiple zeta-functions}, 
\newblock Number Theory for the Millenium, II (Urbana, 2000), A. K. Peters, Natick, MA (2002), 417--440.

\bibitem[OT]{OT}  J.~Okuda and Y.~Takeyama:
\newblock {\it On relations for the multiple q-zeta values}, \newblock  Ramanujan J. {\bf 14} (2007), 379--387

\bibitem[OZ]{OZ} Y.~Ohno and D.~Zagier:
\newblock {\itshape Multiple zeta values of fixed weight, depth, and height},
\newblock Indag. Math. {\bf 12}:4 (2001), 483--487.

% \bibitem[O]{O}  A.~Okounkov:
% \newblock {\it Hilbert schemes and multiple $q$-zeta values}, \newblock Funct. Anal. Appl. 48 (2014), 138--144.

\bibitem[P]{P} Y.~Pupyrev: 
\newblock {\itshape Linear and algebraic independence of q-zeta values}, \newblock Math. Notes {\bf 78}(4), 563--568 (2005). Translated from Matematicheskie Zametki, {\bf 78}(4), 2005, pp. 608--613.

\bibitem[Sch]{S} K.~Schlesinger:
\newblock {\itshape Some remarks on q-deformed multiple polylogarithms}, \newblock preprint, \href{https://arxiv.org/abs/math/0111022}{arXiv:math/0111022} (2001), 11 pp.

\bibitem[Sin]{Sin} J. Singer: 
\newblock {\itshape $q$-Analogues of Multiple Zeta Values and their application in renormalization},  \newblock Dissertation, Erlangen-N\"urnberg University (2017).

\bibitem[Za]{Z} D.~Zagier:
\newblock {\itshape Partitions, quasimodular forms, and the Bloch--Okounkov theorem}, \newblock Ramanujan J. {\bf 41} (2016), no. 1-3, 345--368.	

\bibitem[Za2]{Zag} D.~Zagier: 
\newblock {\itshape The Mellin transform and other useful analytic techniques},
\newblock Appendix to E.~Zeidler, Quantum Field Theory I: Basics in Mathematics and Physics. A Bridge Between Mathematicians and Physicists, Springer-Verlag, Berlin-Heidelberg-New York (2006), 305--323 
		
\bibitem[Zh]{Zh} J.~Zhao:
\newblock {\itshape Uniform approach to double shuffle and duality relations of various $q$-analogs of multiple zeta values via Rota-Baxter algebras},  
\newblock Periods in Quantum Field Theory and Arithmetic, Springer Proceedings in Mathematics \& Statistics {\bf 314} (2020), 259--292.
		
\bibitem[Zu]{Zu} V.V.~Zudilin:
\newblock {\itshape
Algebraic relations for multiple zeta values},
\newblock Uspekhi Mat. Nauk {\bf 58} (2013), 1(349), 3--32.

\bibitem[Zu2]{Zu2} V.V.~Zudilin:
\newblock {\itshape
Multiple $q$-Zeta Brackets},
\newblock Mathematics {\bf 3} (2015), 119-130.

\end{thebibliography}
\end{document}